\documentclass{amsart}
\usepackage{amsfonts}
\usepackage{amsthm}
\usepackage{amssymb}
\usepackage{amsmath}
\usepackage{upref}
\usepackage{mathrsfs}
\usepackage{color}
\usepackage[citecolor=blue,colorlinks=true]{hyperref}
\usepackage{enumitem}
\usepackage{graphicx}
\usepackage[left=3cm,right=3cm,top=4cm,bottom=3cm]{geometry}

\newtheorem{thm}{Theorem}[section]
\newtheorem{lemma}[thm]{Lemma}
\newtheorem{prop}[thm]{Proposition}
\newtheorem{cor}[thm]{Corollary}
\theoremstyle{definition}
\newtheorem{defin}[thm]{Definition}
\theoremstyle{remark}
\newtheorem{rem}[thm]{Remark}

\numberwithin{equation}{section}
\allowdisplaybreaks

\newcommand{\cev}[1]{\reflectbox{\ensuremath{\vec{\reflectbox{\ensuremath{#1}}}}}}

\newcommand{\dif}{\mathrm{d}}
\newcommand{\dd}{\mathrm{d}}

\newcommand{\totdif}{\mathrm{D}}
\DeclareMathOperator{\sgn}{sgn}
\newcommand{\mf}{\mathscr{F}}
\newcommand{\mr}{\mathbb{R}}
\newcommand{\R}{\mathbb{R}}
\newcommand{\prst}{\mathbb{P}}
\newcommand{\p}{\mathbb{P}}
\newcommand{\stred}{\mathbb{E}}
\newcommand{\E}{\mathbb{E}}
\newcommand{\ind}{\mathbf{1}}
\newcommand{\mn}{\mathbb{N}}

\DeclareMathOperator{\supp}{supp}

\DeclareMathOperator{\diver}{div}
\DeclareMathOperator*{\esssup}{ess\,sup}

\newcommand{\me}{\mathrm{e}}

\title[Scalar conservation laws with rough flux and stochastic forcing]{Scalar conservation laws with rough flux and stochastic forcing}

\author{Martina Hofmanov\'a}
\address[M. Hofmanov\'a]{Technical University Berlin, Institute of Mathematics, Stra\ss e des 17. Juni 136, 10623 Berlin, Germany}
\email{hofmanov@math.tu-berlin.de}

\begin{document}
\begin{abstract}
In this paper, we study scalar conservation laws where the flux is driven by a geometric H\"older $p$-rough path for some $p\in (2,3)$ and the forcing is given by an It\^o stochastic integral driven by a Brownian motion.
In particular, we derive the corresponding kinetic formulation and define an appropriate notion of kinetic solution. In this context, we are able to establish well-posedness, i.e. existence, uniqueness and the $L^1$-contraction property that leads to continuous dependence on initial condition. Our approach combines tools from rough path analysis, stochastic analysis and theory of kinetic solutions for conservation laws. As an application, this allows to cover the case of flux driven for instance by another (independent) Brownian motion enhanced with L\'evy's stochastic area.
\end{abstract}

\subjclass[2010]{60H15, 35R60, 35L65}

\keywords{scalar conservation laws, rough paths, kinetic formulation, kinetic solution, BGK approximation, method of characteristics}

\date{\today}

\maketitle


\section{Introduction}

The goal of the present paper is to develop a well-posedness theory for the following scalar rough conservation law
\begin{equation}\label{eq}
\begin{split}
\dif u+\diver\big(A(x,u)\big)\,\dif z&=g(x,u)\,\dif W,\qquad t\in(0,T),\,x\in\mr^N,\\
u(0)&=u_0,
\end{split}
\end{equation}
where $z=(z^1,\dots,z^M)$ is a deterministic rough driving signal and $W=(W^1,\dots,W^K)$ is a Wiener process and the stochastic integral is understood in the It\^ o sense. The coefficients $A:\mr^N\times\mr\rightarrow\mr^{N\times M}$, $g:\mr^N\times\mr\rightarrow\mr^K$ satisfy a sufficient regularity assumption introduced in Section \ref{sec:hypotheses}. The above equation can be rewritten by using the Einstein summation convention as follows
\begin{equation*}
\begin{split}
\dif u+\partial_{x_i}\big(A_{ij}(x,u)\big)\,\dif z^j&=g_k(x,u)\,\dif W^k,\qquad t\in(0,T),\,x\in\mr^N,\\
u(0)&=u_0.
\end{split}
\end{equation*}
As an application of our analysis, one can replace $z$, for instance, by another Brownian motion $B$, which is independent of $W$, and give meaning to
\begin{equation}\label{eq:stoch1}
\begin{split}
\dif u+\diver\big(A(x,u)\big)\circ\dif B&=g(x,u)\,\dif W,\qquad t\in(0,T),\,x\in\mr^N,\\
u(0)&=u_0.
\end{split}
\end{equation}


Conservation laws and related equations have been paid an increasing attention lately and have become a very active area of research, counting nowadays quite a number of results for deterministic and stochastic setting, that is for conservation laws either of the form
\begin{equation}\label{eq:det}
\partial_t u+\diver\big(A(u)\big)=0
\end{equation}
(see \cite{vov1}, \cite{car}, \cite{vov}, \cite{kruzk}, \cite{lpt1}, \cite{lions}, \cite{perth}, \cite{tadmor}) or
\begin{equation}\label{eq:stoch}
\dif u+\diver\big(A(u)\big)\dif t=g(x,u)\dif W
\end{equation}
where the It\^o stochastic forcing is driven by a finite- or infinite-dimensional Wiener process (see \cite{bauzet}, \cite{karlsen}, \cite{debus2}, \cite{debus}, \cite{DV}, \cite{feng}, \cite{bgk},  \cite{holden}, \cite{kim},  \cite{stoica}, \cite{wittbold}). Degenerate parabolic PDEs were studied in \cite{car}, \cite{chen} and in the stochastic setting in \cite{BVW}, \cite{degen2}, \cite{hof}.

Since the theory of rough paths viewed as a tool that allows for deterministic treatment of stochastic differential equations has been of growing interest recently, several attempts have already been made to extend this theory to conservation laws as well. First, Lions, Perthame and Souganidis (see \cite{lps}, \cite{lps1}) developed a pathwise approach for
$$\dif u+\diver\big(A(x,u)\big)\circ \dif W=0$$
where $W$ is a continuous real-valued signal and $\circ$ stands for the Stratonovich product in the Brownian case, then Friz and Gess (see \cite{friz2}) studied
$$\dif u+\diver f(t,x,u)\dif t=F(t,x,u)\dif t+\Lambda_k(x,u,\nabla u)\dif z^k$$
where $\Lambda_k$ is affine linear in $u$ and $\nabla u$ and $z=(z^1,\dots,z^K)$ is a rough driving signal and Gess and Souganidis \cite{gess} considered
$$\dif u+\diver\big(A(x,u)\big)\dif z=0$$
where $z=(z^1,\dots,z^M)$ is a geometric $\alpha$-H\" older path and in \cite{GS} they studied the long-time behavior for $z$ being a Brownian motion.



In order to find a suitable concept of solution for problems of the form \eqref{eq}, it was observed already some time ago that, on the one hand, classical $C^1$ solutions do not exist in general and, on the other hand, weak or distributional solutions lack uniqueness.
The first claim is a consequence of the fact that any smooth solution has to be constant along characteristic lines, which can intersect in finite time (even in the case of smooth data) and shocks can be produced. The second claim demonstrates the inconvenience that often appears in the study of PDEs and SPDEs: the usual way of weakening the equation leads to the occurrence of nonphysical solutions and therefore additional assumptions need to be imposed in order to select the physically relevant ones and to ensure uniqueness. Hence one needs to find some balance that allows to establish existence of a unique (physically reasonable) solution.

Towards this end, we pursue the ideas of kinetic approach, a concept of solution that was first introduced by Lions, Perthame, Tadmor \cite{lions} for deterministic hyperbolic conservation laws and further studied in \cite{vov1}, \cite{chen}, \cite{vov}, \cite{lpt1}, \cite{lions}, \cite{tadmor}, \cite{perth}. This direction also appears in several works on stochastic conservation laws and degenerate parabolic SPDEs, see \cite{degen2}, \cite{debus2}, \cite{debus}, \cite{DV}, \cite{bgk}, \cite{hof} and in the (rough) pathwise works \cite{GS}, \cite{gess}, \cite{lps}, \cite{lps1}. In comparison to the notion of entropy solution introduced by Kru\v{z}kov \cite{kruzk} and further developed e.g. in \cite{bauzet}, \cite{car}, \cite{feng}, \cite{kim}, \cite{wittbold}, kinetic solutions are more general in the sense that they are well defined even in situations when neither the original conservation law nor the corresponding entropy inequalities can be understood in the sense of distributions which is part of the definition of entropy solution. Usually this happens due to lack of integrability of the flux and entropy-flux terms, e.g. $A(u)\notin L^1_{\text{loc}}$. Therefore, further assumptions on initial data or the flux function $A$ are in place in order to overcome this issue and remain in the entropy setting. It will be seen later on that no such restrictions are necessary in the kinetic approach as the equation that is actually to be solved -- the so-called kinetic formulation, see \eqref{eq:kinform} -- is in fact linear. In addition, various proofs simplify as methods for linear PDEs are available.

Let us now shortly present the main ideas of our approach. Apart from the above mentioned difficulties there is another one that originates in the low regularity of driving signals and solution. Namely, the corresponding rough integrals are not well defined so we present a formulation, see \eqref{eq:weakkinformul}, that does not include any rough path driven terms and therefore provides a suitable notion of kinetic solution in this context. To this end, we adapt the ideas of \cite{lps}, \cite{lps1}, \cite{gess} where the authors introduced a method of modified test functions to eliminate the rough terms. Our method then combines this approach with the ideas for stochastic conservation laws treated in \cite{debus}, \cite{bgk}.
As usual for this class of problems, we define a second notion of solution -- a generalized kinetic solution -- which, roughly speaking, takes values in the set of Young measures. The general idea is that in order to get existence of such a solution, weak convergence (in some $L^p$ space) of the corresponding approximations is sufficient which allows for an easier proof. The key result is then the Reduction Theorem that states that any generalized kinetic solution is actually a kinetic one, that is, the Young measure at hand is a parametrized Dirac mass.

Concerning the existence part, we make use of the so-called Bhatnagar-Gross-Krook approximation (BGK for short) which allows to describe the conservation law as the hydrodynamic limit of the corresponding BGK model, as the microscopic scale $\varepsilon$ vanishes. This is nowadays a standard tool in the deterministic setting where the literature is quite extensive (see \cite{vov1}, \cite{vov}, \cite{lpt1}, \cite{lions}, \cite{tadmor}, \cite{perth}). Even though the general concept is analogous in the stochastic case, the techniques are significantly different, the result was established recently by the author \cite{bgk}. To be more precise, the key point is to solve the corresponding characteristic system which for the general case of \eqref{eq} reads as follows
\begin{equation}\label{eq:charr}
\begin{split}
\dif\varphi^0_t&= -\partial_{x_i}A_{ij}(\varphi_t)\,\dif z^j+g_k(\varphi_t)\dif W^k,\\
\dif\varphi^i_t&=\partial_u A_{ij}(\varphi_t)\,\dif z^j,\qquad i=1,\dots,N.
\end{split}
\end{equation}
Already at this level one can understand the difficulties coming from the complex structure of \eqref{eq}.
Namely, in the case of \eqref{eq:det} the characteristic system reduces to a set of independent equations (note that the flux function $A$ is independent of the space variable)
\begin{equation*}
\begin{split}
\dif\varphi^0_t&=0\\
\dif\varphi^i_t&=\partial_u A_{ij}(\varphi^0_t)\,\dif t,\qquad i=1,\dots,N,
\end{split}
\end{equation*}
which can be solved immediately and as a consequence the method simplifies. For the stochastic case \eqref{eq:stoch} we obtain
\begin{equation*}
\begin{split}
\dif\varphi^0_t&= g_k(\varphi_t)\,\dif W^k\\
\dif\varphi^i_t&=\sum_{j=1}^M\partial_u A_{ij}(\varphi^0_t)\,\dif t,\qquad i=1,\dots,N,
\end{split}
\end{equation*}
hence the first coordinate of the characteristic curve is governed by an SDE and further difficulties arise due to randomness (see \cite{bgk}).
If $A$ is also $x$-dependent we observe an additional term in the first equation of the characteristic system \eqref{eq:charr}, however, let us point out that there is a major difference between this and the term coming from the forcing. In particular, if $g=0$ then the flow of diffeomorphisms generated by \eqref{eq:charr} is volume preserving as can be seen easily by calculating divergence of the corresponding vector field. This does not hold true anymore if forcing in nonconservative form is present in the equation, unless $g$ is independent of the solution, i.e. we consider additive noise in \eqref{eq}.

The method of BGK approximation has another advantage especially when dealing with rough path driven conservation laws. As already mentioned above, the problem boils down to solving the characteristic system \eqref{eq:charr}, i.e. an ordinary (stochastic, rough) differential equation and the theory for those problems is well-established, unlike the theory for rough path driven equations in infinite dimension, i.e. rough PDEs. Furthermore, the BGK model also provides an explicit formula for the approximate solutions and therefore the necessary estimates independent of the microscopic scale $\varepsilon$ come rather naturally.

The exposition is organized as follows. In Section \ref{sec:hypotheses}, we introduce the basic setting, define the notion of kinetic solution and state our main result, Theorem \ref{thm:main}. In order to make the paper more self-contained, Section \ref{sec:rough} provides a brief overview of the relevant concepts from rough path theory. Section \ref{sec:uniqueness} is devoted to the proof of uniqueness, reduction of a generalized kinetic solution to a kinetic one and the $L^1$-contraction property, Theorem \ref{thm:reduction} and Corollary \ref{cor:contraction}. The remainder of the paper deals with the existence part of Theorem \ref{thm:main} which is divided into several parts and finally established through Theorem \ref{thm:bgkconvergence}. Subsections \ref{sec:roughtr} and \ref{sec:bgksol} give a rough-pathwise existence of unique solutions to the BGK approximation. In these two sections we work with a fixed realization of the driving signals or, to be more precise, a fixed realization of a joint lift of $(t,z,W)$. In Subsection \ref{sec:conv}, the stochastic approach is resumed and we pass to the limit and obtain a kinetic solution to \eqref{eq}.

\section{Definitions and the main result}
\label{sec:hypotheses}

Let us now introduce the precise setting of \eqref{eq}. We work on a finite time interval $[0,T],\,T>0$ and on the whole space concerning the space variable $x\in\mr^N$.
Throughout the paper, we assume that $z$ can be lifted to a geometric H\"older $p$-rough path for some $p\in (2,3)$ and denote its lift by $\mathbf{z}=(\mathbf{z}^1,\mathbf{z}^2)$. Next we consider the following joint lift of $(z,W)$: we define $\mathbf{\Lambda}=(\mathbf{\Lambda}^1,\mathbf{\Lambda}^2)$ by
\begin{align}\label{jointlift}
\begin{aligned}
\mathbf{\Lambda}_t^{1,i}&=\begin{cases}
							z^i_t, & \text{ if }\,i\in\{1,\dots,M\},\\
							W_t^{i-M}, &\text{ if }\,i\in\{M+1,\dots,M+K\},
							\end{cases}\\
\mathbf{\Lambda}_t^{2,i,j}&=\begin{cases}
							\mathbf{z}_t^{2,i,j}, & \text{ if }\,i,j\in\{1,\dots,M\},\\
							\int_0^t W^{i-M}_r\circ\dif W^{j-M}_r &\text{ if }\,i,j\in\{M+1,\dots,M+K\},\\
							\int_0^t z^i_r\dif W^{j-M}_r & \text{ if }\,i\in\{1,\dots,M\},\,j\in\{M+1,\dots,M+K\},\\
							z^j_tW^{i-M}_t-\int_0^tz^j_r\dif W^{i-M}_r& \text{ if }\,i\in\{M+1,\dots,M+K\},\,j\in\{1,\dots,M\}.
					\end{cases}
\end{aligned}
\end{align}
It was shown in \cite{DOR15} that such a stochastic process exists and could be considered as the canonical joint lift of $z$ and $W$. Note that although in the original equation \eqref{eq} we consider It\^o stochastic integral, the lift of $W$ used in the construction of $\mathbf{\Lambda}$ above corresponds to the Stratonovich version.

Regarding the coefficients in \eqref{eq}, let us fix the following notation. Let
\begin{align*}
a&=(a_{ij})=(\partial_u A_{ij}):\mr^N\times\mr\longrightarrow \mr^{N \times M},\\
b&=(b_{j})=(\diver_x A_{\cdot j}):\mr^N\times\mr\longrightarrow \mr^M,
\end{align*}
and assume that $a,\,b\in \mathrm{Lip}^{\gamma+2}$ and $g\in \mathrm{Lip}^{\gamma+3}$ for some $\gamma>p.$ Here we adopt the notation of \cite[Definition 10.2]{friz}, namely, a mapping $V:\mr^e\rightarrow\mr^d$ belongs to $\mathrm{Lip}^\beta$ provided it is bounded, $\lfloor \beta \rfloor$-times continuously differentiable with bounded derivatives of all orders and its $\lfloor \beta\rfloor^{\text{th}}$ derivative is $\{\beta\}$-H\"older continuous.

Furthermore,we suppose that
\begin{equation}\label{eq:null}
b(x,0)=0\quad g(x,0)=0\;\qquad\forall x\in\mr^N,
\end{equation}
and denote
$$G^2(x,\xi)=\sum_{k=1}^K|g_k(x,\xi)|^2,\qquad \forall x\in\mr^N,\,\xi\in\mr.$$

\subsection{Notations}

We adopt the following notations. The brackets $\langle\cdot,\cdot \rangle$ are used to denote the duality between the space of distributions over $\mr^N\times\mr$ and $C_c^1(\mr^N\times\mr)$. We denote similarly the integral
$$\langle f,h\rangle=\int_{\mr^N}\int_\mr f(x,\xi) h(x,\xi)\,\dif x\,\dif \xi,\qquad f\in L^p(\mr^N\times\mr),\;h\in L^q(\mr^N\times\mr),$$
where $p,q\in[1,\infty]$ are conjugate exponents. By $\mathcal{M}_b([0,T)\times\R^N\times\R)$ we denote the space of Borel measures over $[0,T)\times\R^N\times\R$ and $\mathcal{M}^+_b([0,T)\times\R^N\times\R)$ are then nonnegative Borel measures. We also use the shorthand
$$n(\phi)=\int_{[0,T)\times\mr^N\times\mr}\phi(t,x,\xi)\,\dd n(t,x,\xi),\qquad n\in\mathcal{M}_b([0,T)\times\R^N\times\R),\,\phi \in C_c([0,T)\times\R^N\times\R).$$
and $C_0([0,T]\times\mr^N\times\mr)$ denotes the space of continuous functions $[0,T]\times\mr^N\times\mr$ that vanish at infinity, i.e. for large $(x,\xi)$.
The differential operators gradient $\nabla$ and divergence $\diver$ are (unless otherwise stated) understood with respect to the space variable $x$.


\subsection{Definitions}
\label{subsec:def}

As the next step, let us introduce the kinetic formulation of \eqref{eq} as well as the basic definitions concerning the notion of kinetic solution. The motivation behind this approach is given by the nonexistence of a strong solution and, on the other hand, the nonuniqueness of weak solutions, even in simple cases. The idea is to establish an additional criterion -- the kinetic formulation --
which is automatically satisfied by any strong solution to \eqref{eq} (in case it exists) and which permits to ensure the well-posedness.

We start with the definition of kinetic measure.

\begin{defin}[Kinetic measure]\label{def:kinmeasure}
A mapping $m$ from $\Omega$ to $\mathcal{M}_b^+([0,T]\times\mr^N\times\mr)$, the set of nonnegative bounded measures over $[0,T]\times\mr^N\times\mr$, is said to be a kinetic measure provided
\begin{enumerate}
 \item $m$ is measurable in the following sense: for each $\phi\in C_0([0,T]\times\mr^N\times\mr)$, the mapping $m(\phi):\Omega\rightarrow\mr$ is measurable,
 \item the mapping $$\int_{[0,T)\times\R^N\times\R}\dd m(t,x,\xi):\Omega\to\R$$ is measurable and
$$\E\int_{[0,T)\times\mr^{N}\times\R}\dd m(t,x,\xi)<\infty,$$
 \item for any $\phi\in C_0(\mr^N\times\mr)$, $t\mapsto m(\ind_{[0,t]}\phi)$ is progressively measurable.
\end{enumerate}
\end{defin}

Formally speaking, the kinetic formulation corresponding to the conservation law at hand is given as follows
\begin{equation}\label{eq:kinform}
\begin{split}
\dif F+\nabla F\cdot a\,\dif z-\partial_\xi F\, b\,\dif z&=-\partial_\xi F\, g\,\dif W+\frac{1}{2}\partial_\xi(G^2\partial_\xi F)\,\dif t+\partial_\xi m,\\
F(0)&=F_0,
\end{split}
\end{equation}
where $F=\ind_{u>\xi}$ and $m$ is a kinetic measure\footnote{Here $u$ is a function of $(\omega,t,x)$ so $F(\omega,t,x,\xi)=\ind_{u(\omega,t,x)>\xi}$ is well-defined and regarded as a function of four variables $(\omega,t,x,\xi)$.}. However, since the expected regularity of solutions is low and consequently the rough path driven integrals are not well defined, it is necessary to define a suitable notion of weak solution to this problem.
This leads us to the notion of kinetic solution to rough path driven conservation laws that we introduce in this work.
Note that it is a consistent extension of the corresponding notion of kinetic solution for the case of a smooth driving signal $z$, for further discussion on this subject we refer the reader to Subsection \ref{subsec:smoothdrivers}.

\begin{defin}[Kinetic solution]\label{kinsol}
Let $u_0\in L^1\cap L^2(\Omega\times\mr^N).$
Then a progressively measurable
$$u\in  L^2(\Omega;L^2(0,T;L^2(\mr^N)))$$
satisfying
\begin{equation}\label{fd}
\E\esssup_{0\leq t\leq T}\|u(t)\|_{L^1_x}\leq C
\end{equation}
is said to be a kinetic solution to \eqref{eq} with initial datum $u_0$ provided there exists a kinetic measure $m$ such that the pair $(F=\ind_{u>\xi},m)$ satisfies, for all $\phi\in C^1_c(\mr^N\times\mr)$ and $\alpha\in C^1_c([0,T))$, $\p$-a.s.,
\begin{align}\label{eq:weakkinformul}
\begin{split}
\int_0^T&\big\langle F(t),\phi(\theta_t)\big\rangle\partial_t\alpha(t)\,\dif t+\big\langle F_0,\phi\big\rangle\alpha(0)\\
&=\int_0^T\langle\partial_\xi F(t) g,\phi(\theta_t)\rangle\alpha(t)\,\dif W-\frac{1}{2}\int_0^T\langle\partial_\xi(G^2\partial_\xi F(t)),\phi(\theta_t)\rangle\alpha(t)\,\dif t+ m\big(\alpha\partial_\xi\phi(\theta_t)\big),
\end{split} 
\end{align}
where $\theta=(\theta^0,\theta^x)$ is the inverse flow corresponding to
\begin{align}\label{eq:fl1}
\begin{aligned}
\dif\pi^0_t&=-b(\pi_t)\,\dif \mathbf{z}\\
\dif\pi^x_t&=a(\pi_t)\,\dif \mathbf{z}.
\end{aligned}
\end{align}
\end{defin}


To be more precise, it follows from \cite[Proposition 11.11]{friz} that under our assumptions \eqref{eq:fl1} possesses a unique solution $\pi$ that defines a flow of $C^2$-diffeomorphisms. We denote by $\pi_{s,t}(x,\xi)$ the solution of \eqref{eq:fl1} starting from $(x,\xi)$ at time $s$. To simplify the notation, we write $\pi_t$ instead of $\pi_{0,t}$ and we denote the corresponding inverse flow by $\theta$. Then $\theta_{t,s}=\pi_{t,s}^{-1} $ is the unique solution to the time-reversed problem
\begin{equation}\label{eq:fl2}
\begin{split}
\dif\theta^0_{t,s}&=-b(\theta_{t,s})\,\dif \cev{\mathbf{z}}^s,\\
\dif\theta^x_{t,s}&=a(\theta_{t,s})\,\dif \cev{\mathbf{z}}^s,
\end{split}
\end{equation}
where $\cev{\mathbf{z}}^s(\cdot)=\mathbf{z}(s-\cdot)$ is the time-reversed path to $\mathbf{z}$. We point out that the flow $\pi$ as well as $\theta$ is volume preserving as can be seen easily by calculating divergence of the corresponding vector field and recalling the fact that divergence free vector fields generate volume preserving flows. Thus the Jacobian of $\pi_{s,t}$ satisfies: $\mathrm{J}\pi_{s,t}\equiv1$ and similarly for $\theta$.

Note that with a classical argument of separability, the set of full probability where \eqref{eq:weakkinformul} holds true does not depend on the particular choice of test functions $\phi,\,\alpha$.


We proceed with a reminder of Young measures and the related definition of kinetic function that will eventually lead to the notion of generalized kinetic solution, see Definition \ref{genkinsol}. The concept of Young measures was developed in \cite{young} as a
technical tool for describing composite limits of nonlinear functions with weakly convergent sequences, for further reading we refer the reader e.g. to \cite{malek}.

In what follows, we denote by $\mathcal{P}_1(\mr)$ the set of probability measures on $\mr$.

\begin{defin}[Young measure]
Let $(X,\lambda)$ be a $\sigma$-finite measure space. A mapping $\nu:X\rightarrow\mathcal{P}_1(\mr)$ is called a Young measure provided it is weakly measurable, that is, for all $\phi\in C_b(\mr)$ the mapping $z\mapsto \nu_z(\phi)$ from $X$ to $\mr$ is measurable.
%
A Young measure $\nu$ is said to vanish at infinity if
\begin{equation*}
\int_X\int_\mr|\xi|\,\dif\nu_z(\xi)\,\dif\lambda(z)<\infty.
\end{equation*}
\end{defin}

\begin{defin}[Kinetic function]
Let $(X,\lambda)$ be a $\sigma$-finite measure space.
A measurable function $f:X\times\mr\rightarrow[0,1]$ is called a kinetic function on $X$ if there exists a Young measure $\nu$ on $X$ that vanishes at infinity such that for a.e. $z\in X$ and for all $\xi\in\mr$
$$f(z,\xi)=\nu_z(\xi,\infty).$$
\end{defin}

\begin{lemma}[Compactness of Young measures]\label{kinetcomp}
Let $(X,\lambda)$ be a $\sigma$-finite measure space such that $L^1(X)$ is separable. Let $(\nu^n)$ be a sequence of Young measures on $X$ such that for some $p\in[1,\infty)$
\begin{equation}\label{eq:estyoungm}
\sup_{n\in\mn}\int_X\int_\mr|\xi|^p\,\dif\nu^n_z(\xi)\,\dif \lambda(z)<\infty.
\end{equation}
Then there exists a Young measure $\nu$ on $X$ and a subsequence still denoted by $(\nu^n)$ such that for all $h\in L^1(X)$ and all $\phi\in C_b(\mr)$
$$\lim_{n\rightarrow \infty}\int_X h(z)\int_\mr \phi(\xi)\,\dif\nu^n_z(\xi)\,\dif\lambda(z)=\int_X h(z)\int_\mr \phi(\xi)\,\dif\nu_z(\xi)\,\dif\lambda(z)$$
Moreover, if $f_n,\,n\in\mn,$ are the kinetic functions corresponding to $\nu^n,\,n\in\mn,$ such that \eqref{eq:estyoungm} holds true, then there exists a kinetic function $f$ (which correponds to the Young measure $\nu$ whose existence was ensured by the first part of the statement) and a subsequence still denoted by $(f^n)$ such that
$$f_n\overset{w^*}{\longrightarrow} f\quad \text{ in }\quad L^\infty(X\times\mr).$$

\begin{proof}
Various results of this form are somewhat classical in the literature, a proof for the case of $(X,\lambda)$ being a finite measure space can be found in \cite[Theorem 5, Corollary 6]{debus}, however, one can actually observe that this additional assumption is not used in the proof and therefore the same proof applies to our setting of $(X,\lambda)$ being a $\sigma$-finite measure space.
\end{proof}

\end{lemma}

\begin{rem}\label{chi}
If $f : X \times \R\to [0, 1]$ is a kinetic function corresponding to the Young measure $\nu$ satisfying, for some $p\in(1,\infty)$,
$$\int_X\int_\mr|\xi|^p\,\dif\nu_z(\xi)\,\dif \lambda(z)<\infty,$$
then we denote by $\chi_f$ the function defined by $\chi_f(z,\xi) = f (z,\xi)- \ind_{0>\xi}$.
Contrary to $f$, this modification is integrable on $E\times\R$ whenever $E\subset X$ with $\lambda(E)<\infty$. Indeed,
$$\chi_f(z,\xi)=\begin{cases}
				-\int_{(-\infty,\xi]}\dd\nu_z,&\xi<0,\\
				\int_{(\xi,\infty)}\dd\nu_z,&\xi\geq0,
				\end{cases}
				$$
hence
\begin{align*}
|\xi|^p\int_X|\chi_f(z,\xi)|\dd\lambda(z)\leq \int_X\int_\R|\zeta|^p\dd\nu_z(\zeta)\dd\lambda(z)<\infty
\end{align*}
which implies
\begin{align*}
\int_E\int_\R|\chi_f(z,\xi)|\dd\xi\dd\lambda(z)\leq \int_\R\frac{1}{1+|\xi|^p}\dd\xi\bigg(\lambda(E)+\int_X\int_\R|\zeta|^p\dd\nu_z(\zeta)\dd\lambda(z)\bigg)<\infty.
\end{align*}
Besides, if $f$ is at equilibrium, that is, there exists $u\in L^1(X)$ such that $f=\ind_{u>\xi}$ and $\nu=\delta_{u=\xi}$ then we will rather write $\chi_u$ instead of $\chi_f$ and it holds true that
$$\int_\R|\chi_{u(z)}(\xi)|\dd\xi= |u(z)|.$$
\end{rem}

\begin{defin}[Generalized kinetic solution]\label{genkinsol}
Let $F_0:\Omega\times\mr^N\times\mr\rightarrow[0,1]$ be a kinetic function. A progressively measurable function $F:\Omega\times[0,T]\times\mr^N\times\mr\rightarrow[0,1]$ is called a generalized solution to \eqref{eq} with initial datum $F_0$ if $F(t)$ is a kinetic function for a.e. $t\in[0,T]$, there exists $C>0$ such that
\begin{equation}\label{integrov}
\E\esssup_{0\leq t\leq T}\int_{\mr^N}\int_{\mr}|\xi|\,\dif\nu_{t,x}(\xi)\,\dif x+\E\int_0^T\int_{\mr^N}\int_{\mr}|\xi|^2\,\dif\nu_{t,x}(\xi)\,\dif x\leq C, 
\end{equation}
where $\nu=-\partial_\xi F$, and if there exists a kinetic measure $m$ such that \eqref{eq:weakkinformul} holds true.
\end{defin}

\subsection{Kinetic solutions in the case of smooth driver $z$}
\label{subsec:smoothdrivers}

For the sake of completeness, let us now present a formal derivation of \eqref{eq:kinform} in case of sufficiently smooth solution to the conservation law \eqref{eq} driven by a smooth path $z$ and a Wiener process $W$. We denote such a solution by $u$ and note that it satisfies \eqref{eq} pointwise in $x\in\R^N$. Let us fix $x\in\mr^N$ and derive the equation satisfied by $\ind_{u(\cdot,x)>\xi}$ which is understood as a distribution in $\xi$. Towards this end, we denote by $\langle \cdot,\cdot\rangle_\xi$ the duality between the space of distributions over $\R$ and $C^\infty_c(\R)$ and observe that if $\phi(\xi)=\int_{-\infty}^\xi\phi_1(\zeta)\,\dif\zeta$ for some $\phi_1\in C_c^\infty(\mr)$, then
$$\langle \ind_{u>\xi},\phi_1\rangle_\xi=\phi(u)-\phi(-\infty)=\phi(u).$$
Therefore using the It\^o formula, we obtain
\begin{align*}
\dif&\,\langle \ind_{u>\xi},\phi_1\rangle_\xi=-\phi_1(u)\diver\big(A(x,u)\big)\dif z+\phi_1(u) g(x,u)\dif W+\frac{1}{2}\phi'_1(u)G^2(x,u)\dif t\\
&=-\phi_1(u)a(x,u)\cdot\nabla u\,\dif z-\phi_1(u)b(x,u)\dif z+\phi_1(u) g(x,u)\dif W+\frac{1}{2}\phi'_1(u)G^2(x,u)\dif t
\end{align*}
and since
$$\nabla\ind_{u>\xi}=\delta_{u=\xi}\nabla u\qquad\text{in}\qquad \mathcal{D}'(\R^N\times\R)\quad\text{a.s.}$$
we test the above against $\phi_2\in C_c^\infty(\mr^N)$ and deduce
\begin{align*}
\dif\langle \ind_{u>\xi},\phi_1\phi_2\rangle&=-\langle a(x,\xi)\nabla\ind_{u>\xi},\phi_1\phi_2\rangle\dif z-\langle b(x,\xi)\delta_{u=\xi},\phi_1\phi_2\rangle\dif z+\langle g(x,\xi)\delta_{u=\xi},\phi_1\phi_2\rangle\dif W\\
&\quad+\frac{1}{2}\langle G^2(x,\xi)\delta_{u=\xi},\phi'_1\phi_2\rangle\dif t\\
&=-\langle a(x,\xi)\nabla\ind_{u>\xi},\phi_1\phi_2\rangle\dif z+\langle b(x,\xi)\partial_\xi \ind_{u>\xi},\phi_1\phi_2\rangle\dif z-\langle g(x,\xi)\partial_\xi \ind_{u>\xi},\phi_1\phi_2\rangle\dif W\\
&\quad+\frac{1}{2}\langle \partial_\xi(G^2(x,\xi)\partial_\xi\ind_{u>\xi}),\phi_1\phi_2\rangle\dif t
\end{align*}
and conclude that the kinetic formulation \eqref{eq:kinform} with the kinetic measure $m=0$ is valid in the sense of $\mathcal{D}'(\mr^N\times\mr)$ a.s. In general, the kinetic measure is not known in advance and becomes part of the solution that takes account of possible singularities of the solution $u$. In particular, it vanishes in the above computation because we assumed certain level of regularity of $u$.

In order to justify the formulation \eqref{eq:weakkinformul}, note that, due to \cite[Theorem 4]{caruana} (similarly to Proposition \ref{prop:aux}), if $\phi\in C^1_c(\R^N\times\R)$ then the composition $\phi(\theta_t)$ is the unique strong solution to
\begin{align*}
\dd \phi(\theta_t)+\nabla \phi(\theta_t)\cdot a\,\dd z-\partial_\xi\phi(\theta_t)\,b\,\dd z&=0,\\
\phi(\theta_0)&=\phi.
\end{align*}
As a consequence, if $(F,m)$ solves \eqref{eq:kinform} in the sense of $\mathcal{D}'(\R^N\times\R)$ a.s.\footnote{For instance $(F,m)=(\ind_{u>\xi},0)$ from the discussion above.} then applying formally the It\^o formula to the product $\langle F(t),\phi(\theta_t)\rangle$ we deduce
\begin{align}\label{hu}
\begin{aligned}
 \langle F(t),\phi(\theta_t)\rangle&=\langle F_0,\phi\rangle -\int_0^t\langle\nabla F\cdot a,\phi(\theta_t)\rangle\dd z+\int_0^t\langle\partial_\xi F\, b,\phi(\theta_t)\rangle\dd z\\
&\quad-\int_0^t\langle\partial_\xi F\, g,\phi(\theta_t)\rangle\dd W+\frac{1}{2}\int_0^t\langle \partial_\xi(G^2\partial_\xi F),\phi(\theta_t)\rangle\dd t- m\big(\ind_{[0,t)}\partial_\xi \phi(\theta_t)\big)\\
&\quad-\int_0^t\langle F,\nabla\phi(\theta_t)\cdot a\rangle\dd z+\int_0^t\langle F,\partial_\xi\phi(\theta_t)\, b\rangle\dd z,
\end{aligned}
\end{align}
where
\begin{align*}
&-\int_0^t\langle\nabla F\cdot a,\phi(\theta_t)\rangle\dd z+\int_0^t\langle\partial_\xi F\, b,\phi(\theta_t)\rangle\dd z-\int_0^t\langle F,\nabla\phi(\theta_t)\cdot a\rangle\dd z+\int_0^t\langle F,\partial_\xi\phi(\theta_t)\, b\rangle\dd z\\
&\qquad\qquad=-\int_0^t\big\langle\nabla \big[F\phi(\theta_t)\big],a\big\rangle\dd z+\int_0^t\big\langle\partial_\xi\big[F\phi(\theta_t)\big], b\big\rangle\dd z\\
&\qquad\qquad=\int_0^t\big\langle \big[F\phi(\theta_t)\big],\diver a-\partial_\xi b\big\rangle\dd z=0
\end{align*}
due to the fact that $\diver a-\partial_\xi b=0$. Thus, \eqref{hu} is a stronger version of \eqref{eq:weakkinformul} that does not require time dependent test functions.

To conclude, let us also shortly discuss the connection with the notion of kinetic solution from \cite{debus} in the case of $z_t=t$. As no rough integrals then appear on the left hand side of \eqref{eq:kinform}, the flow transformation method used in Definition \ref{kinsol} was not needed in \cite[Definition 2]{debus}. Consequently, their notion of kinetic solution relied on solving the corresponding version of \eqref{eq:kinform} directly in the sense of distributions. According to the above reasoning, this is possible in the case of sufficiently smooth solution to \eqref{eq}. Since the setting of \cite{debus} differs from ours in several ways: spatial domain, $x$-dependence of the coefficient $A$, integrability assumptions on a kinetic solution as well as on a kinetic measure, we will not compare the two notions of kinetic solution further.

\subsection{The main result}
\label{subsec:main}

Our main result reads as follows.

\begin{thm}\label{thm:main}
Let $u_0\in L^2(\Omega;L^1(\mr^N))\cap L^4(\Omega;L^2(\mr^N)) $. Under the above assumptions, the following statements hold true:
\begin{enumerate}
\item[\emph{(i)}] There exists a unique kinetic solution to \eqref{eq}. In addition, it satisfies
\begin{equation}\label{eqqq}
\E\sup_{0\leq t\leq T}\|u(t)\|_{L^1_x}^2+\E\bigg|\int_0^T\|u(t)\|_{L^2_{x}}^2\,\dif t\bigg|^2\leq C\Big(\E\|u_0\|_{L^1_x}^2+\E\|u_0\|_{L^2_x}^4\Big).
\end{equation}
\item[\emph{(ii)}] Any generalized kinetic solution is actually a kinetic solution, that is, if $F$ is a generalized kinetic solution to \eqref{eq} with initial datum $\ind_{u_0>\xi}$ then there exists a kinetic solution $u$ to \eqref{eq} with initial datum $u_0$ such that $F=\ind_{u>\xi}$ for a.e. $(t,x,\xi)$.
\item[\emph{(iii)}] If $u_1,\,u_2$ are kinetic solutions to \eqref{eq} with initial data $u_{1,0}$ and $u_{2,0}$, respectively, then for a.e. $t\in[0,T]$
\begin{equation*}
\E\|(u_1(t)-u_2(t))^+\|_{L^1_x}\leq \E\|(u_{1,0}-u_{2,0})^+\|_{L^1_x}.
\end{equation*}
\end{enumerate}
\end{thm}

\subsection{Application to conservation laws with stochastic flux and forcing}

It is immediately possible to apply our theory to conservation laws of the form \eqref{eq:stoch1}, where $B$ is a stochastic process that is independent on $W$ and that can be enhanced to a stochastic process $\mathbf B$ for which almost every realization is a geometric $p$-H\"older rough path. For instance, one may consider another Brownian motion but also more general Gaussian or Markov processes (the reader is referred to \cite[Part 3]{friz} for further examples).

Let $B$ be a Brownian motion defined on a stochastic basis $(\bar\Omega,\bar{\mf},(\bar{\mf}_t),\bar\prst)$ and let $\mathbf B(\bar\omega)$ be a realization of its (Stratonovich) lift. For $\bar\omega$ fixed, we set
$$\mathbf z=(\mathbf z^1,\mathbf z^2):=\big(\mathbf B^1(\bar\omega),\mathbf B^2(\bar\omega)\big)$$
and Theorem \ref{thm:main} yields the existence of a unique kinetic solution
$$u(\bar\omega)\in L^1(\Omega;L^\infty(0,T;L^1(\mr^N)))\cap L^2(\Omega;L^2(0,T;L^2(\mr^N)))$$
together with the corresponding $L^1$-contraction property.
Let us now set
$$\tilde\Omega=\bar\Omega\times\Omega,\qquad \tilde\mf=\bar\mf\otimes\mf,\qquad \tilde\mf_t=\bar\mf_t\otimes\mf_t,\qquad \tilde\prst=\bar\prst\otimes\prst$$
and
$$\tilde B(\bar\omega,\omega):=B(\bar\omega),\qquad \tilde W(\bar\omega,\omega):=W(\omega).$$
Both these processes are Brownian motions on $(\tilde \Omega,\tilde\mf,(\tilde\mf_t),\tilde\prst)$ and if $\tilde {\mathbf B}$ denotes the (Stratonovich) lift of $\tilde B$ then
$$\tilde {\mathbf B}(\bar\omega,\omega)=\mathbf B(\bar\omega).$$
Moreover, standard results in rough path theory guarantee that the RDE solutions to \eqref{eq:fl1} and \eqref{eq:fl2} driven by $\mathbf z=\tilde{\mathbf B}(\bar\omega,\omega)$ coincide $\tilde\prst$-a.s. with the corresponding SDE solutions and in particular they are $(\tilde\mf_t)$-adapted. It remains to verify that the It\^o integral in the definition of kinetic solution \eqref{eq:weakkinformul}, which is now constructed for a.e. $\bar\omega $ as an It\^o integral on $ (\Omega,\mf,(\mf_t),\prst)$, can be extended to $(\tilde\Omega,\tilde\mf,(\tilde\mf_t),\tilde\prst)$. But that follows directly from its construction: the corresponding Riemann sums converge in $\prst$ for a.e. $\bar\omega$, therefore by Fubini's theorem they also converge in $\tilde\prst.$

As a consequence, \eqref{eq:stoch1} or more precisely its kinetic formulation \eqref{eq:weakkinformul} is solved in the natural (stochastic) sense and Theorem \ref{thm:main} applies pathwise in $\bar\omega$.

\section{Elements of rough path theory}
\label{sec:rough}

In this section, we recall some basic notions and results from rough path theory which are used throughout the paper. For the general exposition we refer the reader to \cite{friz}, \cite{lyons1}, \cite{lyons2}.

Let $G^{[p]}(\mr^d)$ denote the step-$[p]$ nilpotent free group over $\mr^d$ and let us consider the following rough differential equation (RDE)
\begin{equation}\label{eq:rde}
\begin{split}
\dif y&=V(y)\dif\mathbf{x},\\
y(0)&=y_0,
\end{split}
\end{equation}
where $V=(V_1,\dots,V_d)$ is a family of sufficiently smooth vector fields on $\mr^e$ and $\mathbf{x}:[0,T]\rightarrow G^{[p]}(\mr^d)$ is a geometric H\"older $p$-rough path, namely, it is a path with values in $G^{[p]}(\mr^d)$ satisfying
$$\|\mathbf{x}\|_{\frac{1}{p}\text{-H\"ol};[0,T]}=\sup_{0\leq s< t\leq T}\frac{\|\mathbf{x}_{s,t}\|}{|s-t|^{1/p}}<\infty.$$
We denote by $C^{0,\frac{1}{p}\text{-H\"ol}}([0,T];G^{[p]}(\mr^d))$ the space of all geometric H\"older $p$-rough paths.

Definition of a solution to problems of the form \eqref{eq:rde} is based on Davie's lemma (see \cite[Lemma 10.7]{friz}) that gives uniform estimates for ODE solutions depending only on the rough path regularity (e.g. $p$-variation) of the canonical lift of a regular driving signal $x:[0,T]\rightarrow\mr^d.$ As a consequence, a careful limiting procedure yields a reasonable notion of solution to \eqref{eq:rde}. To be more precise, we define a solution to \eqref{eq:rde} as follows.

\begin{defin}\label{def:solutionrde}
Let $\mathbf{x}\in C^{0,\frac{1}{p}\text{-H\"ol}}([0,T];G^{[p]}(\mr^d))$ be a geometric H\"older $p$-rough path and suppose that $(x^n)$ is a sequence of Lipschitz paths with the corresponding step-$[p]$ lifts denoted by $\mathbf{x}^n$, i.e.
$$\mathbf{x}^n_t:=S_{[p]}(x^n)_t=\bigg(1,\int_{0<v<t}\dif x^n_v,\dots,\int_{0<v_1<\cdots<v_{[p]}<t}\dif x^n_{v_1}\otimes\cdots\otimes \dif x^n_{v_{[p]}}\bigg),$$
such that
$$\mathbf{x}^n\longrightarrow \mathbf{x}$$
uniformly on $[0,T]$ and $\sup_n \|\mathbf{x}^n\|_{\frac{1}{p}\text{-H\"ol};[0,T]}<\infty$. We say that $y\in C([0,T];\mr^e)$, also denoted by $\pi_{(V)}(0,y_0;\mathbf{x})$, is a solution to \eqref{eq:rde} provided it is a limit point (in uniform topology on $[0,T]$) of
$$\big\{\pi_{(V)}(0,y_0;x^n);\,n\geq1\big\}$$
where $\pi_{(V)}(0,y_0;x^n)$ denotes the solution to the ODE
\begin{equation*}
\begin{split}
\dif y^n&=V(y^n)\dif{x}^n,\\
y^n(0)&=y_0.
\end{split}
\end{equation*}
\end{defin}

Under a sufficient regularity assumption upon the collection of vector fields $V$, there exists a unique solution to \eqref{eq:rde} which defines a flow of diffeomorphisms.

\begin{thm}\label{thm:existence}
Let $p\geq 1$. Let $\mathbf{x}$ be a geometric H\"older $p$-rough path and suppose that the vector fields $V=(V_1,\dots,V_d)$ are $\mathrm{Lip}^{\gamma+k}$ for some $\gamma>p$. Then the following holds true.

\begin{enumerate}
 \item[\emph{(i)}] There exists a unique solution to \eqref{eq:rde}, say $\pi_{(V)}(0,y_0;\mathbf{x})\in C([0,T];\mr^e)$, and
the map
$$\phi:(t,y)\in[0,T]\times\mr^e\mapsto\pi_{(V)}(0,y;\mathbf{x})_t\in\mr^e$$
is a flow of $C^k$-diffeomorphisms.
 \item[\emph{(ii)}] There exists $C>0$ that depends on $|V|_{\text{\emph{Lip}}^{\gamma+k}}$ and $\|\mathbf{x}\|_{\frac{1}{p}\text{\emph{-H\"ol}};[0,T]}$ such that for every multiindex $\alpha$ with $1\leq |\alpha|\leq k$ the following estimates hold true\footnote{Here, $|\cdot|_{\frac{1}{p}\text{{-H\"ol}};[0,T]}$ denotes the H\"older seminorm of a mapping taking values in $\mr^e$.}
$$\sup_{y\in\mr^e}|\partial_\alpha\phi(y)|_{\frac{1}{p}\text{\emph{-H\"ol}};[0,T]}\leq C,\qquad \sup_{y\in\mr^e}|\partial_\alpha\phi^{-1}(y)|_{\frac{1}{p}\text{\emph{-H\"ol}};[0,T]}\leq C.$$
%
\end{enumerate}

\begin{proof}
The proof of these results can be found in \cite[Proposition 11.11]{friz} and in \cite[Lemma 13]{crisan}.
\end{proof}
\end{thm}

\section{Contraction property}
\label{sec:uniqueness}

Let us start with the question of uniqueness. Our first result gives the existence of representatives of a generalized kinetic solution that possess certain left- and right- continuity properties. This builds the foundation for the key estimate of this section, Proposition \ref{prop:doubling}, as it allows us to strengthen the sense in which \eqref{eq:weakkinformul} is satisfied, namely, we obtain a weak formulation that is only weak in $x,\,\xi$ (cf. Corollary \ref{cor:strongerversion}) and therefore using time dependent test functions is no longer necessary.


\begin{prop}[Left- and right-continuous representatives]\label{limits}
Let $F$ be a generalized kinetic solution to \eqref{eq}. Then $F$ admits representatives $F^-$ and $F^+$ which are a.s. left- and right-continuous, respectively, in the sense of $\mathcal{D}'(\mr^N\times\mr)$. More precisely, for all $t^*\in[0,T)$ there exist kinetic functions $F^{*,+}$ on $\Omega\times\mr^N$ such that setting $F^+(t^*)=F^{*,+}$ yields $F^+=F$ for a.e. $(\omega,t,x,\xi)$ and 
\begin{equation*}
\big\langle F^+(t^*+\,\varepsilon),\phi(\theta_{t^*+\varepsilon})\big\rangle\longrightarrow\big\langle F^{+}(t^*),\phi(\theta_{t^*})\big\rangle\quad\varepsilon\downarrow 0\quad\forall\phi\in C^1_c(\mr^N\times\mr)\quad \text{a.s}.
\end{equation*}
Similarly, for all $t^*\in(0,T]$ there exist kinetic functions $F^{*,-}$ on $\Omega\times\mr^N$ such that setting $F^-(t^*)=F^{*,-}$ yields $F^-=F$ for a.e. $(\omega,t,x,\xi)$ and 
\begin{equation*}
\big\langle F^-(t^*-\,\varepsilon),\phi(\theta_{t^*-\varepsilon})\big\rangle\longrightarrow\big\langle F^{-}(t^*),\phi(\theta_{t^*})\big\rangle\quad\varepsilon\downarrow 0\quad\forall\phi\in C^1_c(\mr^N\times\mr)\quad \text{a.s}.
\end{equation*}
Furthermore, there exists a set of full measure $\mathcal{A}\subset(0,T)$ such that, for all $t\in\mathcal{A}$,
\begin{equation}\label{1s1}
\langle F^-(t),\phi(\theta_t)\rangle=\langle F^+(t),\phi(\theta_t)\rangle\qquad\forall\phi\in C^1_c(\mr^N\times\R)\quad\text{a.s}.
\end{equation}

\begin{proof}
{\em Step 1:} First, we show that $F$ admits representatives $F^-$ and $F^+$ that are left- and right-continuous, respectively, in the required sense.
It follows the ideas of \cite[Proposition 8]{debus} and \cite[Proposition 3.1]{hof}, nevertheless, as our mixed rough-stochastic setting introduces new difficulties we present the proof in full detail.

As the space $C^1_c(\mr^N\times\mr)$ (endowed with the topology of the uniform convergence on any compact set of functions and their first derivatives) is separable, let us fix a countable dense subset $\mathcal{D}_1$. Let $\phi\in \mathcal{D}_1$ and $\alpha\in C^1_c([0,T))$.
Integration by parts and the stochastic version of Fubini's theorem applied to \eqref{eq:weakkinformul} yield
\begin{equation*}
\int_0^T g_\phi(t)\alpha'(t)\dif t+\langle F_0,\phi\rangle\alpha(0)=\langle m,\partial_\xi\phi(\theta_t)\rangle(\alpha)\qquad\prst\text{-a.s.}
\end{equation*}
where
\begin{equation}\label{fceg}
\begin{split}
g_\phi(t)&=\big\langle F(t),\phi(\theta_t)\big\rangle+\int_0^t\langle\partial_\xi F g,\phi(\theta_s)\rangle\dif W-\frac{1}{2}\int_0^t\langle\partial_\xi(G^2\partial_\xi F,\phi(\theta_s)\rangle\dif s.
\end{split}
\end{equation}
Hence $\partial_tg_\phi$ is a (pathwise) Radon measure on $[0,T]$ and by the Riesz representation theorem $g_\phi\in BV([0,T])$. 
Due to the properties of $BV$-functions, we obtain that 
$g_\phi$  admits left- and right-continuous representatives which coincide except for an at most countable set.
Moreover, apart from the first one all terms in \eqref{fceg} are almost surely continuous in $t$.  Hence, on a set of full measure, denoted by $\Omega_\phi$, $\langle F(t),\phi(\theta_t)\rangle$ also
admits left- and right-continuous representatives which coincide except for an at most countable set.
Let them be denoted by
$\langle F,\phi(\theta_t)\rangle^\pm$ and set
$\Omega_0=\cap_{\phi\in\mathcal{D}_1}\Omega_\phi.$ Since $\mathcal{D}_1$ is countable, $\Omega_0$ is also a set of full measure.
Besides, for $\phi\in \mathcal{D}_1$, $(t,\omega)\mapsto \langle F(t,\omega),\phi(\theta_t)\rangle^\pm$ has left- and right-continuous trajectories 
in time, respectively,
and are thus measurable with respect to $(t,\omega)$. For $\phi\in C^1_c(\mr^N\times\mr)$, we define $\langle F(t,\omega),\phi(\theta_t)\rangle^\pm$ on 
$[0,T]\times \Omega_0$ as the limit 
of $\langle F(t,\omega),\phi_n(\theta_t)\rangle^\pm$ for any sequence $(\phi_n)$ in $\mathcal{D}_1$ converging to $\phi$ in the topology of $C^1_c(\R^N\times\R)$. Then as a pointwise limit of measurable functions $\langle F(\cdot,\cdot),\phi(\theta_t)\rangle^\pm$ is also measurable in $(t,\omega)$. Moreover, due to the uniform convergence $\phi_n$ to $\phi$ and boundedness of $F$, it has left- and right-continuous 
trajectories, respectively.
Consequently, let $\omega\in\Omega_0$ and let $\mathcal{N}_\omega\subset(0,T)$ denote the corresponding countable set of times, where
$$\langle F(t,\omega),\phi(\theta_t)\rangle^-\neq\langle F(t,\omega),\phi(\theta_t)\rangle^+\qquad\text{for some}\quad\phi\in\mathcal{D}_1.$$
It follows from the Fubini theorem that there exists a set of full measure $\mathcal{A}\subset(0,T)$ such that for all $t\in\mathcal{A}$ it holds true that $t\notin\mathcal{N}_\omega$ for a.e. $\omega\in\Omega$, i.e.
$$\langle F(t),\phi(\theta_t)\rangle^-=\langle F(t),\phi(\theta_t)\rangle^+\qquad\forall\phi\in\mathcal{D}_1\quad\text{a.s.}$$
thus
\begin{equation}\label{1s}
\langle F(t),\phi(\theta_t)\rangle^-=\langle F(t),\phi(\theta_t)\rangle^+\qquad\forall\phi\in C^1_c(\mr^N\times\R)\quad\text{a.s}.
\end{equation}

Let us proceed with the construction of $F^+,$ the construction of $F^-$ being analogous. In the sequel, we will frequently write expressions of the type $F(t,\pi_{t})$ as a shorthand for the composition of $F(t,\cdot)$ with $\pi_t$ (similarly for $F^+$ etc.). Therefore, $\langle F(t,\pi_t),\phi\rangle $ is understood as
$$\int_{\mr^{N+1}}F(t,\pi_t(x,\xi))\phi(x,\xi)\,\dd x\,\dd\xi$$
and since the flow $\pi$ is volume preserving
$$\langle F(t,\pi_t),\phi\rangle=\int_{\mr^{N+1}}F(t,x,\xi)\phi(\theta_t(x,\xi))\,\dd x\,\dd\xi=\langle F(t),\phi(\theta_t)\rangle.$$

It is now straightforward to define $F^+(t),\, t\in[0,T),$ by
\begin{equation}\label{eq:ff}
\big\langle F^+(t),\phi(\theta_t)\big\rangle:=\langle F(t),\phi(\theta_t)\rangle^+,\qquad\phi\in C^1_c(\mr^N\times\mr) \qquad \text{a.s.}
\end{equation}
and observe that $F^+(t)$ is right-continuous in the required sense.
Note that $F^+(t)$ is a.s. a well-defined distribution since, due to the flow properties of $\pi$ and $\theta$, the mapping $\phi\mapsto\phi(\theta_t)$ is a bijection on $C^1_c(\mr^N\times\mr)$ with the inverse mapping $\phi\mapsto\phi(\pi_t)$. Moreover, it belongs to $L^\infty(\mr^N\times\R)$ a.s. Indeed, for $\omega\in\Omega_0$ and $\phi\in \mathcal{D}_1$, if $t\in[0,T)$ is such that
$$\langle F(t,\omega),\phi(\theta_t)\rangle^+=\langle F(t,\omega),\phi(\theta_t)\rangle$$
then it follows immediately from the fact that $F(t)\in L^\infty(\mr^N\times\R)$ a.s.
For a general $t\in[0,T)$ we take a sequence $t_n\searrow t$ with the property above and use lower semicontinuity. Moreover, seen as a function $F^+ :\Omega\times [0, T ] \to L^p_{loc}(\R^N\times\R)$, for some $p\in[1,\infty)$, it is weakly measurable and therefore measurable. Hence, according to the Fubini theorem
$F^+$, as a function of four variables $\omega,t,x,\xi$, is measurable.

Next, we show that $F^+$ is a representative (in time) of $F$, i.e. for a.e. $t^*\in[0,T)$ it holds that $F^+(t^*)=F(t^*)$ where the equality is understood a.e. in $(\omega,x,\xi)$. Indeed, due to the Lebesgue differentiation theorem,
\begin{equation*}
\lim_{\varepsilon\rightarrow 0}\frac{1}{\varepsilon}\int_{t^*}^{t^*+\varepsilon}F\big(t,\pi_t(x,\xi)\big)\,\dif t=F\big(t^*,\pi_{t^*}(x,\xi)\big)\qquad\text{a.e. }(\omega,t^*,x,\xi)
\end{equation*}
hence by the dominated convergence theorem
\begin{equation*}
\lim_{\varepsilon\rightarrow 0}\frac{1}{\varepsilon}\int_{t^*}^{t^*+\varepsilon}\big\langle F(t,\pi_t),\phi\big\rangle\,\dif t=\big\langle F(t^*,\pi_{t^*}),\phi\big\rangle\qquad\forall\phi\in C_c^1(\mr^N\times\mr)\quad\text{a.e. } (\omega,t^*).
\end{equation*}
Since the left hand side is equal to $\big\langle F^+(t^*,\pi_{t^*}),\phi\big\rangle$ for all $t^*\in [0,T]$ and $\omega\in\Omega_0$, it follows that
$$\big\langle F^+(t^*),\phi(\theta_{t^*})\big\rangle=\big\langle F(t^*),\phi(\theta_{t^*})\big\rangle\qquad\forall\phi\in C_c^1(\mr^N\times\mr)\quad\text{a.e. } (\omega,t^*)$$
which implies
$$ F^+(t^*)= F(t^*)\qquad\text{in}\quad L^\infty(\R^N\times\R)\quad\text{a.e. } (\omega,t^*)$$
and Fubini theorem together with the measurability of $F^+$ and $F$ regarded as functions of four variables $\omega,t,x,\xi$ yields
$$ F^+(t^*,x,\xi)= F(t^*,x,\xi)\qquad\text{a.e. } (\omega,t^*,x,\xi).$$


{\em Step 2:} Second, we prove that for all $t^*\in[0,T)$
\begin{equation}\label{eq:aaa}
\big\langle F^+(t^*\pm\,\varepsilon),\phi\big\rangle\longrightarrow\big\langle F^{+}(t^*),\phi\big\rangle\quad\quad\forall\phi\in C^1_c(\mr^N\times\mr) \quad\text{a.s}.
\end{equation}
Towards this end, we verify
\begin{equation}\label{eq:aaaa}
\big\langle F^+(t^*+\varepsilon,\pi_{t^*}),\phi\big\rangle\longrightarrow\big\langle F^{+}(t^*,\pi_{t^*}),\phi\big\rangle\quad\quad\forall\phi\in C^1_c(\mr^N\times\mr)\quad\text{a.s}.
\end{equation}
and observe that since the flow $\pi$ is volume preserving, testing in \eqref{eq:aaaa} by $\phi(\pi_{t^*})\in C^1_c(\mr^N\times\mr)$ yields \eqref{eq:aaa}.
In order to prove \eqref{eq:aaaa}, we write
\begin{align*}
\big| &\langle F^+(t^*+\varepsilon,\pi_{t^*}),\phi\rangle- \langle F^+(t^*,\pi_{t^*}),\phi\rangle\big|\\
&\quad\leq\big| \langle F^+(t^*+\varepsilon,\pi_{t^*}),\phi\rangle-\big\langle F^+(t^*+\varepsilon,\pi_{t^*+\varepsilon}),\phi\big\rangle\big|+\big|\big\langle F^+(t^*+\varepsilon,\pi_{t^*+\varepsilon}),\phi\big\rangle-\big\langle F^+(t^*,\pi_{t^*}),\phi\big\rangle\big|.
\end{align*}
The second term on the right hand side vanishes a.s. as $\varepsilon\rightarrow0$ according to {\em Step 1} and regarding the first one, we have
\begin{align*}
\big| &\langle F^+(t^*+\varepsilon,\pi_{t^*}),\phi\rangle-\big\langle F^+(t^*+\varepsilon,\pi_{t^*+\varepsilon}),\phi\big\rangle\big|=\big|\big\langle F^+(t^*+\varepsilon),\phi(\theta_{t^*})-\phi(\theta_{t^*+\varepsilon})\big\rangle\big|
\end{align*}
and we argue by the dominated convergence theorem:
according to Theorem \ref{thm:existence}, if $\supp\phi\subset B_{R}$, where $B_{R}\subset\mr^N\times\mr$ is the ball of radius $R$ centred at $0$, then $\phi(\theta_{t^*})$ and $\phi(\theta_{t^*+\varepsilon})$ remain compactly supported uniformly in $\varepsilon$ and their support is included in $ B_{R+CT^{1/p}}$.
Besides,
$$\phi(\theta_{t^*}(x,\xi))-\phi(\theta_{t^*+\varepsilon}(x,\xi))\longrightarrow 0\qquad\forall (x,\xi)\in\mr^N\times\mr$$
and since $F^+$ takes values in $[0,1]$, the claim follows.

{\em Step 3:}
Now, it only remains to show that $F^+(t^*)$ is a kinetic function on $X=\Omega\times\mr^N$ for all $t^*\in[0,T)$. Towards this end, we observe that for all $t^*\in[0,T)$
$$F_n(t^*,x,\xi):=\frac{1}{\varepsilon_n}\int_{t^*}^{t^*+\varepsilon_n}F(t,x,\xi)\,\dif t$$
is a kinetic function on $X=\Omega\times\mr^N$ and by \eqref{integrov} the assumptions of Lemma \ref{kinetcomp} are fulfilled\footnote{Note that we may assume without loss of generality that the $\sigma$-field $\mathcal{F}$ is countably generated and hence, according to \cite[Proposition 3.4.5]{cohn}, the space $L^1(\Omega)$ is separable.}. Accordingly, there exists a kinetic function $F^{*,+}$
and a subsequence $(n^*_k)$ such that, on the one hand,
\begin{equation*}
F_{n_k^*}(t^*)\overset{w^*}{\longrightarrow} F^{*,+}\qquad \text{in}\qquad L^\infty(\Omega\times\mr^N\times\mr).
\end{equation*}
Since, on the other hand, we have due to {\em Step 2} that
\begin{equation*}
F_{n^*_k}(t^*)\longrightarrow F^+(t^*)\qquad \text{in}\qquad \mathcal{D}'(\mr^N\times\mr)\quad\text{a.s.},
\end{equation*}
we deduce that $F^+(t^*)=F^{*,+}$ for all $t^*\in[0,T)$, which completes the proof.

{\em Step 4:}
The proof of existence of the left-continuous representative $F^-$ on $(0,T]$ can be carried out similarly and \eqref{1s1} then follows immediately from \eqref{1s}.
\end{proof}
\end{prop}

Remark that according to Proposition \ref{limits}, a stronger version of \eqref{eq:weakkinformul} holds true provided $F$ is replaced by $F^+$ or $F^-$. The precise result is presented in the following corollary.

\begin{cor}\label{cor:strongerversion}
Let $F$ be a generalized kinetic solution to \eqref{eq} and let $F^-$ and $F^+$ be its left- and right-continuous representatives. Then the couples $(F^+,m)$ and $(F^-,m)$, respectively, satisfy for all $t\in(0,T)$ and every $\phi\in C^1_c(\mr^N\times\mr)$, a.s.,
\begin{align*}
\langle F^+(t),\phi(\theta_t)\rangle&=\langle F_{0},\phi\rangle-\int_0^t\langle\partial_\xi F g,\phi(\theta_s)\rangle\dif W+\frac{1}{2}\int_0^t\langle\partial_\xi(G^2\partial_\xi F,\phi(\theta_s)\rangle\dif s-m\big(\ind_{[0,t]}\partial_{\xi}\phi(\theta_s)\big)
\end{align*}
and
\begin{align*}
\langle F^-(t),\phi(\theta_t)\rangle&=\langle F_{0},\phi\rangle-\int_0^t\langle\partial_\xi F g,\phi(\theta_s)\rangle\dif W+\frac{1}{2}\int_0^t\langle\partial_\xi(G^2\partial_\xi F,\phi(\theta_s)\rangle\dif s-m\big(\ind_{[0,t)}\partial_{\xi}\phi(\theta_s)\big).
\end{align*}
Furthermore, setting $F^-(0):=F_0$, then for every $\phi\in C^1_c(\mr^N\times\mr)$ and $t^*\in [0,T)$ it holds true
\begin{equation}\label{eq:cont1}
\big\langle F^+(t^*)-F^-(t^*),\phi(\theta_{t^*})\big\rangle=-m\big(\ind_{\{t^*\}}\partial_\xi\phi(\theta_{t^*})\big)\quad\text{a.s}.
\end{equation}
In particular, if $\mathcal{A}\subset(0,T)$ is the set of full measure constructed in Proposition \ref{limits}, then for all $t^*\in\mathcal{A}$, $m$ does not have atom at $t^*$ a.s.

\begin{proof}
Consider \eqref{eq:weakkinformul} with a test function of the form $(s,x,\xi)\mapsto \phi(x,\xi)\alpha_\varepsilon(s)$ where $\phi\in C^1_c(\mr^N\times\mr)$ and
$$\alpha_\varepsilon(s)=\begin{cases}
						1,&\quad s\leq t,\\
						1-\frac{s-t}{\varepsilon},&\quad t\leq s\leq t+\varepsilon,\\
						0,&\quad t+\varepsilon\leq s.
						\end{cases}$$
Due to Proposition \ref{limits} we obtain convergence of the left hand side of \eqref{eq:weakkinformul} as $\varepsilon\rightarrow0$:
\begin{align*}
\int_0^T\langle F(s),\phi(\theta_s)\rangle\,\partial_s\alpha_\varepsilon(s)\,\dif s=-\frac{1}{\varepsilon}\int_t^{t+\varepsilon}\langle F(s),\phi(\theta_s)\rangle\,\dif s\longrightarrow-\big\langle F^+(t),\phi(\theta_t)\big\rangle.
\end{align*}
Then we have
\begin{align*}
m\big(\alpha_\varepsilon(s) \partial_\xi\phi(\theta_s)\big)&=\int_{[0,t)\times\mr^N\times\mr}\partial_\xi\phi(\theta_s)\,\dif m+\int_{[t,t+\varepsilon)\times\mr^N\times\mr}\Big(1-\frac{s-t}{\varepsilon}\Big)\partial_\xi\phi(\theta_s)\,\dif m\\
&\longrightarrow \int_{[0,t]\times\mr^N\times\mr}\partial_\xi\phi(\theta_s)\,\dif m
\end{align*}
and the convergence of the remaining two terms on the right hand side of \eqref{eq:weakkinformul} is obvious because they are continuous in $t$.
Therefore, we have justified the equation for $(F^+,m)$.

Concerning $(F^-,m)$, we apply a similar approach but use the function
$$\alpha_\varepsilon(s)=\begin{cases}
                         1,&s\leq t-\varepsilon,\\
			 \frac{t-s}{\varepsilon},& t-\varepsilon\leq s\leq t,\\
			 0,&t\leq s
                        \end{cases}$$
instead.

Finally, the formula \eqref{eq:cont1} follows from \eqref{eq:weakkinformul} by testing by $(t,x,\xi)\mapsto\phi(x,\xi)\alpha_\varepsilon(t)$ where
$$\alpha_\varepsilon(t)=\frac{1}{\varepsilon}\min\big\{(t-t^*+\varepsilon)^+,(t-t^*-\varepsilon)^-\big\}$$
and sending $\varepsilon\rightarrow 0$. As a consequence, $m$ has an atom at $t^*$ if and only if $F^-(t^*)\neq F^+(t^*)$ hence, in view of \eqref{1s1} the proof is complete.
\end{proof}
\end{cor}

\begin{lemma}\label{lem:equil}
Let $F$ be a generalized kinetic solution to \eqref{eq} and let $F^+$ be its right-continuous representative. Assume that the initial condition $F_0$ is at equilibrium: there exists $u_0\in L^1(\Omega\times\mr^N)$ such that $F_0=\ind_{u_0>\xi}$. Then $F^+(0)=F_0$ and in particular the corresponding kinetic measure $m$ has no atom at $0$ a.s., i.e. the restriction of $m$ to $\{0\}\times\mr^N\times\mr$ vanishes a.s.

\begin{proof}
%
Let $m_0$ be the restriction of $m$ to $\{0\}\times\mr^N\times\mr$. Then we deduce from \eqref{eq:cont1} at $t^*=0$ (recall that $F^-(0)=F_0$) that
\begin{equation}\label{m0}
\langle F^+(0)-\ind_{u_0>\xi},\phi\rangle=-m_0\big(\partial_\xi \phi\big)\qquad\forall \phi\in C^1_c(\mr^N\times\mr)\quad\text{a.s}.
\end{equation}
Let $H_R$ be a smooth truncation on $\R$ such that $0\leq H_R\leq 1$, $H_R(\xi)= 1$ if $|\xi|\leq R$ and $H_R(\xi)= 0$ if $|\xi|\geq 2R$, $|\partial_\xi H_R|
\leq 1$. For $\phi\in C^1_c(\mr^N)$ we intend to pass to the limit $R\to\infty$ in
\begin{equation}\label{fg67}
\langle F^+(0)-\ind_{u_0>\xi},\phi H_R\rangle=-m_0\big(\phi\partial_\xi H_R\big)
\end{equation}
Since $m$ and consequently $m_0$ is a finite measure a.s., the right hand side converges to $0$ a.s. as $R\to \infty$ due to the dominated convergence theorem, whereas for the left hand side, we write
$$\langle F^+(0)-\ind_{u_0>\xi},\phi H_R\rangle=\langle F^+(0)-\ind_{0>\xi},\phi H_R\rangle-\langle\ind_{u_0>\xi}-\ind_{0>\xi},\phi H_R\rangle.$$
We make use of Remark \ref{chi} and there introduced shorthand notation $\chi_{u_0}=\chi_{\ind_{u_0>\xi}}$,
which together with the dominated convergence theorem implies the convergence
$$\langle\ind_{u_0>\xi}-\ind_{0>\xi},\phi H_R\rangle\longrightarrow\langle \chi_{u_0},\phi\rangle.$$
Consequently, we deduce that also the remaining term from \eqref{fg67}, i.e.
$$\langle F^+(0)-\ind_{0>\xi},\phi H_R\rangle$$
is converging since all the other terms converge.
Therefore, necessarily,
$$\langle F^+(0)-\ind_{0>\xi},\phi H_R\rangle\longrightarrow\langle F^+(0)-\ind_{0>\xi},\phi\rangle.$$

As a consequence, we obtain
$$\int_\mr F^+(0,\xi)-\ind_{0>\xi}\,\dif\xi=u_0\quad\text{a.s.}$$
and using this fact one can easily observe that
$$p(\xi):=\int_{-\infty}^\xi\ind_{u_0>\zeta}-F^+(0,\zeta)\,\dif \zeta=\int_{-\infty}^\xi(\ind_{u_0>\zeta}-\ind_{0>\zeta})-(F^+(0,\zeta)-\ind_{0>\zeta})\,\dif \zeta\geq0\quad\text{a.s.}$$
Indeed, $p(-\infty)=p(\infty)=0$ and $p$ is increasing if $\xi\in (-\infty,u_0)$ and decreasing if $\xi\in(u_0,\infty)$.
However, it follows from \eqref{m0} that $p=-m_0$ and since $m_0$ is nonnegative, we deduce that $m_0\equiv 0$.
\end{proof}
\end{lemma}

The following result allows us to restart the evolution given by \eqref{eq} at an arbitrary time $t^*\in(0,T)$.

\begin{lemma}\label{lemma:neworigin}
Let $F$ be a generalized kinetic solution to \eqref{eq} on $[0,T]$ with the initial datum $F_0$. Then for every $t^*\in[0,T)$, $t\mapsto F(t^*+t)$ is a generalized kinetic solution to \eqref{eq} on $[0,T-t^*]$ with the initial datum $F^-(t^*)$.

\begin{proof}
Let $\alpha\in C^1_c([0,T))$, $\phi\in C^1_c(\mr^N\times\mr)$ and test \eqref{eq:weakkinformul} by $(t,x,\xi)\mapsto\phi(x,\xi)\alpha_\varepsilon(t)$ where
$$\alpha_\varepsilon(t)=\begin{cases}
						\alpha(t),&\quad t\leq t^*,\\
						\alpha(t)\Big(1-\frac{t-t^*}{\varepsilon}\Big),&\quad t^*\leq t\leq t^*+\varepsilon,\\
						0,&\quad t^*+\varepsilon\leq t.
						\end{cases}$$
In the limit $\varepsilon\rightarrow0$ we infer
\begin{align}\label{eq:cont2}
\begin{aligned}
\int_0^{t^*}&\langle F(t),\phi(\theta_t)\rangle\partial_t\alpha(t)\,\dif t-\langle F^+(t^*),\phi(\theta_{t^*})\rangle\alpha(t^*)+\langle F_0,\phi\rangle\alpha(0)\\
&=\int_0^{t^*}\langle\partial_\xi F g,\phi(\theta_t)\rangle\alpha(t)\,\dif W-\frac{1}{2}\int_0^{t^*}\langle\partial_\xi(G^2\partial_\xi F),\phi(\theta_t)\rangle\alpha(t)\,\dif t+m\big(\ind_{[0,t^*]}\alpha\partial_\xi\phi(\theta_t)\big).
\end{aligned}
\end{align}
Subtracting \eqref{eq:cont2} from \eqref{eq:weakkinformul} and using \eqref{eq:cont1} leads to
\begin{align*}
\begin{aligned}
\int_{t^*}^T&\langle F(t),\phi(\theta_t)\rangle\partial_t\alpha(t)\,\dif t+\langle F^-(t^*),\phi(\theta_{t^*})\rangle\alpha(t^*)\\
&=\int^T_{t^*}\langle\partial_\xi F g,\phi(\theta_t)\rangle\alpha(t)\,\dif W-\frac{1}{2}\int^T_{t^*}\langle\partial_\xi(G^2\partial_\xi F),\phi(\theta_t)\rangle\alpha(t)\,\dif t+m\big(\ind_{[t^*,T]}\alpha\partial_\xi\phi(\theta_t)\big).
\end{aligned}
\end{align*}
Finally we take $(t,x,\xi)\mapsto\phi(\pi_{t^*}(x,\xi))\alpha(t)$ as a test function and deduce that
\begin{align*}
\begin{aligned}
\int_{t^*}^T&\langle F(t),\phi(\theta_{t^*,t})\rangle\partial_{t}\alpha(t)\,\dif t+\langle F^-(t^*),\phi\rangle\alpha(t^*)\\
&=\int^T_{t^*}\langle\partial_\xi F g,\phi(\theta_{t^*,t})\rangle\alpha(t)\,\dif W-\frac{1}{2}\int^T_{t^*}\langle\partial_\xi(G^2\partial_\xi F),\phi(\theta_{t^*,t})\rangle\alpha(t)\,\dif t+m\big(\ind_{[t^*,T]}\alpha\partial_\xi\phi(\theta_{t^*,t})\big).
\end{aligned}
\end{align*}
As a consequence, $t\mapsto F(t^*+t)$ is a kinetic solution to \eqref{eq} on $[0,T-t^*]$ with initial condition $F^-(t^*)$.
\end{proof}
\end{lemma}

We proceed with a result that plays the key role in our proof of reduction of generalized kinetic solution to a kinetic one, Theorem \ref{thm:reduction}, and the $L^1$-contraction property, Corollary \ref{cor:contraction}. To be more precise, the relevant estimate that we need to verify in case of two generalized kinetic solutions $F_1,F_2$ with initial conditions $F_{1,0}$ and $F_{2,0}$, respectively, is the following:
$$\int_{\mr^N\times\mr} F_1(t) (1-F_2(t))\,\dif x\,\dif \xi\leq\int_{\mr^N\times\mr} F_{1,0} (1-F_{2,0})\,\dif x\,\dif \xi.$$
Indeed, setting $F:=F_1=F_2$ then leads to the Reduction Theorem and using the identity
$$\int_{\mr} \ind_{u_1(t)>\xi} (1-\ind_{u_2(t)>\xi})\,\dif \xi=(u_1(t)-u_2(t))^+$$
gives the $L^1$-contraction property. However, it is not possible to multiply the two equations directly, i.e. the equation for $F_1$ and the one for $1-F_2$ as it is necessary to mollify first. This is taken care of in Proposition \ref{prop:doubling}.
In order to simplify the notation, we denote $\overline{F}=1-F$.

%
%
%
%

In the sequel, let $(\varrho_\delta)$ be an approximation to the identity on $\mr^N\times\mr$ and $(\kappa_R)$ be a truncation on $\mr^N$ such that $\kappa_R\equiv 1$ on $B_R$, $\supp\kappa_R\subset B_{2R}$, $0\leq \kappa_R\leq 1$ and $|\nabla_x\kappa_R|\leq R^{-1}$.

\begin{prop}\label{prop:doubling}
Let $F_1,\, F_2$ be generalized kinetic solutions to \eqref{eq}. Let $s,\,t\in[0,T)$, $s\leq t$, be such that neither of the kinetic measures $m_1$ and $m_2$ has atom at $s$ a.s. Then it holds true
\begin{align}\label{eq:doubling}
\begin{aligned}
\E\int&\,\kappa_R(x)\langle F_1^+(t),\varrho_\delta((x,\xi)-\theta_{s,t}(\cdot,\cdot))\rangle_{z,\sigma}\langle \overline{F}_2^+(t),\varrho_\delta((x,\xi)-\theta_{s,t}(\cdot,\cdot))\rangle_{y,\zeta}\dif x\dif\xi\\
&\quad-\E\int\kappa_R(x)\langle F_1^-(s),\varrho_\delta((x,\xi)- (\cdot,\cdot))\rangle_{z,\sigma}\langle \overline{F}_2^-(s),\varrho_\delta((x,\xi)- (\cdot,\cdot))\rangle_{y,\zeta}\dif x\dif\xi\\
&\leq C(t-s)^{1/p}\bigg(\stred\int_{(s,t]}\int\dif m_1(r,z,\sigma)+\stred\int_{(s,t]}\int\dif m_2(r,y,\zeta)\bigg)\\
&\quad+C(t-s)^{1/p}\bigg(\stred\int_s^t\int |\zeta|^2\dif\nu^2_{r,y}(\zeta)\dif y\dif r+\stred\int_s^t\int |\sigma|^2\dif\nu_{r,z}^1(\sigma)\dif z\dif r\bigg)\\
&\quad+CR\delta(t-s)^{1+1/p}.
\end{aligned}
\end{align}

\begin{proof}
{\em Step 1:}
Applying Corollary \ref{cor:strongerversion} and Lemma \ref{lemma:neworigin} to the generalized kinetic solutions $F_1,F_2$ and to the mollifier $\varrho_\delta$ in place of a test function, we deduce that for all $(x,\xi)\in \mr^N\times\mr$
\begin{equation}\label{eq:stronger}
\begin{split}
\langle& F_1^+(t),\varrho_\delta((x,\xi)-\theta_{s,t}(\cdot,\cdot))\rangle_{z,\sigma}\\
&=\langle F_1^-(s),\varrho_\delta((x,\xi)- (\cdot,\cdot))\rangle_{z,\sigma}-\int_s^t\big\langle\partial_\sigma F_1g,\varrho_\delta((x,\xi)-\theta_{s,r}(\cdot,\cdot))\big\rangle_{z,\sigma}\,\dif W\\
&\qquad-m_1\Big(\ind_{[s,t]}(\cdot)\partial_{\sigma}\varrho_\delta((x,\xi)-\theta_{s,r}(\cdot,\cdot))\Big)-\frac{1}{2}\int_s^t\big\langle G^2\partial_\sigma F_1,\partial_\sigma\varrho_\delta((x,\xi)-\theta_{s,r}(\cdot,\cdot))\big\rangle_{z,\sigma}\dif r
\end{split}
\end{equation}
and similarly
\begin{equation}\label{eq:stronger1}
\begin{split}
\langle& \overline{F}_2^+(t),\varrho_\delta((x,\xi)-\theta_{s,t}(\cdot,\cdot))\rangle_{y,\zeta}\\
&=\langle \overline{F}_2^-(s),\varrho_\delta((x,\xi)- (\cdot,\cdot))\rangle_{y,\zeta}+\int_s^t\big\langle\partial_\zeta F_2g,\varrho_\delta((x,\xi)-\theta_{s,r}(\cdot,\cdot))\big\rangle_{y,\zeta}\,\dif W\\
&\qquad+m_2\Big(\ind_{[s,t]}(\cdot)\partial_{\zeta}\varrho_\delta((x,\xi)-\theta_{s,r}(\cdot,\cdot))\Big)+\frac{1}{2}\int_s^t\big\langle G^2\partial_\zeta F_2,\partial_\zeta\varrho_\delta((x,\xi)-\theta_{s,r}(\cdot,\cdot))\big\rangle_{y,\zeta}\dif r
\end{split}
\end{equation}
Our notation in the above is the following: the variable $(x,\xi)$ is fixed in both \eqref{eq:stronger} and \eqref{eq:stronger1}; in \eqref{eq:stronger} we consider the test function $(z,\sigma)\mapsto\varrho_\delta((x,\xi)-\theta_{s,t}(z,\sigma))$ whereas in \eqref{eq:stronger1} we make use of $(y,\zeta)\mapsto\varrho_\delta((x,\xi)-\theta_{s,t}(y,\zeta))$.

We denote
\begin{align*}
\mathcal{I}_1(t)&:=-\int_s^t\big\langle\partial_\sigma F_1g,\varrho_\delta((x,\xi)-\theta_{s,r}(\cdot,\cdot))\big\rangle_{z,\sigma}\,\dif W,\\
\mathcal{I}_2(t)&:=\int_s^t\big\langle\partial_\zeta F_2g,\varrho_\delta((x,\xi)-\theta_{s,r}(\cdot,\cdot))\big\rangle_{y,\zeta}\,\dif W,
\end{align*}
and observe that
\begin{align*}
\mu_1:C_b([s,T])&\rightarrow\mr,\\
\alpha&\mapsto -m_1\Big(\alpha(\cdot)\,\partial_{\sigma}\varrho_\delta((x,\xi)-\theta_{s,r}(\cdot,\cdot)))\Big)-\frac{1}{2}\int_s^T\alpha(r)\big\langle G^2\partial_\sigma F_1,\partial_\sigma\varrho_\delta((x,\xi)-\theta_{s,r}(\cdot,\cdot))\big\rangle_{z,\sigma}\dif r,
\end{align*}
and
\begin{align*}
\mu_2:C_b([s,T])&\rightarrow\mr,\\
\alpha&\mapsto m_2\Big(\alpha(\cdot)\,\partial_{\zeta}\varrho_\delta((x,\xi)-\theta_{s,r}(\cdot,\cdot)))\Big)+\frac{1}{2}\int_s^T\alpha(r)\big\langle G^2\partial_\zeta F_2,\partial_\zeta\varrho_\delta((x,\xi)-\theta_{s,r}(\cdot,\cdot))\big\rangle_{y,\zeta}\dif r,
\end{align*}
are Radon measures hence
$$t\mapsto \mu_1([s,t]),\qquad t\mapsto \mu_2([s,t])$$
are c\`adl\`ag functions of bounded variation.

Multiplying \eqref{eq:stronger} and \eqref{eq:stronger1}, we have
\begin{align}\label{123}
\begin{aligned}
\E\langle F_1^+(t),&\,\varrho_\delta((x,\xi)-\theta_{s,t}(\cdot,\cdot))\rangle_{z,\sigma}\langle \overline{F}_2^+(t),\varrho_\delta((x,\xi)-\theta_{s,t}(\cdot,\cdot))\rangle_{y,\zeta}\\
&\qquad-\E\langle F_1^-(s),\varrho_\delta((x,\xi)- (\cdot,\cdot))\rangle_{z,\sigma}\langle \overline{F}_2^-(s),\varrho_\delta((x,\xi)- (\cdot,\cdot))\rangle_{y,\zeta}\\
&=\E\langle F_1^-(s),\varrho_\delta((x,\xi)- (\cdot,\cdot))\rangle_{z,\sigma}\mathcal{I}_2(t)+\stred\langle F_1^-(s),\varrho_\delta((x,\xi)-\theta_s(\cdot,\cdot))\rangle_{z,\sigma}\mu_2([s,t])\\
&\quad+\E\mathcal{I}_1(t)\langle \overline{F}_2^-(s),\varrho_\delta((x,\xi)- (\cdot,\cdot))\rangle_{y,\zeta}+\E\mathcal{I}_1(t)\mathcal{I}_2(t)+\E\mathcal{I}_1(t)\mu_2([s,t])\\
&\quad+\E\mu_1([s,t])\langle \overline{F}_2^-(s),\varrho_\delta((x,\xi)-\theta_s(\cdot,\cdot))\rangle_{y,\zeta}+\E\mu_1([s,t])\mathcal{I}_2(t)+\stred\mu_1([s,t])\mu_2([s,t])\\
&=J_1+\cdots+J_8.
\end{aligned}
\end{align}
For $J_5$, we apply the It\^o formula to the product of a continuous martingale and a c\`adl\`ag function of bounded variation, that is $\mathcal{I}_1(t)$ and $\mu_2([s,t])$, to deduce the following integration by parts formula (see e.g. \cite[Chapter II, Theorem 33]{protter})
\begin{align*}
\mathcal{I}_1(t)\mu_2([s,t])&=\int_s^t\mu_2([s,r))\,\dif\mathcal{I}_1(r)+\int_{(s,t]}\mathcal{I}_1(r)\,\dif\mu_2(r)
\end{align*}
which implies
$$ J_5=\stred\int_{(s,t]}\mathcal{I}_1(r)\,\dif\mu_2(r)$$
and similarly we obtain
$$ J_7=\stred\int_{(s,t]}\mathcal{I}_2(r)\,\dif\mu_1(r).$$
Next,
\begin{align*}
J_1=\E\int_s^t\langle F_1^-(s),\varrho_\delta((x,\xi)- (\cdot,\cdot))\rangle_{z,\sigma}\big\langle\partial_\zeta F_2g,\varrho_\delta((x,\xi)-\theta_{s,r}(\cdot,\cdot))\big\rangle_{y,\zeta}\,\dif W_r=0
\end{align*}
and similarly $J_3=0$.
$J_4$ can be rewritten as follows
\begin{align*}
J_4&=\stred\langle\!\langle\mathcal{I}_1(t),\mathcal{I}_2(t)\rangle\!\rangle\\
&=-\E\int_s^t\int g(z,\sigma)\cdot g(y,\zeta)\varrho_\delta((x,\xi)-\theta_{s,r}(z,\sigma))\varrho_\delta((x,\xi)-\theta_{s,r}(y\zeta))\dif\nu^1_{r,z}(\sigma)\dif\nu^2_{r,y}(\zeta)\dif z\dif y\dif r.
\end{align*}

Next, we apply the integration by parts formula for functions of bounded variation (see e.g. \cite[Chapter 0, Proposition 4.5]{revuz}) to $J_8$ and deduce
\begin{align*}
J_8&=\E\mu_1(\{s\})\mu_2(\{s\})+\E\int_{(s,t]}\mu_1([s,r])\,\dif\mu_2(r)+\E\int_{(s,t]}\mu_2([s,r))\,\dif \mu_1(r)\\
&=J_{80}+J_{81}+J_{82}.
\end{align*}
Since the kinetic measures $m_1$ and $m_2$ do not have atoms at $s$ a.s. due to the assumptions, it follows that $J_{80}=0.$ Regarding $J_{81}$, we make use of the formula for $\mu_1([s,r])$ given by \eqref{eq:stronger}. Namely
\begin{align*}
J_{81}&=\E\int_{(s,t]}\big\langle F_1^+(r),\varrho_\delta((x,\xi)-\theta_{s,r}(\cdot,\cdot))\big\rangle_{z,\sigma}\,\dif\mu_2(r)\\
&\quad-\E\big\langle F^-_{1}(s),\varrho_\delta((x,\xi)-\theta_s(\cdot,\cdot))\big\rangle_{z,\sigma}\mu_2([s,t])-\E\int_{(s,t]}\mathcal{I}_1(r)\,\dif\mu_2(r)=J_{811}-J_2-J_{5}.
\end{align*}
Going back to the product of \eqref{eq:stronger} and \eqref{eq:stronger1} we see that $J_{2}$ and $J_5$ cancel and for $J_{811}$ we write
\begin{align*}
J_{811}&=\E\int_{(s,t]}\int F_1^+(r)\varrho_\delta((x,\xi)-\theta_{s,r}(z,\sigma))\partial_\zeta \varrho_\delta((x,\xi)-\theta_{s,r}(y,\zeta))\dif m_2(r,y,\zeta)\dif z\dif \sigma\\
&\quad-\frac{1}{2}\E\int_{s}^t\int F_1(r)G^2(y,\zeta)\varrho_\delta((x,\xi)-\theta_{s,r}(z,\sigma))\partial_\zeta \varrho_\delta((x,\xi)-\theta_{s,r}(y,\zeta))\dif \nu^2_{r,y}(\zeta)\dif y\dif z\dif \sigma\dif r
\end{align*}

Let us now continue with $J_{82}$. We apply the formula for $\mu_2([s,r))$ which can be obtained from the formula for $F^-$ in Corollary \ref{cor:strongerversion} using Lemma \ref{lemma:neworigin}, cf. \eqref{eq:stronger1}. It yields
\begin{align*}
J_{82}&=\E\int_{(s,t]}\langle \overline{F}_2^-(r),\varrho_\delta((x,\xi)-\theta_{s,r}(\cdot,\cdot))\rangle_{y,\zeta}\,\dif\mu_1(r)\\
&\quad-\E\langle \overline{F}^-_{2}(s),\varrho_\delta((x,\xi)- (\cdot,\cdot))\rangle_{y,\zeta}\mu_1([s,t])-\E\int_{(s,t]}\mathcal{I}_2(r)\dif\mu_1(r)=J_{821}-J_6-J_7,
\end{align*}
where $J_{6}$ and $J_7$ cancel and for $J_{821}$ we have
\begin{align*}
J_{821}&=-\E\int_{(s,t]}\int \overline{F}_2^-(r)\varrho_\delta((x,\xi)-\theta_{s,r}(y,\zeta))\partial_\sigma \varrho_\delta((x,\xi)-\theta_{s,r}(z,\sigma))\dif m_1(r,z,\sigma)\dif y\dif \zeta\\
&\quad+\frac{1}{2}\E\int_{s}^t\int \overline F_2(r)G^2(z,\sigma) \varrho_\delta((x,\xi)-\theta_{s,r}(y,\zeta))\partial_\sigma\varrho_\delta((x,\xi)-\theta_{s,r}(z,\sigma))\dif \nu^1_{r,z}(\sigma)\dif z\dif y\dif \zeta\dif r.
\end{align*}

{\em Step 2:} Finally, we have all in hand to proceed with the proof.
Integrating \eqref{123} with respect to $x,\xi$ we obtain
\begin{equation*}
\begin{split}
\E\int&\, \kappa_R(x)F^+_1(t)\overline{F}^+_2(t)\varrho_\delta((x,\xi)-\theta_{s,t}(z,\sigma))\varrho_\delta((x,\xi)-\theta_{s,t}(y,\zeta))\dif x\dif \xi\dif y\dif\zeta\dif z\dif \sigma\\
&-\E \int \kappa_R(x)F^-_{1}(s)\overline{F}^-_{2}(s)\varrho_\delta((x,\xi)-(z,\sigma))\varrho_\delta((x,\xi)-(y,\zeta))\dif x\dif \xi\dif y\dif\zeta\dif z\dif \sigma=I_1+\cdots+I_5,
\end{split}
\end{equation*}
where
\begin{align*}
I_1&=\E\int_{(s,t]}\int \kappa_R(x)F_1^+(r)\varrho_\delta((x,\xi)-\theta_{s,r}(z,\sigma))\partial_\zeta\varrho_\delta((x,\xi)-\theta_{s,r}(y,\zeta))\dif m_2(r,y,\zeta),\\
I_2 &=-\E\int_{(s,t]}\int \kappa_R(x)\overline{F}^-_2(r) \varrho_\delta((x,\xi)-\theta_{s,r}(y,\zeta))\partial_\sigma\varrho_\delta((x,\xi)-\theta_{s,r}(z,\sigma))\dif m_1(r,z,\sigma)\\
I_3&=-\E\int_s^t\int\kappa_R(x) g(z,\sigma)\cdot g(y,\zeta)\varrho_\delta((x,\xi)-\theta_{s,r}(z,\sigma))\varrho_\delta((x,\xi)-\theta_{s,r}(y,\zeta))\dif\nu^1_{r,z}(\sigma)\dif\nu^2_{r,y}(\zeta),\\
I_4&=-\frac{1}{2}\E\int_{s}^t\int\kappa_R(x) F_1(r)G^2(y,\zeta)\varrho_\delta((x,\xi)-\theta_{s,r}(z,\sigma))\partial_\zeta \varrho_\delta((x,\xi)-\theta_{s,r}(y,\zeta))\dif \nu^2_{r,y}(\zeta),\\
I_5&=\frac{1}{2}\E\int_{s}^t\int \kappa_R(x)\overline F_2(r)G^2(z,\sigma) \varrho_\delta((x,\xi)-\theta_{s,r}(y,\zeta))\partial_\sigma\varrho_\delta((x,\xi)-\theta_{s,r}(z,\sigma))\dif \nu^1_{r,z}(\sigma).
\end{align*}
Above as well as in the sequel, when we integrate with respect to all variables $r,x,\xi,z,\sigma,y,\zeta$, we only specify the kinetic measures $m_1,\,m_2$ and the Young measures $\nu^1,\,\nu^2$, the Lebesgue measure being omitted.

Since the technique of estimation of $I_1$ and $I_2$ is similar, let us just focus on $I_1$. It holds that
\begin{align*}
I_1&=-\E\int_{(s,t]}\int\kappa_R(x)  \varrho_\delta((x,\xi)-\theta_{s,r}(z,\sigma))\varrho_\delta((x,\xi)-\theta_{s,r}(y,\zeta))\dd\nu^{1,+}_{r,z}(\sigma)\dif m_2(r,y,\zeta)\\
&\quad+\E\int_{(s,t]}\int \kappa_R(x)F_1^+ (r)  \partial_\sigma\varrho_\delta((x,\xi)-\theta_{s,r}(z,\sigma))\varrho_\delta((x,\xi)-\theta_{s,r}(y,\zeta))\dif m_2(r,y,\zeta)\\
&\quad+\E\int_{(s,t]}\int \kappa_R(x)F_1^+ (r) \varrho_\delta((x,\xi)-\theta_{s,r}(z,\sigma))\partial_\zeta\varrho_\delta((x,\xi)-\theta_{s,r}(y,\zeta))\dif m_2(r,y,\zeta)\\
&=I_{11}+I_{12}+I_{13},
\end{align*}
where $\nu^{1,+}$ is the Young measure corresponding to $F_1^+$ and the first equality holds true because the sum of the first two terms is equal to zero due to\footnote{By $\langle\cdot,\cdot\rangle_\sigma$ we denote the duality between the space of distributions on $\R_\sigma$ and $C^1_c(\R_\sigma)$.}
\begin{align}\label{ibp}
\begin{aligned}
\langle F_1^+(r,z,\cdot),\partial_\sigma\varrho_\delta((x,\xi)-\theta_{s,r}(z,\cdot))\rangle_\sigma&=-\langle\partial_\sigma F_1^+(r,z,\cdot),\varrho_\delta((x,\xi)-\theta_{s,r}(z,\cdot))\rangle_\sigma\\
&=\langle\nu^{1,+}_{r,z}(\cdot),\varrho_\delta((x,\xi)-\theta_{s,r}(z,\cdot))\rangle_\sigma, 
\end{aligned}
\end{align}
which holds true for a.e. $(\omega,z)$ and every $(r,x,\xi)$.
Consequently, $I_{11}\leq0$. We will show that $I_{12}+I_{13}$ is small if $t-s$ is small. Towards this end, we observe that
\begin{equation}\label{1}
\partial_\sigma\varrho_\delta((x,\xi)-\theta_{s,r}(z,\sigma))=-\nabla_{x,\xi}\varrho_\delta((x,\xi)-\theta_{s,r}(z,\sigma))\cdot\partial_\sigma\theta_{s,r}(z,\sigma)
\end{equation}
and
\begin{align}\label{2}
\begin{aligned}
&\int  \kappa_R(x)\nabla_{x,\xi}\varrho_\delta((x,\xi)-\theta_{s,r}(z,\sigma))\varrho_\delta((x,\xi)-\theta_{s,r}(y,\zeta))\dif x\dif\xi\\
&\quad+\int  \kappa_R(x)\varrho_\delta((x,\xi)-\theta_{s,r}(z,\sigma))\nabla_{x,\xi} \varrho_\delta((x,\xi)-\theta_{s,r}(y,\zeta))\dif x\dif\xi\\
&\qquad=-\int \nabla_{x,\xi} \kappa_R(x) \varrho_\delta((x,\xi)-\theta_{s,r}(z,\sigma)) \varrho_\delta((x,\xi)-\theta_{s,r}(y,\zeta))\dif x\dif\xi
\end{aligned}
\end{align}
it follows that (recall that $\theta^x$ was defined in \eqref{eq:fl2})
\begin{align}\label{eq:estim}
\begin{aligned}
I_{12}+I_{13}= \E\int_{(s,t]}\int& F_1^+ (r)\kappa_R(x)\big[\partial_\sigma\theta_{s,r}(z,\sigma)-\partial_\zeta\theta_{s,r}(y,\zeta)\big]\\
&\cdot\nabla_{x,\xi}\varrho_\delta((x,\xi)-\theta_{s,r}(z,\sigma))\varrho_\delta((x,\xi)-\theta_{s,r}(y,\zeta))\dif m_2(r,y,\zeta)\\
&\hspace{-1.5cm}+\E\int_{(s,t]}\int F_1^+(r)\nabla_{x} \kappa_R(x) \cdot\partial_\sigma\theta^x_{s,r}(z,\sigma)\\
&\times\varrho_\delta((x,\xi)-\theta_{s,r}(z,\sigma)) \varrho_\delta((x,\xi)-\theta_{s,r}(y,\zeta))\dif m_2(r,y,\zeta)\\
&\hspace{-1.9cm}=I_{121}+I_{131}
\end{aligned}
\end{align}
To proceed, we use the ideas from \cite{gess}. Namely, we observe that for every $r\in[s,t]$, $\partial_\xi\theta_{s,r}(\cdot)$ is Lipschitz continuous and its Lipschitz constant can be estimated by $C (t-s)^{1/p}$.
Indeed, according to Theorem \ref{thm:existence}
\begin{align*}
\sup_{x,\xi\in\mr^N\times\mr}|\totdif^2\theta_{s,r}(x,\xi)|=\sup_{x,\xi\in\mr^N\times\mr}|\totdif^2\theta_{s,r}(x,\xi)-\totdif^2 \theta_{s,s}(x,\xi)|\leq C (r-s)^{1/p}\leq C (t-s)^{1/p}.
\end{align*}
Besides, we only have to consider the case of
$$|(x,\xi)-\theta_{s,r}(z,\sigma)|<\delta,\qquad |(x,\xi)-\theta_{s,r}(y,\zeta)|<\delta$$
hence
$$|\theta_{s,r}(z,\sigma)-\theta_{s,r}(y,\zeta)|<2\delta$$
and consequently
\begin{align}\label{eq:4}
|(z,\sigma)-(y,\zeta)|&=\big|\pi_{s,r}\big(\theta_{s,r}(z,\sigma)\big)-\pi_{s,r}\big(\theta_{s,r}(y,\zeta)\big)\big|\leq C|\theta_{s,r}(z,\sigma)-\theta_{s,r}(y,\zeta)|<C\delta
\end{align}
which implies
\begin{equation}\label{eq:estim1}
\big|\partial_\sigma\theta_{s,r}(z,\sigma)-\partial_\zeta\theta_{s,r}(y,\zeta)\big|\leq C (t-s)^{1/p}\delta.
\end{equation}
Let us now consider the remaining terms in the integral with respect to $(x,\xi,z,\sigma)$ and use the fact that $\theta_{s,r}$ is volume preserving and $\delta|\nabla_{x,\xi}\varrho_\delta|(\cdot)\leq C\varrho_{2\delta}(\cdot)$. It yields
\begin{align}\label{eq:estim2}
\begin{aligned}
\int & \kappa_R(x)\big|\nabla_{x,\xi}\varrho_\delta((x,\xi)-\theta_{s,r}(z,\sigma))\big|\varrho_\delta((x,\xi)-\theta_{s,r}(y,\zeta))\dif x\dif\xi\dif z\dif\sigma\\
&\leq\int   \big|\nabla_{x,\xi}\varrho_\delta((x,\xi)-(z,\sigma))\big|\varrho_\delta((x,\xi)-\theta_{s,r}(y,\zeta))\dif x\dif\xi\dif z\dif\sigma\\
&\leq\frac{C}{\delta}\int \varrho_{2\delta}((x,\xi)-(z,\sigma))\varrho_\delta((x,\xi)-\theta_{s,r}(y,\zeta))\dif x\dif\xi\dif z\dif\sigma=\frac{C}{\delta}.
\end{aligned}
\end{align}
Plugging \eqref{eq:estim1} and \eqref{eq:estim2} into $I_{121}$ we conclude that
\begin{align*}
I_{121}\leq C (t-s)^{1/p}\E\int_{(s,t]}\int \dif m_2(r,y,\zeta).
\end{align*}
Since due to Theorem \ref{thm:existence}
\begin{equation}\label{e3}
\sup_{(x,\xi)\in\mr^N\times\mr}|\partial_\sigma\theta_{s,r}^x(x,\xi)|=\sup_{(x,\xi)\in\mr^N\times\mr}|\partial_\sigma\theta_{s,r}^x(x,\xi)-\partial_\sigma\theta_{s,s}^x(x,\xi)|\leq C(t-s)^{1/p},
\end{equation}
we use the fact that $|\nabla_x\kappa_R(x)|\leq R^{-1}$ and that $\theta_{s,r}$ is volume preserving to deduce
\begin{align*}
I_{131}&\leq CR^{-1}(t-s)^{1/p}\E\int_{(s,t]}\int\varrho_{\delta}((x,\xi)-\theta_{s,r}(z,\sigma))\varrho_{\delta}((x,\xi)-\theta_{s,r}(y,\zeta))\dif m_2(r,y,\zeta)\\
&\leq CR^{-1}(t-s)^{1/p}\E\int_{(s,t]}\int\dif m_2(r,y,\zeta)\leq C(t-s)^{1/p}\E\int_{(s,t]}\int\dif m_2(r,y,\zeta).
\end{align*}

Let us proceed with $I_4$ and $I_5$. Using the ideas from \eqref{1} and \eqref{2} we conclude
\begin{align*}
I_4&=\frac{1}{2}\E\int_s^t\int F_1(r)G^2(y,\zeta)\kappa_R(x)\big[\partial_\zeta\theta_{s,r}(y,\zeta)-\partial_\sigma\theta_{s,r}(z,\sigma)\big]\\
&\hspace{2cm}\cdot\nabla_{x,\xi}\varrho_\delta((x,\xi)-\theta_{s,r}(y,\zeta))\varrho_\delta((x,\xi)-\theta_{s,r}(z,\sigma))\dif\nu^2_{r,y}(\zeta)\\
&\quad-\frac{1}{2}\E\int_s^t\int F_1(r)G^2(y,\zeta)\nabla_x\kappa_R(x)\cdot\partial_\sigma\theta_{s,r}(z,\sigma)\\
&\hspace{2cm}\times\varrho_\delta((x,\xi)-\theta_{s,r}(z,\sigma))\varrho_\delta((x,\xi)-\theta_{s,r}(y,\zeta))\dif\nu^2_{r,y}(\zeta)\\
&\quad+\frac{1}{2}\E\int_s^t\int G^2(y,\zeta)\kappa_R(x)\varrho_\delta((x,\xi)-\theta_{s,r}(z,\sigma))\varrho_\delta((x,\xi)-\theta_{s,r}(y,\zeta))\dif\nu^1_{r,z}(\sigma)\dif\nu^2_{r,y}(\zeta)\\
&=I_{41}+I_{42}+I_{43}
\end{align*}
and similarly
\begin{align*}
I_5&=-\frac{1}{2}\E\int_s^t\int \overline F_2(r)G^2(z,\sigma)\kappa_R(x)\big[\partial_\sigma\theta_{s,r}(z,\sigma)-\partial_\zeta\theta_{s,r}(y,\zeta)\big]\\
&\hspace{2cm}\cdot\nabla_{x,\xi}\varrho_\delta((x,\xi)-\theta_{s,r}(z,\sigma))\varrho_\delta((x,\xi)-\theta_{s,r}(y,\zeta))\dif\nu^1_{r,z}(\sigma)\\
&\quad+\frac{1}{2}\E\int_s^t\int \overline F_2(r)G^2(z,\sigma)\nabla_x\kappa_R(x)\cdot\partial_\zeta\theta_{s,r}(y,\zeta)\\
&\hspace{2cm}\times\varrho_\delta((x,\xi)-\theta_{s,r}(z,\sigma))\varrho_\delta((x,\xi)-\theta_{s,r}(y,\zeta))\dif\nu^1_{r,z}(\sigma)\\
&\quad+\frac{1}{2}\E\int_s^t\int G^2(z,\sigma)\kappa_R(x)\varrho_\delta((x,\xi)-\theta_{s,r}(z,\sigma))\varrho_\delta((x,\xi)-\theta_{s,r}(y,\zeta))\dif\nu^1_{r,z}(\sigma)\dif\nu^2_{r,y}(\zeta)\\
&=I_{51}+I_{52}+I_{53}.
\end{align*}
Using a similar approach as for $I_{121}$ together with the fact that $G^2(y,\zeta)\leq C|\zeta|^2$ which follows from the assumptions on $g$, we deduce that
\begin{align*}
I_{41}+I_{51}&\leq C(t-s)^{1/p}\bigg(\stred\int_s^t\int G^2(y,\zeta)\dif\nu^2_{r,y}(\zeta)\dif y\dif r+\stred\int_s^t\int G^2(z,\sigma)\dif\nu_{r,z}^1(\sigma)\dif z\dif r\bigg)\\
&\leq C(t-s)^{1/p}\bigg(\stred\int_s^t\int |\zeta|^2\dif\nu^2_{r,y}(\zeta)\dif y\dif r+\stred\int_s^t\int |\sigma|^2\dif\nu_{r,z}^1(\sigma)\dif z\dif r\bigg)
\end{align*}
and similarly to $I_{131}$ we get
\begin{align*}
I_{42}+I_{52}&\leq  CR^{-1}(t-s)^{1/p}\E\int_{s}^t\int|\zeta|^2\dif \nu^2_{r,y}(\zeta)\dif y\dif r+ CR^{-1}(t-s)^{1/p}\E\int_{s}^t\int|\sigma|^2\dif \nu^1_{r,z}(\sigma)\dif z\dif r\\
&\leq  C(t-s)^{1/p}\E\int_{s}^t\int|\zeta|^2\dif \nu^2_{r,y}(\zeta)\dif y\dif r+ C(t-s)^{1/p}\E\int_{s}^t\int|\sigma|^2\dif \nu^1_{r,z}(\sigma)\dif z\dif r.
\end{align*}
Besides,
\begin{align*}
I_3&+I_{43}+I_{53}\\
&=\frac{1}{2}\E\int_s^t\int \kappa_R(x)\big|g(z,\sigma)-g(y,\zeta)\big|^2\varrho_\delta((x,\xi)-\theta_{s,r}(z,\sigma))\varrho_\delta((x,\xi)-\theta_{s,r}(y,\zeta))\dif\nu^1_{r,z}(\sigma)\dif\nu^2_{r,y}(\zeta)\\
&\leq C\delta^2\E\int_s^t\int_{Q_R} \varrho_\delta((x,\xi)-\theta_{s,r}(z,\sigma))\varrho_\delta((x,\xi)-\theta_{s,r}(y,\zeta))\dif\nu^1_{r,z}(\sigma)\dif\nu^2_{r,y}(\zeta)
\end{align*}
where $Q_R=B_{2R,x}\times\mr_\xi\times B_{2R+CT^{1/p},z}\times\mr_\sigma\times B_{2R+CT^{1/p},y}\times\mr_\zeta$ and we used \eqref{eq:4} and the Lipschitz continuity of $g$. Next, we estimate using \eqref{ibp}, the fact that $\delta|\nabla_{x,\xi}\varrho_\delta|(\cdot)\leq C\varrho_{2\delta}(\cdot)$, \eqref{e3} and the fact that $\theta$ is volume preserving, as follows
\begin{align*}
&\int_{B_{2R+CT^{1/p},z}\times\R_\sigma}\varrho_\delta((x,\xi)-\theta_{s,r}(z,\sigma))\dd\nu^1_{r,z}(\sigma)\dd z\leq \int_{\mr^N\times\R}F^{-}_1(r,z,\sigma)\partial_\sigma\varrho_\delta((x,\xi)-\theta_{s,r}(z,\sigma))\dd z\dd\sigma\\
&\qquad=-\int_{\mr^N\times\R}F^{-}_1(r,z,\sigma)\nabla_{x,\xi}\varrho_\delta((x,\xi)-\theta_{s,r}(z,\sigma))\partial_\sigma\theta_{s,r}(z,\sigma)\dd z\dd\sigma\\
&\qquad \leq C\delta^{-1}(t-s)^{1/p}\int_{\R^N\times\R}\varrho_{2\delta}((x,\xi)-\theta_{s,r}(z,\sigma))\dd z\dd\sigma=C\delta^{-1}(t-s)^{1/p}
\end{align*}
which holds true for a.e. $(\omega,r)$ and every $(x,\xi)$.
Hence
\begin{align*}
I_3+I_{43}+I_{53}&\leq C\delta(t-s)^{1/p}\E\int_s^t\int_{B_{2R+CT^{1/p},y}\times\mr_\zeta}\int_{B_{2R,x}\times\mr_\xi}\varrho_\delta((x,\xi)-\theta_{s,r}(y,\zeta))\dd x\dd\xi\dd\nu^2_{r,y}(\zeta)\dd y\\
&\leq C\delta(t-s)^{1/p}\E\int_s^t\int_{B_{2R+CT^{1/p},y}\times\mr_\zeta}\dd\nu^2_{r,y}(\zeta)\dd y\leq CR\delta(t-s)^{1+1/p}
\end{align*}
which completes the proof of \eqref{eq:doubling}.
\end{proof}
\end{prop}

Finally, we have all in hand to prove the Reduction Theorem and the $L^1$-contraction property.

\begin{thm}[Reduction Theorem]\label{thm:reduction}
Let $u_0\in L^1\cap L^2(\Omega\times\R^N)$ and let $F$ be a generalized kinetic solution to \eqref{eq} with initial datum $F_0=\ind_{u_0>\xi}$. Then it is a kinetic solution to \eqref{eq}, that is, there exists $$u\in  L^2(\Omega,L^2(0,T;L^2(\mr^N)))$$
satisfying \eqref{fd} such that
$F(t,x,\xi)=\ind_{u(t,x)>\xi}$ a.e.

\begin{proof}
Let $t\in[0,T)$, $h>0$ and $\{0=t_0<t_1<\cdots<t_n=t\}$ be a partition of $[0,t]$ such that $|t_i-t_{i+1}|\le h$, $i=0,\dots,n-1$. According to Corollary \ref{cor:strongerversion}, we can assume (without loss of generality) that the kinetic measure $m$ has a.s. no atom at $t_i$, $i=1,\dots,n-1,$ that is, we only consider points from the dense set $\mathcal{A}\subset[0,T]$ which was constructed in Proposition \ref{limits}. Besides, due to Lemma \ref{lem:equil}, $m$ has no atom at $0$ a.s. Furthermore, if $\nu^+$ denotes the Young measure corresponding to $F^+$ regarded as a kinetic function on $\Omega\times[0,T]\times\R^N$, then it follows from \eqref{integrov} that
$$\E\int_0^T\int_{\R^N}\int_\R|\xi|^2\dd\nu^+_{t,x}(\xi)\dd x\dd t<\infty.$$ 
Consequently we may assume (without loss of generality) that for all $i\in\{1,\dots,n-1\}$
\begin{equation}\label{ref}
\E\int_{\R^N}\int_\R|\xi|^2\dd\nu^+_{t_i,x}(\xi)\dd x<\infty
\end{equation}
and since $u_0\in L^2(\Omega\times\R^N)$ and $\nu_{t_0,x}^+=\delta_{u_0(x)}$ due to Lemma \ref{lem:equil}, the same holds true also for $i=0$.

Application of Proposition \ref{prop:doubling} yields
\begin{align}\label{eqq:123}
\begin{aligned}
& \E\int\,\kappa_R(x)\langle F^+(t_i),\varrho_\delta((x,\xi)-\theta_{t_{i-1},t_i}(\cdot,\cdot))\rangle_{z,\sigma}\langle \overline{F}^+(t_{i}),\varrho_\delta((x,\xi)-\theta_{t_{i-1},t_i}(\cdot,\cdot))\rangle_{y,\zeta}\dif x\dif\xi\\
&\quad-\E\int\kappa_R(x)\langle F^+(t_{i-1}),\varrho_\delta((x,\xi)-\theta_{t_{i-1},t_{i-1}}(\cdot,\cdot))\rangle_{z,\sigma}\langle \overline{F}^+(t_{i-1}),\varrho_\delta((x,\xi)-\theta_{t_{i-1},t_{i-1}}(\cdot,\cdot))\rangle_{y,\zeta}\dif x\dif\xi\\
&\leq Ch^{1/p}\stred\int_{(t_{i-1},t_i]}\int\dif m(r,y,\zeta)+Ch^{1/p}\stred\int_{t_{i-1}}^{t_i}\int |\zeta|^2\dif\nu_{r,y}(\zeta)\dif y\dif r+CR\delta h^{1+1/p}
\end{aligned}
\end{align}
Now, we send $\delta\rightarrow0$ first and then $R\to\infty$.


Regarding the first term on the left hand side of \eqref{eqq:123}, since the flow is volume preserving we have
\begin{align*}
\langle F^+(t_i),\varrho_\delta((x,\xi)-\theta_{t_{i-1},t_i}(\cdot,\cdot))\rangle_{z,\sigma}=\langle F^+(t_i,\pi_{t_{i-1},t_i}(\cdot,\cdot)),\varrho_\delta((x,\xi)-(\cdot,\cdot))\rangle_{z,\sigma}
\end{align*}
which converges to $F^+(t_i,\pi_{t_{i-1},t_i}(x,\xi))$ for a.e. $(x,\xi)$ as $\delta\to0$ and similarly for the term including $\overline F^+(t_i)$.
Thus, if $i=n$ we apply Fatou's Lemma to estimate the $\liminf_{\delta\to 0}$ of the first term on the left hand side from below by
\begin{align*}
& \E\int\,\kappa_R(x)F^+(t_n,\pi_{t_{n-1},t_n}(x,\xi))\overline F^+(t_n,\pi_{t_{n-1},t_n}(x,\xi))\dif x\dif\xi\\
&\quad=\E\int\,\kappa_R(\theta^x_{t_{n-1},t_n}(x,\xi))F^+(t_n,x,\xi)\overline F^+(t_n,x,\xi)\dif x\dif\xi.
\end{align*}
If $i\in\{0,\dots,n-1\}$ we need to pass to the limit. To this end, we note that according to \eqref{ref} and Remark \ref{chi} the functions $\chi_{F^+(t_i)}$ are integrable in $\omega,\xi$ and locally integrable in $x$. Therefore, we estimate
\begin{align*}
&\langle F^+(t_i),\varrho_\delta((x,\xi)-\theta_{t_{i-1},t_i}(\cdot,\cdot))\rangle_{z,\sigma}\langle \overline{F}^+(t_{i}),\varrho_\delta((x,\xi)-\theta_{t_{i-1},t_i}(\cdot,\cdot))\rangle_{y,\zeta}\\
&\quad\quad\quad\leq\begin{cases}
		\langle F^+(t_i),\varrho_\delta((x,\xi)-\theta_{t_{i-1},t_i}(\cdot,\cdot))\rangle_{z,\sigma},&\text{ if }\xi\geq 0,\\
		\langle \overline{F}^+(t_{i}),\varrho_\delta((x,\xi)-\theta_{t_{i-1},t_i}(\cdot,\cdot))\rangle_{y,\zeta},&\text{ if }\xi<0,
		\end{cases}
\end{align*}
and observe that the left hand side converges a.e. in $(\omega,x,\xi)$. Indeed, both functions $F^+(t_i)$ and $\overline{F}^+(t_i)$ are locally integrable in $(\omega,x,\xi)$ so their mollifications converge a.e. Next, we apply Remark \ref{chi} to deduce that $F^+(t_i)$ is integrable on $\Omega\times B_{2R+CT^{1/p},x}\times[0,\infty)$ whereas $\overline F^+(t_i)$ is integrable on $\Omega\times B_{2R+CT^{1/p},x}\times(-\infty,0).$ As a consequence,
\begin{align*}
\langle F^+(t_i),\varrho_\delta((x,\xi)-\theta_{t_{i-1},t_i}(\cdot,\cdot))\rangle_{z,\sigma}\longrightarrow F^+(t_i,\pi_{t_{i-1},t_i}(x,\xi))\quad\text{in}\quad L^1(\Omega\times B_{2R}\times [0,\infty))
\end{align*}
and similarly
\begin{align*}
\langle \overline F^+(t_i),\varrho_\delta((x,\xi)-\theta_{t_{i-1},t_i}(\cdot,\cdot))\rangle_{y,\zeta}\longrightarrow\overline F^+(t_i,\pi_{t_{i-1},t_i}(x,\xi))\quad\text{in}\quad L^1(\Omega\times B_{2R}\times (-\infty,0))
\end{align*}
and for subsequences we obtain convergence a.e. in $(\omega,x,\xi)$.
Therefore, we may apply the generalization of the Vitali convergence theorem \cite[Corollaire 4.14]{kavian} to deduce that
\begin{align*}
&\langle F^+(t_i),\varrho_\delta((x,\xi)-\theta_{t_{i-1},t_i}(\cdot,\cdot))\rangle_{z,\sigma}\langle \overline{F}^+(t_{i}),\varrho_\delta((x,\xi)-\theta_{t_{i-1},t_i}(\cdot,\cdot))\rangle_{y,\zeta}\\
&\qquad\qquad\longrightarrow F^+(t_i,\pi_{t_{i-1},t_i}(x,\xi))\overline F^+(t_i,\pi_{t_{i-1},t_i}(x,\xi))\quad\text{in}\quad L^1(\Omega\times B_{2R}\times \R)
\end{align*}
and accordingly to pass to the limit as $\delta\to 0$ in the first and second term on the left hand side of \eqref{123} to obtain terms of the form
\begin{align*}
& \E\int\,\kappa_R(x)F^+(t_i,\pi_{t_{i-1},t_i}(x,\xi))\overline F^+(t_i,\pi_{t_{i-1},t_i}(x,\xi))\dif x\dif\xi\\
&\quad=\E\int\,\kappa_R(\theta^x_{t_{i-1},t_i}(x,\xi))F^+(t_i,x,\xi)\overline F^+(t_i,x,\xi)\dif x\dif\xi.
\end{align*}

In order to pass to the limit as $R\to\infty$ we intend to apply the dominated convergence theorem. To this end, it is necessary to justify that for all $i=0,\dots,n,$
$$F^+(t_i)\overline F^+(t_i)\in L^1(\Omega\times\mr^N\times\R),$$
which can be proved inductively: if $i=0$ then $F^+(0)\overline{F}^+(0)=\ind_{u_0>\xi}(1-\ind_{u_0>\xi})=0$ hence the claim holds true. If it is true for some $i-1$ where $i\in\{1,\dots,n\}$ then the Fatou lemma together with the discussion above leads to
\begin{align*}
\E\int F^+(t_i,x,\xi)\overline{F}^+(t_i,x,\xi) \dif x\dif\xi&\leq \lim_{R\to\infty}\E\int \kappa_R (\theta^x_{t_{i-1},t_i}(x,\xi))F^+(t_{i-1},x,\xi)\overline{F}^+(t_{i-1},x,\xi) \dif x\dif\xi\\
&\quad+ Ch^{1/p}\stred\int_{(t_{i-1},t_i]}\int\dif m(r,y,\zeta)+Ch^{1/p}\stred\int_{t_{i-1}}^{t_i}\int |\zeta|^2\dif\nu_{r,y}(\zeta)\dif y\dif r\\
&=\E\int F^+(t_{i-1},x,\xi)\overline{F}^+(t_{i-1},x,\xi) \dif x\dif\xi\\
&\quad+ Ch^{1/p}\stred\int_{(t_{i-1},t_i]}\int\dif m(r,y,\zeta)+Ch^{1/p}\stred\int_{t_{i-1}}^{t_i}\int |\zeta|^2\dif\nu_{r,y}(\zeta)\dif y\dif r.
\end{align*}
Thus Definition \ref{def:kinmeasure}(ii) and \eqref{integrov} give the claim for $i$.
\color{black}

Altogether, sending $\delta\to0$ and $R\to\infty$ gives for all $i\in\{1,\dots,n\}$
\begin{align*}
&\E\int F^+(t_i,x,\xi)\overline{F}^+(t_i,x,\xi) \dif x\dif\xi-\E\int F^+(t_{i-1},x,\xi)\overline{F}^+(t_{i-1},x,\xi) \dif x\dif\xi\\
&\quad\leq Ch^{1/p}\stred\int_{(t_{i-1},t_i]}\int\dif m(r,y,\zeta)+Ch^{1/p}\stred\int_{t_{i-1}}^{t_i}\int |\zeta|^2\dif\nu_{r,y}(\zeta)\dif y\dif r.
\end{align*}
Consequently,
\begin{align*}
&\E\int F^+(t,x,\xi)\overline{F}^+(t,x,\xi) \dif x\dif\xi-\E\int F^+_0(x,\xi)\overline{F}^+_0(x,\xi) \dif x\dif\xi\\
&=\sum_{i=1}^n\E\int F^+(t_i,x,\xi)\overline{F}^+(t_i,x,\xi) \dif x\dif\xi-\E\int F^+(t_{i-1},x,\xi)\overline{F}^+(t_{i-1},x,\xi) \dif x\dif\xi\\
&\leq Ch^{1/p}\stred\int_{(0,t]}\int\dif m(r,y,\zeta)+Ch^{1/p}\stred\int_{0}^{t}\int |\zeta|^2\dif\nu_{r,y}(\zeta)\dif y\dif r
\end{align*}
and sending $h\rightarrow0$ yields (using Definition \ref{def:kinmeasure}(ii) and \eqref{integrov})
\begin{align*}
\E\int& F^+(t,x,\xi)\overline{F}^+(t,x,\xi)\, \dif x\,\dif\xi\leq \E\int F_{0}(x,\xi)\overline{F}_{0}(x,\xi)\,\dif x\,\dif \xi=\E\int \ind_{u_0(x)>\xi}(1-\ind_{u_0(x)>\xi})\,\dif x\,\dif\xi=0.
\end{align*}
Hence $F^+(t)(1-F^+(t))=0$ for a.e. $(\omega,x,\xi)$. Now, the fact that $F^+(t)$ is a kinetic function for all $t\in[0,T)$ gives the
conclusion: indeed, by Fubini Theorem, for any $t \in [0, T )$, there is a set $E_t$ of
full measure in $\Omega\times\R^N$ such that, for all $(\omega,x)\in E_t$, $F^+(t,x,\xi)\in\{0,1\} $ for
a.e. $\xi\in\R$. Recall that $-\partial_\xi F^+(\omega,t,x, \cdot)$ is a probability measure on $\R$ hence,
necessarily, there exists $u:\Omega\times [0,T)\times\mr^N\rightarrow\mr$ measurable such that $F^+(\omega,t,x,\xi)=\ind_{u(\omega,t,x)>\xi}$ for a.e. $(\omega,x,\xi)$ and all $t\in[0,T)$. Moreover, according to \eqref{integrov}, it holds
$$\E\esssup_{0\leq t\leq T}\int_{\mr^N}|u(t,x)|\,\dif x=\E\esssup_{0\leq t\leq T}\int_{\mr^N}\int_\mr|\xi|\,\dif\nu^+_{t,x}(\xi)\,\dif x\leq C$$
and
$$\E\int_0^T\int_{\mr^N}|u(t,x)|^2\,\dif x\dif t=\E\int_0^T\int_{\mr^N}\int_\mr|\xi|^2\,\dif\nu^+_{t,x}(\xi)\,\dif x\dif t\leq C$$
hence $u$ is a kinetic solution.
%
%
%
%
\end{proof}
\end{thm}

\begin{cor}[$L^1$-contraction property]\label{cor:contraction}
Let $u_1$ and $u_2$ be kinetic solutions to \eqref{eq} with initial data $u_{1,0}$ and $u_{2,0}$, respectively. Then there exist representatives $\tilde{u}_1$ and $\tilde{u}_2$ of the class of equivalence $u_1$ and $u_2$, respectively, such that for all $t\in[0,T)$
\begin{equation*}
\E\big\|\big(\tilde{u}_1(t)-\tilde{u}_2(t)\big)^+\big\|_{L^1_x}\leq \E\|(u_{1,0}-u_{2,0})^+\|_{L^1_x}.
\end{equation*}

\begin{proof}
First, we observe that the following identity holds true
$$(u_1-u_2)^+=\int_\mr\ind_{u_1>\xi}\overline{\ind_{u_2>\xi}}\,\dif \xi.$$
Let $\tilde{u}_1$ and $\tilde{u}_2$ denote the representatives constructed in Theorem \ref{thm:reduction}. Then proceeding similarly to Theorem \ref{thm:reduction}, we apply Proposition \ref{prop:doubling} and obtain
\begin{align*}
\E\big\|\big(\tilde u_1(t) -\tilde u_2(t)\big)^+\big\|_{L^1_x}&=\E\int F^+_1(t,x,\xi)\overline{F}^+_2(t,x,\xi)\,\dif \xi\dif x\\
&\leq\E\int \ind_{u_{1,0}(x)>\xi}\overline{\ind_{u_{2,0}(x)>\xi}}\,\dif \xi\dif x=\E\|(u_{1,0}-u_{2,0})^+\|_{L^1_x}
\end{align*}
which completes the proof.
\end{proof}
\end{cor}

\section{Existence}
\label{sec:existence}

In the existence part of the proof of Theorem \ref{thm:main} we make use of the so-called Bhatnagar-Gross-Krook approximation (BGK), which is a standard tool in the deterministic setting. Its stochastic counterpart was established in \cite{bgk} and could be understood as an alternative proof of existence to \cite{debus}.

Let us now briefly describe this method. Towards the proof of existence of a kinetic solution to \eqref{eq}, which is (formally speaking) a distributional solution to the kinetic formulation \eqref{eq:kinform}, we consider the following BGK model
\begin{equation}\label{eq:bgk}
 \begin{split}
\dif F^\varepsilon+ \nabla F^\varepsilon\cdot a\,\dif z-\partial_\xi F^\varepsilon\,b\,\dif z&=-\partial_\xi F^\varepsilon\, g\,\dif W+\frac{1}{2}\partial_\xi(G^2\partial_\xi F^\varepsilon)\,\dif t+\frac{\ind_{u^\varepsilon>\xi}-F^\varepsilon}{\varepsilon}\dif t,\\
F^\varepsilon(0)&=F^\varepsilon_0,
 \end{split}
\end{equation}
where the local density of particles is defined by
$$u^\varepsilon(t,x)=\int_\mr \big(F^\varepsilon(t,x,\xi)-\ind_{0>\xi}\big)\dif\xi.$$
In other words, the unknown kinetic measure is replaced by a right hand side written in terms of $F^\varepsilon$.
Heuristically, solving \eqref{eq:bgk} is significantly easier and can be reduced to solving the homogeneous problem
\begin{equation}\label{eq:aux1}
\begin{split}
\dif X+\nabla X\cdot a\,\dif z-\partial_\xi X\,b\,\dif z&=-\partial_\xi X\, g\,\dif W+\frac{1}{2}\partial_\xi(G^2\partial_\xi X)\,\dif t,\\
X(s)&=X_0,
\end{split}
\end{equation}
establishing the properties of its solution operator and employing Duhamel's principle.
The final idea is that, as the microscopic scale $\varepsilon$ vanishes, $F^\varepsilon$ converges to $F$ which is a generalized kinetic solution to \eqref{eq}. The Reduction theorem, Theorem \ref{thm:reduction}, then applies and as a consequence there exists $u$ such that $F=\ind_{u>\xi}$ and $u$ is the unique kinetic solution to \eqref{eq}.

We point out that the expected regularity in $(x,\xi)$ of solutions to \eqref{eq:aux}, \eqref{eq:bgk} is low, namely $L^\infty$, and therefore the rigorous treatment of the above outlined technique requires a suitable notion of weak solution in the context of rough paths. Here the usual definition of distributional solutions encounters several obstacles due to rough regularity of the driving signals. Motivated by the discussion in Section \ref{sec:hypotheses}, we introduce weak formulations which correspond to solving \eqref{eq:aux} and \eqref{eq:bgk} in the sense of distributions but do not involve any rough path driven terms.

In the sequel, we will use the following notation: for $\alpha\in\mr$ we denote by $\chi_{\alpha}:\mr\rightarrow\mr$ the so-called equilibrium function which is defined as
\begin{equation}\label{def:equil}
\chi_{\alpha}(\xi)=\ind_{0<\xi<\alpha}-\ind_{\alpha<\xi<0}.
\end{equation}

\subsection{Rough transport equation}
\label{sec:roughtr}

This subsection is devoted to the study of the auxiliary equation \eqref{eq:aux1}. Our approach here as well as in Subsection \ref{sec:bgksol} is rough-pathwise and therefore we work with one fixed realization $\omega$ and the probability space $(\Omega,\mf,\prst)$ remains hidden. The stochastic integral will then re-appear in Section \ref{sec:conv} for the final passage to the limit.

It can be seen that the Stratonovich form of \eqref{eq:aux1} reads as follows
\begin{equation*}
\begin{split}
\dif X+\nabla X\cdot a\,\dif z-\partial_\xi X\,b\,\dif z&=-\partial_\xi X\, g\circ\dif W+\frac{1}{4}\partial_\xi X\partial_\xi G^2\,\dif t.
\end{split}
\end{equation*}
Now we recall that (the stochastic process) $\mathbf{\Lambda}$ is the joint lift of $z$ and $W$ constructed in \eqref{jointlift}. Let us fix one realization $\mathbf{\Lambda}(\omega)$.
This leads us to the study of the following rough transport equation
\begin{equation*}\label{eq:tr}
\begin{split}
\dif X&=\left(\begin{array}{c}
			\partial_\xi X\\
			\nabla X
			\end{array}\right)\cdot\left(\begin{array}{ccc}
									\frac{1}{4}\partial_\xi G^2&b& -g\\
									0&-a&0
									\end{array}\right)\,\dif (t,\mathbf{\Lambda}(\omega)),
\end{split}
\end{equation*}
However, for notational simplicity (and with a slight abuse of notation) we may rather write
\begin{equation}\label{eq:aux}
\begin{split}
\dif X+\nabla X\cdot a\,\dif \mathbf{z}-\partial_\xi X\,b\,\dif \mathbf{z}&=-\partial_\xi X\, g\,\dif \mathbf{w}+\frac{1}{4}\partial_\xi X\partial_\xi G^2\,\dif t,\\
X(s)&=X_0,
\end{split}
\end{equation}
where $\mathbf{w}$ denotes the corresponding realization of the Stratonovich lift of $W$ and the cross-iterated integrals between $t$, $z$ and $W$ are not explicitly mentioned.

We introduce two notions of solution to \eqref{eq:aux}. The first definition follows the usual approach from rough path theory which is based on the approximation of the driving signals (see \cite[Definition 3]{caruana}).

\begin{defin}\label{def:strongsol}
Let $(\mathbf{z},\mathbf{w})$ be a geometric H\"older $p$-rough path and suppose that $(z^n,w^n)$ is a sequence of Lipschitz paths such that
$$S_{[p]}(z^n,w^n)\equiv(\mathbf{z}^n,\mathbf{w}^n)\longrightarrow(\mathbf{z},\mathbf{w})$$
uniformly on $[0,T]$ and
$$\sup_n\|(\mathbf{z}^n,\mathbf{w}^n)\|_{\frac{1}{p}\text{-H\"ol}}<\infty.$$
Assume that for each $n\in\mn$
\begin{equation*}
\begin{split}
\dif X^n+\nabla X^n\cdot a\,\dif z^n-\partial_\xi X^n\,b\,\dif z^n&=-\partial_\xi X^n\, g\,\dif w^n+\frac{1}{4}\partial_\xi X^n\partial_\xi G^2\,\dif t\\
X^n(s)&=X_0\in C_b^1(\mr^N\times\mr)
\end{split}
\end{equation*}
has a unique solution $X^n$ which belongs to $C^{1}_b([0,T]\times\mr^N\times\mr)$. Then any limit point (in the uniform topology) of $(X^n)$ is called a solution for the rough PDE, denoted formally by \eqref{eq:aux}.
\end{defin}

Such a solution can be constructed by making use of the method of characteristics established in \cite{caruana} provided the initial datum $X_0$ is sufficiently regular. To be more precise, the associated characteristic system is given by
\begin{equation}\label{eq:char}
\begin{split}
\dif\varphi^0_t&= -b(\varphi_t)\,\dif \mathbf{z}+g(\varphi_t)\,\dif \mathbf{w}-\frac{1}{4}\partial_\xi G^2(\varphi_t)\,\dif t,\\
\dif\varphi^x_t&=a(\varphi_t)\,\dif \mathbf{z}.
\end{split}
\end{equation}
where $\varphi^0_t$ and $\varphi^x_t$ describe the evolution of the $\xi$-coordinate and $x$-coordinate, respectively, of the characteristic curve.
 Let us denote by $\varphi_{s,t}(x,\xi)$ the solution of \eqref{eq:char} starting from $(x,\xi)$ at time $s$. It follows from \cite[Proposition 11.11]{friz} that under our assumptions $\varphi$ defines a flow of $C^2$-diffeomorphisms and we denote by $\psi$ the corresponding inverse flow.

Our first existence and uniqueness result for \eqref{eq:aux} is taken from \cite[Theorem 4]{caruana} and reads as follows.

\begin{prop}\label{prop:aux}
Let $X_0\in C_b^1(\mr^N\times\mr).$ Then there exists a unique solution to \eqref{eq:aux} given explicitly by
$$X(t,x,\xi;s)=X_0\big(\psi_{s,t}(x,\xi)\big).$$
\end{prop}

We conclude that the solution operator 
$\mathcal{S}(t,s)X_0=X_0\big(\psi_{s,t}(x,\xi)\big)$
is well defined on $C^1_b(\mr^N\times\mr)$.
Nevertheless, as the right hand side makes sense even for more general initial conditions $X_0$ which do not necessarily fulfill the assumptions of Proposition \ref{prop:aux}, the domain of definition of the operator $\mathcal{S}$ can be extended. In particular, since diffeomorphisms preserve sets of measure zero the above is well defined also if $X_0$ is only defined almost everywhere. In this case, we define consistently
$$\mathcal{S}(t,s)X_0=X_0\big(\psi_{s,t}(x,\xi)\big),\qquad 0\leq s\leq t\leq T.$$
The aim is to show that this extension is the solution operator to \eqref{eq:aux} in a weak sense. Towards this end, it is necessary to weaken the assumption upon the initial condition and therefore a suitable notion of weak solution is required.

In the following proposition we establish basic properties of the ope\-rator $\mathcal{S}.$

\begin{prop}\label{oper}
Let $\mathcal{S}=\{\mathcal{S}(t,s),0\leq s\leq t\leq T\}$ be defined as above. Then
\begin{enumerate}
 \item[\emph{(i)}]\label{item1} for any $p\in[1,\infty]$, $\mathcal{S}$ is a family of linear operators on $L^p(\mr^N\times\mr)$ which is uniformly bounded in the operator norm, i.e. there exists $C>0$ such that for any $X_0\in L^p(\mr^N\times\mr)$, $0\leq s\leq t\leq T$,
\begin{equation}\label{first} 
 \big\|\mathcal{S}(t,s)X_0\big\|_{L^p_{x,\xi}}\leq C\|X_0\|_{L^p_{x,\xi}},
\end{equation}
\item[\emph{(ii)}]\label{item4} $\mathcal{S}$ verifies the semigroup law
\begin{equation*}
\begin{split}
\mathcal{S}(t,s)&=\mathcal{S}(t,r)\circ\mathcal{S}(r,s),\qquad 0\leq s\leq r\leq t\leq T,\\
\mathcal{S}(s,s)&=\mathrm{Id},\hspace{3.3cm} 0\leq s\leq T.
\end{split}
\end{equation*}
\end{enumerate}

\begin{proof}
The proof of (ii) as well as (i) in the case of $p=\infty$ follows immediately from the definition of the operator $\mathcal{S}$ and the flow property of $\psi$. In order to show (i) for $p\in[1,\infty)$, we observe that due to \cite[Proposition 11.11]{friz} the map
$$(s,t,x,\xi)\longmapsto |\mathrm{J}\psi_{s,t}(x,\xi)|,$$
where $\mathrm{J}$ denotes the Jacobian, is bounded from above and below by a positive constant. Consequently, using the notation of Proposition \ref{prop:aux}, we have
\begin{equation*}
\begin{split}
\|X_0\|_{L^p_{x,\xi}}^p&=\int_{\mr^N\times\mr}\big|X\big(t,\varphi_{s,t}(x,\xi);s\big)\big|^p\,\dif x\,\dif\xi\\
&=\int_{\mr^N\times\mr}|X(t,x,\xi;s)|^p\,|\mathrm{J}\psi_{s,t}(x,\xi)|\,\dif x\,\dif \xi\geq C\|X\|^p_{L^p_{x,\xi}}.
\end{split}
\end{equation*}
\end{proof}
\end{prop}

\begin{defin}\label{def:weaksol}
$X:[s,T]\times\mr^N\times\mr\rightarrow\mr$ is called a weak solution to \eqref{eq:aux} provided
\begin{equation}\label{eq:weaktransport}
\begin{split}
\dif X\big(t,\varphi_{s,t}\big)&=0\\
X(s)&=X_0
\end{split}
\end{equation}
holds true in the sense of $\mathcal{D}'(\mr^N\times\mr)$, i.e. for all $\phi\in C^1_c(\R^N\times\R)$ and $t\in[s,T]$
$$\big\langle X\big(t,\varphi_{s,t}\big),\phi\big\rangle=\langle X_0,\phi\rangle.$$
\end{defin}


\begin{cor}\label{cor:weaksol}
Let $X_0\in L^\infty(\mr^N\times\mr)$. Then there exists a unique $X\in L^\infty([s,T]\times\mr^N\times\mr)$ that is a weak solution to \eqref{eq:aux}. Moreover, it is represented by
$$X(t)=\mathcal{S}(t,s)X_0.$$

\begin{proof}
In order to prove existence, observe that if $X$ is given by the above representation then clearly for all $0\leq s\leq t\leq T$
$$X\big(t,\varphi_{s,t}(x,\xi);s\big)=X_0\big(\psi_{s,t}\circ\varphi_{s,t}(x,\xi)\big)=X_0(x,\xi)\qquad \text{ for a.e. }(x,\xi)$$
and hence $X$ solves \eqref{eq:weaktransport}. However, let us also give another proof of existence which in addition justifies Definition \ref{def:weaksol} as the appropriate notion of weak solution to \eqref{eq:aux}.
To this end, let us consider smooth approximations of $X_0$, namely, let $(\varrho_\delta)$ be an approximation to the identity on $\mr^N\times\mr$ and set $X_0^\delta=X_0*\varrho_\delta$. According to Proposition \ref{prop:aux}, there exists a unique $X^\delta$ which solves \eqref{eq:aux} with the initial condition $X_0^\delta$ and is given by $X^\delta(t;s)=X^\delta_0(\psi_{s,t})$ hence
\begin{equation}\label{eq:delta}
X^\delta\big(t,\varphi_{s,t}(x,\xi);s\big)=X_0^\delta(x,\xi).
\end{equation}
Moreover,
$$X_0^\delta\overset{w^*}{\longrightarrow} X_0\quad\text{ in }\quad L^\infty(\mr^N\times\mr)$$
so
$$X^\delta(t;s)\overset{w^*}{\longrightarrow} X_0(\psi_{s,t})\quad\text{ in }\quad L^\infty(\mr^N\times\mr)\quad\forall t\in[s, T]$$
which justifies the passage to the limit in \eqref{eq:delta} and completes the proof.

Regarding uniqueness, due to linearity it is enough to show that any weak solution with $X_0=0$ vanishes identically. Let $X$ be such a weak solution, i.e. it holds
$$\langle X(t,\varphi_{s,t};s),\phi\rangle=0\qquad \forall\phi\in C^1_c(\mr^N\times\mr).$$
Testing by the mollifier $\varrho_\delta((x,\xi)-(\cdot,\cdot))$ we deduce that for all $(x,\xi)$
\begin{align*}
0&=\int_{\mr^N\times\mr}X\big(t,\varphi_{s,t}(y,\zeta);s\big)\,\varrho_\delta\big((x,\xi)-(y,\zeta)\big)\,\dif y\,\dif \zeta
\end{align*}
and since for a.e. $(x,\xi)$ the right hand side converges to $X\big(t,\varphi_{s,t}(x,\xi);s\big)$, the claim follows.
\end{proof}
\end{cor}


\begin{lemma}\label{indik}
For all $0\leq s\leq t\leq T$ it holds true that
$$\mathcal{S}(t,s)\ind_{0>\xi}-\ind_{0>\xi}=0.$$

\begin{proof}
It follows from \eqref{eq:null} that for all $x\in\mr^N$ the solution to \eqref{eq:char} starting from $(x,0)$ satisfies $\varphi^{0}_{s,t}(x,0)\equiv 0$. Moreover, since the solution to \eqref{eq:char} is unique, we deduce that
$$\varphi^{0}_{s,t}(x,\xi)\begin{cases}
						\geq 0,&\quad \text{if }\;\xi\geq 0,\\
						\leq 0,&\quad \text{if }\;\xi\leq 0.
					\end{cases}$$
Indeed, this can be proved by contradiction: let $\xi>0$ and assume that for some $t\in[s,T]$ and $x\in\R^N$, $\varphi^0_{s,t}(x,\xi)<0$. Since $t\mapsto\varphi_{s,t}^0(x,\xi)$ is continuous and $\varphi_{s,s}^0(x,\xi)>0$ there exists $r\in[s,t]$ such that $\varphi^0_{s,r}(x,\xi)=0$. Now, we may take $\varphi_{s,r}(x,\xi)=(0,\varphi^x_{s,r}(x,\xi))$ as the initial condition for \eqref{eq:char} at time $r$ to obtain, on the one hand,
$$\varphi^0_{r,t}(\varphi_{s,r}(x,\xi))=\varphi^0_{s,t}(x,\xi)$$
and, on the other hand,
$$\varphi^0_{r,t}(\varphi_{s,r}(x,\xi))=\varphi^0_{r,t}(0,\varphi^x_{s,r}(x,\xi))=0.$$
which is a contradiction. As a consequence, $\mathcal{S}(t,s)\ind_{0>\xi}=\ind_{0>\xi}$ and the claim follows.
%
\end{proof}
\end{lemma}

\subsection{Solution to the BGK model}
\label{sec:bgksol}

Throughout this section we continue with our rough-pathwise analysis, that is, we consider one fixed realization of the driving path $\mathbf{\Lambda}(\omega)$ and therefore the underlying probability space $(\Omega,\mf,\prst)$ remains hidden.

We will apply the auxiliary results for the rough transport equation \eqref{eq:aux} and establish existence and uniqueness for the BGK model \eqref{eq:bgk}.
First of all, it is necessary to specify in which sense \eqref{eq:bgk} is to be solved. As it can be seen in Definition \ref{def:strongsol}, the solution given by Proposition \ref{prop:aux} satisfies the equation \eqref{eq:aux} only on a formal level. This obstacle was overcome by the definition of weak solution (see Definition \ref{def:weaksol}) which also permitted generalization to less regular initial data, namely $X_0\in L^\infty(\mr^N\times\mr)$. We continue in this fashion and define the notion of weak solution to the BGK model similarly.

Concerning the initial data for the BGK model \eqref{eq:bgk}, let us simply set $F^\varepsilon_0=\ind_{u_0>\xi}$ which is sufficient for our purposes. Nevertheless, our proof of existence and convergence of \eqref{eq:bgk} would remain valid if we considered $F_0^\varepsilon=\ind_{u_0^\varepsilon>\xi}$ where $(u_0^\varepsilon)$ is a suitable approximation of $u_0$.

\begin{defin}\label{def:bgk}
Let $\varepsilon>0$. Then $F^\varepsilon\in L^\infty([0,T]\times\mr^N\times\mr)$ satisfying $F^\varepsilon-\ind_{0>\xi}\in L^1([0,T]\times\mr^N\times\mr)$ is called a weak solution to the BGK model \eqref{eq:bgk} with initial condition $F_0^\varepsilon$ provided
\begin{equation}\label{eq:bgk2}
\begin{split}
\dif F^\varepsilon(t,\varphi_{0,t})&=\frac{\ind_{u^\varepsilon(t)>\xi}\circ\varphi_{0,t}-F^\varepsilon(t,\varphi_{0,t})}{\varepsilon}\dif t\\
F^\varepsilon(0)&=F_0^\varepsilon
\end{split}
\end{equation}
holds true in the sense of $\mathcal{D}'(\mr^N\times\mr)$\footnote{By $\ind_{u^\varepsilon(t)>\xi}\circ\varphi_{0,t}$ we denote the composition of $(x,\xi)\mapsto\ind_{u^\varepsilon(t,x)>\xi}$ with $(x,\xi)\mapsto\varphi_{0,t}(x,\xi)$.}, i.e. for all $\phi\in C^1_c(\R^N\times\R)$ and $t\in[0,T]$
$$\big\langle F^\varepsilon(t,\varphi_{0,t}),\phi\big\rangle=\langle F_0^\varepsilon,\phi\rangle+\frac{1}{\varepsilon}\int_0^t\big\langle \ind_{u^\varepsilon(s)>\xi}\circ\varphi_{0,s}-F^\varepsilon(s,\varphi_{0,s}),\phi\big\rangle\dd s.$$
\end{defin}

The result reads as follows.

\begin{thm}\label{duhamel}
For any $\varepsilon>0$, there exists a unique weak solution of the BGK model \eqref{eq:bgk} and it is represented by
\begin{equation}\label{eq:sol}
F^\varepsilon(t)=\me^{-\frac{t}{\varepsilon}}\mathcal{S}(t,0)F_0^\varepsilon+\frac{1}{\varepsilon}\int_0^t\me^{-\frac{t-s}{\varepsilon}}\mathcal{S}(t,s)\ind_{u^\varepsilon(s)>\xi}\,\dif s.
\end{equation}

\begin{proof}
By Duhamel's principle, the problem \eqref{eq:bgk2} admits an equivalent integral representation
\begin{equation*}
F^\varepsilon(t,\varphi_{0,t})=\me^{-\frac{t}{\varepsilon}}F_0^\varepsilon+\frac{1}{\varepsilon}\int_0^t\me^{-\frac{t-s}{\varepsilon}}\ind_{u^\varepsilon(s)>\xi}(\varphi_{0,s})\,\dif s
\end{equation*}
which can be rewritten as \eqref{eq:sol}. Recall, that the local densities are defined as follows
\begin{equation}\label{dens}
u^\varepsilon(t,x)=\int_\mr \big(F^\varepsilon(t,x,\xi)-\ind_{0>\xi}\big)\,\dif \xi
\end{equation}
hence the function $F^\varepsilon$ is not integrable with respect to $\xi$. For the purpose of the proof it is therefore more convenient to consider rather $f^\varepsilon(t)=F^\varepsilon(t)-\ind_{0>\xi}$ which is integrable and prove its existence. Moreover, we will show $f^\varepsilon$ also admits an integral representation, similar to \eqref{eq:sol}.
Indeed, due to Lemma \ref{indik} and Corollary \ref{cor:weaksol}, $\ind_{0>\xi}=\mathcal{S}(t,s)\ind_{0>\xi}$ is the unique weak solution to \eqref{eq:aux} hence $f^\varepsilon$ solves
\begin{equation}\label{bgk4}
\begin{split}
\dif f^\varepsilon(t,\varphi_{0,t})&=\frac{\chi_{u^\varepsilon(t)}\circ\varphi_{0,t}-f^\varepsilon(t,\varphi_{0,t})}{\varepsilon}\,\dif t\\
f^\varepsilon(0)&=\chi_{u_0}
\end{split}
\end{equation}
in the sense of distributions. By a similar reasoning as above, \eqref{bgk4} has the integral representation
\begin{equation}\label{eq:hsol}
f^\varepsilon(t)=\me^{-\frac{t}{\varepsilon}}\mathcal{S}(t,0)\chi_{u_0}+\frac{1}{\varepsilon}\int_0^t \me^{-\frac{t-s}{\varepsilon}}\mathcal{S}(t,s)\chi_{u^\varepsilon(s)}\,\dif s
\end{equation}
and thus can be solved by a fixed point method.
According to the identity
\begin{equation*}
\int_\mr|\chi_{\alpha}-\chi_{\beta}|\,\dif \xi=|\alpha-\beta|,\qquad \alpha,\,\beta\in\mr,
\end{equation*}
some space of $\xi$-integrable functions would be well suited to deal with the nonlinearity term $\chi_{u^\varepsilon}.$ Let us denote $\mathscr{H}=L^\infty(0,T;L^1(\mr^N\times\mr))$ and show that the mapping
\begin{equation*}
\begin{split}
\big(\mathscr{K}g\big)&(t)=\me^{-\frac{t}{\varepsilon}}\mathcal{S}(t,0)\chi_{u_0}+\frac{1}{\varepsilon}\int_0^t\me^{-\frac{t-s}{\varepsilon}}\mathcal{S}(t,s)\chi_{v(s)}\,\dif s,
\end{split}
\end{equation*}
where the local density $v(s)=\int_\mr g(s,\xi)\,\dif \xi$ is defined consistently with \eqref{dens}, is a contraction on $\mathscr{H}$.
Let $g,\,g_1,\,g_2\in \mathscr{H}$ with corresponding local densities $v,\,v_1,$ $v_2$. By  Proposition \ref{oper} and the assumptions on initial data, we arrive at
\begin{equation*}
\begin{split}
\big\|(\mathscr{K}g)(t)\big\|_{L^1_{x,\xi}}&\leq \me^{-\frac{t}{\varepsilon}}\big\|\mathcal{S}(t,0)\chi_{u_0}\big\|_{L^1_{x,\xi}}+\frac{1}{\varepsilon}\int_0^t\me^{-\frac{t-s}{\varepsilon}}\big\|\mathcal{S}(t,s)\chi_{v(s)}\big\|_{L^1_{x,\xi}}\dif s\\
&\leq C\Big(\|u_0\|_{L^1_x}+\sup_{0\leq s\leq t}\|v(s)\|_{L^1_x}\Big)
\end{split}
\end{equation*}
with a constant independent on $t$,
hence
\begin{equation*}
\begin{split}
\big\|\mathscr{K}g\big\|_{\mathscr{H}}&\leq C\big(\|u_0\|_{L^1_x}+\|g\|_{\mathscr{H}}\big)<\infty.
\end{split}
\end{equation*}
Next, we estimate
\begin{equation*}
\begin{split}
\big\|(\mathscr{K} g_1)(t)-(\mathscr{K}g_2)(t)\big\|_{L^1_{x,\xi}}&\leq \frac{1}{\varepsilon}\int_0^t\me^{-\frac{t-s}{\varepsilon}}\big\|\mathcal{S}(t,s)(\chi_{v_1(s)}-\chi_{v_2(s)})\big\|_{L^1_{x,\xi}}\dif s\\
&\leq\frac{C}{\varepsilon}\int_0^t\me^{-\frac{t-s}{\varepsilon}}\big\|\chi_{v_1(s)}-\chi_{v_2(s)}\big\|_{L^1_{x,\xi}}\dif s\\
&\leq \frac{C}{\varepsilon}\int_0^t\me^{-\frac{t-s}{\varepsilon}}\|v_1(s)-v_2(s)\|_{L^1_{x}}\dif s\\
&\leq \frac{C}{\varepsilon}\int_0^t\me^{-\frac{t-s}{\varepsilon}}\|g_1(s)-g_2(s)\|_{L^1_{x,\xi}}\dif s,
\end{split}
\end{equation*}
so
$$\big\|\mathscr{K} g_1-\mathscr{K}g_2\big\|_{\mathscr{H}}\leq C\big (1-\me^{-\frac{T}{\varepsilon}}\big)\|g_1-g_2\|_{\mathscr{H}}$$
and according to the Banach fixed point theorem, the mapping $\mathscr{K}$ has a unique fixed point in $\mathscr{H}$ provided $T$ was small enough. Nevertheless, since the choice of $T$ does not depend on the initial condition, extension of this existence and uniqueness result to
the whole interval $[0, T ]$ can be done by considering the equation on smaller intervals $[0,\tilde T],\,[\tilde T,2\tilde T]$, etc. and repeating the above procedure.
As a consequence, we obtain the existence of a unique weak solution to \eqref{eq:bgk} that is given by \eqref{eq:sol} and the proof is complete.
\end{proof}
\end{thm}

\subsection{Convergence of the BGK model}
\label{sec:conv}

In this final section we establish another weak formulation of the BGK model \eqref{eq:bgk} that actually includes the stochastic integral and is therefore better suited for proving the convergence to \eqref{eq:weakkinformul}. To be more precise, we prove the following result, which will complete the proof of Theorem \ref{thm:main}.

\begin{thm}\label{thm:bgkconvergence}
Let $f^\varepsilon=F^\varepsilon-\ind_{0>\xi}$. Then there exists $u$ which is a kinetic solution to the conservation law \eqref{eq} and, in addition, $(f^\varepsilon)$ converges weak-star in $L^\infty(\Omega\times[0,T]\times\mr^N\times \mr)$ to the equilibrium function $\chi_u$.
\end{thm}

We start with an auxiliary result.

\begin{prop}\label{lemma:7}
$F^\varepsilon\in L^\infty(\Omega\times[0,T]\times\mr^N\times\mr)$ satisfies the following weak formulation of \eqref{eq:bgk}: let $\phi\in C^2_c(\mr^N\times\mr),$ then it holds true a.s.
\begin{align}\label{eq:bgkweak}
\dif \langle F^\varepsilon,\phi(\theta_t)\rangle&=-\langle \partial_\xi F^\varepsilon g,\phi(\theta_t)\rangle\,\dif W+\frac{1}{2}\langle\partial_\xi(G^2\partial_\xi F^\varepsilon),\phi(\theta_t)\rangle\,\dif t+\frac{1}{\varepsilon}\langle\ind_{u^\varepsilon>\xi}-F^\varepsilon,\phi(\theta_t)\rangle\,\dd t.
\end{align}
Moreover, it is progressively measurable.

\begin{proof}
In order to verify \eqref{eq:bgkweak}, we test \eqref{eq:sol} by $\phi(\theta_{0,t})$ and obtain
\begin{align}\label{eq:5a}
\big\langle F^\varepsilon(t),\phi(\theta_{0,t})\big\rangle=\me^{-\frac{t}{\varepsilon}}\big\langle\mathcal{S}(t,0)F_0^\varepsilon,\phi(\theta_{0,t})\big\rangle+\frac{1}{\varepsilon}\int_0^t\me^{-\frac{t-s}{\varepsilon}}\big\langle\mathcal{S}(t,s)\ind_{u^\varepsilon(s)>\xi},\phi(\theta_{0,s}(\theta_{s,t}))\big\rangle\,\dif s.
\end{align}
According to Duhamel's principle, it is enough to verify that if $X_0\in L^\infty(\mr^N\times\mr)$ then
$$\big(X(t):=\mathcal{S}(t,s)X_0,\phi(\theta_{0,s}(\theta_{s,t}))\big)$$
solves
\begin{align}\label{eq:5}
\begin{aligned}
\dif\big\langle X(t),\phi(\theta_{0,s}(\theta_{s,t}))\big\rangle&=-\big\langle \partial_\xi X g,\phi(\theta_{0,s}(\theta_{s,t}))\big\rangle\,\dif W+\frac{1}{2}\big\langle\partial_\xi(G^2\partial_\xi X),\phi(\theta_{0,s}(\theta_{s,t}))\big\rangle\,\dif t,\\
\langle X(s),\phi(\theta_{0,s})\rangle&=\langle X_0,\phi(\theta_{0,s})\rangle.
\end{aligned}
\end{align}
Indeed, let us calculate the stochastic differential of $\langle F^\varepsilon(t),\phi(\theta_{0,t})\rangle$ given by the right hand side of \eqref{eq:5a}. Since according to the It\^o formula applied to a product
\begin{align*}
\me^{-\frac{t-s}{\varepsilon}}\big\langle X(t),\phi(\theta_{0,s}(\theta_{s,t}))\big\rangle&=\big\langle X_0,\phi(\theta_{0,s})\big\rangle-\frac{1}{\varepsilon}\int_s^t\me^{-\frac{r-s}{\varepsilon}}\big\langle X(r),\phi(\theta_{0,s}(\theta_{s,r}))\big\rangle\dd r\\
&\quad+\int_s^t\me^{-\frac{r-s}{\varepsilon}}\dd\big\langle X(r),\phi(\theta_{0,s}(\theta_{s,r}))\big\rangle,
\end{align*}
it follows due to \eqref{eq:5}
\begin{align*}
\langle F^\varepsilon(t),\phi(\theta_{0,t})\rangle&=\langle F^\varepsilon_0,\phi\rangle-\frac{1}{\varepsilon}\int_0^t\me^{-\frac{r}{\varepsilon}}\big\langle \mathcal{S}(r,0)F_0^\varepsilon,\phi(\theta_{0,r})\big\rangle\dd r\\
&\quad-\int_0^t\me^{-\frac{r}{\varepsilon}}\big\langle \partial_\xi [\mathcal{S}(r,0)F_0^\varepsilon] g,\phi(\theta_{0,r})\big\rangle\,\dif W_r+\frac{1}{2}\int_0^t\me^{-\frac{r}{\varepsilon}}\big\langle\partial_\xi(G^2\partial_\xi [\mathcal{S}(r,0)F_0^\varepsilon]),\phi(\theta_{0,r})\big\rangle\,\dif r\\
&\quad+\frac{1}{\varepsilon}\int_0^t\bigg[\langle\ind_{u^\varepsilon(s)>\xi},\phi(\theta_{0,s})\rangle-\frac{1}{\varepsilon}\int_s^t\me^{-\frac{r-s}{\varepsilon}}\big\langle\mathcal{S}({r,s})\ind_{u^\varepsilon(s)>\xi},\phi(\theta_{0,s}(\theta_{s,r}))\big\rangle\dd r\\
&\hspace{2.5cm}-\int_s^t\me^{-\frac{r-s}{\varepsilon}}\big\langle \partial_\xi [\mathcal{S}(r,s)\ind_{u^\varepsilon(s)>\xi}] g,\phi(\theta_{0,s}(\theta_{s,r}))\big\rangle\,\dif W_r\\
&\hspace{2.5cm}+\frac{1}{2}\int_s^t\me^{-\frac{r-s}{\varepsilon}}\big\langle\partial_\xi(G^2\partial_\xi [\mathcal{S}(r,s)\ind_{u^\varepsilon(s)>\xi}]),\phi(\theta_{0,s}(\theta_{s,r}))\big\rangle\,\dif r\bigg]\dd s\\
&=\langle F^\varepsilon_0,\phi\rangle-\int_0^t\big\langle \partial_\xi F^\varepsilon(r) g,\phi(\theta_{0,r})\big\rangle\,\dif W_r+\frac{1}{2}\int_0^t\big\langle\partial_\xi(G^2\partial_\xi F^\varepsilon(r)),\phi(\theta_{0,r})\big\rangle\,\dif r\\
&\quad -\frac{1}{\varepsilon}\int_0^t\langle F^\varepsilon(s),\phi(\theta_{0,s})\rangle\dd s+\frac{1}{\varepsilon}\int_0^t\langle\ind_{u^\varepsilon(s)>\xi},\phi(\theta_{0,s})\rangle\dd s,
\end{align*}
where the last equality is a consequence of deterministic and stochastic Fubini's theorem and \eqref{eq:sol}. Hence \eqref{eq:bgkweak} is satisfied.

Let us now justify \eqref{eq:5}. It was shown in \cite[Theorem 8]{DOR15} that the RDE solution to the characteristic system \eqref{eq:char} corresponding to the joint lift $\mathbf{\Lambda}$ constructed in \eqref{jointlift} can be obtained as limit of SDE solutions to
\begin{align}\label{eq:sdechar}
\begin{aligned}
\dif\varphi^{n,0}_t&=g(\varphi^{n}_t)\circ\dif W-\frac{1}{4}\partial_\xi G^2(\varphi_t^n)\,\dif t-b(\varphi^n_t)\,\dif z^n,\\
\dif\varphi^{n,x}_t&=a(\varphi^n_t)\,\dif z^n,
\end{aligned}
\end{align}
where the corresponding lifts $(\mathbf{z}^n)$ approximate $\mathbf{z}$ as in Definition \ref{def:solutionrde}.
Let $\psi^n$ and $\theta^n$, respectively, denote the inverse flows corresponding to \eqref{eq:sdechar} and
\begin{align}\label{bbb}
\begin{aligned}
\dif\pi^{n,0}_t&=-b(\pi^n_t)\,\dif z^n,\\
\dif\pi^{n,x}_t&=a(\pi^n_t)\,\dif z^n,
\end{aligned}
\end{align}
respectively. Let $X_0\in C^1_b(\mr^N\times\mr)$ and $\phi\in C^2_c(\mr^N\times\mr)$. Then setting $X^n(t)=X_0(\psi^n_{s,t})$ yields a solution to the stochastic transport equation corresponding to the characteristic system \eqref{eq:sdechar}, which starts at time $s$ from $X_0$. Besides, $\phi(\theta^n_{0,s}(\theta^n_{s,t}))$ yields a solution to the transport equation corresponding to the characteristic system \eqref{bbb}, which starts at time $s$ from $\phi(\theta^n_{0,s})$. Therefore, applying the It\^o formula to their product and integrating with respect to $(x,\xi)$, we observe that the integrals driven by $z^n$ cancel due to the fact that $\diver a-\partial_\xi b=0$ (cf. a similar discussion in Subsection \ref{subsec:smoothdrivers}) and we obtain
\begin{align*}
\begin{aligned}
\dif\big\langle X^n(t),\phi(\theta^n_{0,s}(\theta^n_{s,t}))\big\rangle&=-\big\langle \partial_\xi X^n g,\phi(\theta^n_{0,s}(\theta^n_{s,t}))\big\rangle\,\dif W+\frac{1}{2}\big\langle\partial_\xi(G^2\partial_\xi X^n),\phi(\theta^n_{0,s}(\theta^n_{s,t}))\big\rangle\,\dif t,\\
\langle X(s),\phi(\theta_{0,s})\rangle&=\langle X_0,\phi(\theta^n_{0,s})\rangle.
\end{aligned}
\end{align*}
As mentioned above, due to \cite[Theorem 8]{DOR15}, $\psi^n_{s,t}(x,\xi)\rightarrow\psi_{s,t}(x,\xi)$ uniformly in $t\in[s,T]$ in probability, where $\psi$ is the RDE inverse flow to \eqref{eq:char}. Due to  invariance of the $\mathrm{Lip}^\gamma$-norm under translation, we deduce (for a subsequence) that $\psi^n\rightarrow\psi$ uniformly in $t\in[0,T]$ and $(x,\xi)\in\mr^N\times\mr$ a.s. 
Indeed, since $\varphi_{s,t}^n(x,\xi)$ is the SDE solution to \eqref{eq:sdechar} starting from $(x,\xi)$ at time $s$, $\tilde\varphi_{s,t}^n:=\varphi_{s,t}^n(x,\xi)-(x,\xi)$ solves
\begin{align*}
\begin{aligned}
\dif\tilde\varphi^{n,0}_t&=g\big((x,\xi)+\tilde\varphi^{n}_t\big)\circ\dif W-\frac{1}{4}\partial_\xi G^2\big((x,\xi)+\tilde\varphi_t^n\big)\,\dif t-b\big((x,\xi)+\tilde\varphi^n_t\big)\,\dif z^n,\\
\dif\tilde\varphi^{n,x}_t&=a\big((x,\xi)+\tilde\varphi^n_t\big)\,\dif z^n,\\
\tilde\varphi^n(s)&=(0,0),
\end{aligned}
\end{align*}
and similarly $\tilde\varphi_{s,t}:=\varphi_{s,t}(x,\xi)-(x,\xi)$ solves
\begin{equation*}
\begin{split}
\dif\tilde\varphi^0_t&= -b\big((x,\xi)+\tilde\varphi_t\big)\,\dif \mathbf{z}+g\big((x,\xi)+\tilde\varphi_t\big)\,\dif \mathbf{w}-\frac{1}{4}\partial_\xi G^2\big((x,\xi)+\tilde\varphi_t\big)\,\dif t,\\
\dif\tilde\varphi^x_t&=a\big((x,\xi)+\tilde\varphi_t\big)\,\dif \mathbf{z},\\
\tilde\varphi(s)&=(0,0).
\end{split}
\end{equation*}
Since for a family of vector fields $V=(V_1,\dots,V_d)$ on $\R^e$ it holds true that
$$\|V(y+\cdot)\|_{\text{Lip}^\gamma}=\|V(\cdot)\|_{\text{Lip}^\gamma},$$
\cite[Theorem 8]{DOR15} yields the convergence $\tilde\varphi^n\to\tilde\varphi$ uniformly in $t\in[s,T]$ in probability which is independent of $(x,\xi)$. Therefore $\varphi^n\to\varphi$ uniformly in $t\in[s,T]$ and $(x,\xi)\in\R^N\times\R$ in probability and as a consequence, along a subsequence, $\varphi^n\to\varphi$ uniformly in $t\in[s,T]$ and $(x,\xi)\in\R^N\times\R$ a.s. Finally, according to Theorem \ref{thm:existence}
\begin{align*}
\sup_{t,x,\xi}|\psi^n_{s,t}(x,\xi)-\psi_{s,t}(x,\xi)|&=\sup_{t,x,\xi}|\psi^n_{s,t}(\varphi^n_{s,t}(x,\xi))-\psi_{s,t}(\varphi^n_{s,t}(x,\xi))|\\
&=\sup_{t,x,\xi}|(x,\xi)-\psi_{s,t}(\varphi(x,\xi)+\varphi^n_{s,t}(x,\xi)-\varphi(x,\xi))|\\
&=\sup_{t,x,\xi}|\psi_{s,t}(\varphi_{s,t}(x,\xi))-\psi_{s,t}(\varphi(x,\xi)+\varphi^n_{s,t}(x,\xi)-\varphi(x,\xi))|\\
&\leq \sup_{t,x,\xi}|\totdif \psi_{s,t}|\sup_{t,x,\xi}|\varphi^n_{s,t}(x,\xi)-\varphi_{s,t}(x,\xi)|\\
&\leq C\sup_{t,x,\xi}|\varphi^n_{s,t}(x,\xi)-\varphi_{s,t}(x,\xi)|\rightarrow 0\quad\text{a.s.}
\end{align*}
and the claim follows.

Moreover, $\theta^n\rightarrow\theta$ uniformly in $s,\,t\in[0,T]$, $(x,\xi)\in\mr^N\times\mr$ (see \cite[Theorem 4]{caruana}) and the same holds true for their first and second order derivatives with respect to $\xi$. Therefore, we may pass to the limit, apply dominated convergence theorem (for both Lebesgue and stochastic integral, see \cite[Theorem 32]{protter} for the latter one) and \eqref{eq:5} follows. If $X_0\in L^\infty(\mr^N\times\mr)$ then we consider its smooth approximation $X^\delta_0$, apply the previous result and pass to the limit.

The progressive measurability follows immediately from the construction.
\end{proof}
\end{prop}

As the next step we prove a stochastic version of Proposition \ref{oper}(i) for $p=1$.

\begin{lemma}\label{prop:update}
The family $\mathcal{S}=\{\mathcal{S}(t,s),\,0\leq s\leq t\leq T\}$ consists of bounded linear operators on $L^1(\Omega\times\mr^N\times\mr)$. In particular, if $X_0\in L^1(\Omega\times\mr^N\times\mr)$ then 
\begin{equation}\label{eq:6}
\sup_{0\leq s\leq T}\E\sup_{s\leq t\leq T}\|\mathcal{S}(t,s)X_0\|_{L^1_{x,\xi}}\leq C \,\stred\|X_0\|_{L^1_{x,\xi}}.
\end{equation}

\begin{proof}
Assume in addition that $X_0$ is nonnegative, bounded and compactly supported in $(x,\xi)$. In general, \eqref{eq:5} holds true for all $\phi\in C^1_c(\mr^N\times\mr)$, nevertheless, since $X\in L^1(\mr^N\times\mr)$ a.s., which follows from Proposition \ref{oper}, the assumption on the test function $\phi$ can be relaxed and we may take $\phi\equiv 1$. Taking expectation, the stochastic integral vanishes due to the additional assumption and we obtain
\begin{equation}\label{eq:7}
\E\|\mathcal{S}(t,s)X_0\|_{L^1_{x,\xi}}\leq \stred\|X_0\|_{L^1_{x,\xi}}.
\end{equation}
Besides, the nonnegativity assumption can be immediately omitted by splitting $X_0$ into negative and positive part.
In the general case of $X_0\in L^1(\Omega\times\mr^N\times\mr)$, we approximate $X_0$ in $L^1(\Omega\times\mr^N\times\mr)$ by $X_0^\delta$ bounded and compactly supported such that 
$$\E\|X_0^\delta\|_{L^1_{x,\xi}}\leq \stred\|X_0\|_{L^1_{x,\xi}}.$$
Apply \eqref{eq:7} to $X_0^\delta$. Due to linearity of $\mathcal{S}(t,s)$ it implies that $\mathcal{S}(t,s)X_0^\delta$ is Cauchy in $L^1(\Omega\times\mr^N\times\mr)$, the limit is necessarily $\mathcal{S}(t,s)X_0$ so Fatou's lemma yields \eqref{eq:7} and the first claim follows.

To obtain, \eqref{eq:6}, we test \eqref{eq:5} again by $\phi\equiv 1$, take supremum and expectation. Applying Burkholder-Davis-Gundy's and weighted Young's inequalities and \eqref{eq:7} we obtain
\begin{align*}
\E\sup_{s\leq t\leq T}\|\mathcal{S}(t,s)X_0\|_{L^1_{x,\xi}}&\leq \E\|X_0\|_{L^1_{x,\xi}}+\E\sup_{s\leq t\leq T}\bigg|\int_s^t\langle X,\partial_\xi g\rangle\,\dif W\bigg|\\
&\leq\E\|X_0\|_{L^1_{x,\xi}}+C\E\bigg(\int_s^T\|X\|_{L^1_{x,\xi}}^2\,\dif t\bigg)^{1/2}\\
&\leq \E\|X_0\|_{L^1_{x,\xi}}+C\E\bigg(\sup_{s\leq t\leq T}\|X\|_{L^1_{x,\xi}}\bigg)^{1/2}\bigg(\int_s^T\|X\|_{L^1_{x,\xi}}\,\dif t\bigg)^{1/2}\\
&\leq \E\|X_0\|_{L^1_{x,\xi}}+\frac{1}{2}\E\sup_{s\leq t\leq T}\|X\|_{L^1_{x,\xi}}+C\,\E\int_s^T\|X\|_{L^1_{x,\xi}}\,\dif t\\
&\leq C\,\E\|X_0\|_{L^1_{x,\xi}}+\frac{1}{2}\E\sup_{s\leq t\leq T}\|\mathcal{S}(t,s)X_0\|_{L^1_{x,\xi}}
\end{align*}
and the proof is complete.
\end{proof}
\end{lemma}

\begin{cor}\label{prop:update1}
The family $\mathcal{S}=\{\mathcal{S}(t,s),\,0\leq s\leq t\leq T\}$ consists of bounded linear operators on $L^2(\Omega;L^1(\mr^N\times\mr))$. In particular, if $X_0\in L^2(\Omega;L^1(\mr^N\times\mr))$ then 
\begin{equation}\label{eq:61}
\sup_{0\leq s\leq T}\E\sup_{s\leq t\leq T}\|\mathcal{S}(t,s)X_0\|^2_{L^1_{x,\xi}}\leq C \,\stred\|X_0\|^2_{L^1_{x,\xi}}.
\end{equation}

\begin{proof}
The proof follows the lines of Lemma \ref{prop:update}. The key observation is that if $X_0$ is nonnegative, bounded and compactly supported in $(x,\xi)$ then, on the one hand,
\begin{equation*}
\begin{split}
\E\|\mathcal{S}(t,s)X_0\|^2_{L^1_{x,\xi}}&\leq C\,\stred\|X_0\|^2_{L^1_{x,\xi}}+C\, \E\bigg|\int_s^t \langle X,\partial_\xi g\rangle\,\dif W\bigg|^2\\
&\leq C\,\stred\|X_0\|^2_{L^1_{x,\xi}}+C\, \E\int_s^t\|X(r)\|_{L^1_{x,\xi}}^2\,\dif r\\
\end{split}
\end{equation*}
and the Gronwall lemma implies
\begin{equation}\label{eq:71}
\begin{split}
\E\|\mathcal{S}(t,s)X_0\|^2_{L^1_{x,\xi}}&\leq C\,\stred\|X_0\|^2_{L^1_{x,\xi}}.
\end{split}
\end{equation}
On the other hand, Burkholder-Davis-Gundy's inequality and \eqref{eq:71} implies
\begin{align*}
\E\sup_{s\leq t\leq T}\|\mathcal{S}(t,s)X_0\|^2_{L^1_{x,\xi}}&\leq C\,\E\|X_0\|^2_{L^1_{x,\xi}}+C\,\E\sup_{s\leq t\leq T}\bigg|\int_s^t\langle X,\partial_\xi g\rangle\,\dif W\bigg|^2\\
&\leq C\,\E\|X_0\|^2_{L^1_{x,\xi}}+C\,\E\int_s^T\|X\|_{L^1_{x,\xi}}^2\,\dif t\\
&\leq C\,\E\|X_0\|^2_{L^1_{x,\xi}}.
\end{align*}
\end{proof}
\end{cor}

\begin{proof}[Proof of Theorem \ref{thm:bgkconvergence}]
In view of Proposition \ref{lemma:7} we remark that $F^\varepsilon$ also satisfies the following weak formulation of \eqref{eq:bgk2}: let $\phi\in C_c^2(\mr^N\times\mr)$ and $\alpha\in C^1_c([0,T))$ then
\begin{equation}\label{formul}
\begin{split}
\int_0^T&\big\langle F^\varepsilon(t),\phi(\theta_t)\big\rangle\partial_t\alpha(t)\,\dif t+\big\langle F_0^\varepsilon,\phi\big\rangle\alpha(0)\\
&=\int_0^T\langle\partial_\xi F^\varepsilon(t) g,\phi(\theta_t)\rangle\alpha(t)\,\dif W(t)-\frac{1}{2}\int_0^T\langle\partial_\xi(G^2\partial_\xi F^\varepsilon(t)),\phi(\theta_t)\rangle\alpha(t)\,\dif t\\
&\hspace{1cm}-\frac{1}{\varepsilon}\int_0^T\big\langle\ind_{u^\varepsilon(t)>\xi}-F^\varepsilon(t),\phi(\theta_t)\big \rangle\alpha(t)\,\dif t
\end{split}
\end{equation}
and a similar expression holds true for $f^\varepsilon$, namely, it satisfies the weak formulation of \eqref{bgk4}.
Taking the limit on the left hand side of \eqref{formul} is quite straightforward. Indeed, according to the representation formula \eqref{eq:sol} it holds that the set $(F^\varepsilon)$ is bounded in $L^\infty(\Omega\times[0,T]\times\mr^N\times\mr)$, more precisely, $F^\varepsilon\in[0,1],\,\varepsilon\in(0,1).$ Therefore, by the Banach-Alaoglu theorem, there exists $F\in L^\infty(\Omega\times[0,T]\times\mr^N\times\mr)$ such that, up to subsequences, 
\begin{equation}\label{weakstar}
F^\varepsilon\overset{w^*}{\longrightarrow}F\quad\text{in}\quad L^\infty(\Omega\times[0,T]\times\mr^N\times\mr).
\end{equation}
As a consequence,
$$\int_0^T\big\langle F^\varepsilon(t),\phi(\theta_t)\big\rangle\partial_t\alpha(t)\,\dif t\overset{w^*}{\longrightarrow} \int_0^T\big\langle F(t),\phi(\theta_t)\big\rangle\partial_t\alpha(t)\,\dif t\quad\text{in}\quad L^\infty(\Omega),$$
$$\frac{1}{2}\int_0^T\langle\partial_\xi( G^2\partial_\xi F^\varepsilon(t)), \phi(\theta_t)\rangle\alpha(t)\,\dif t\overset{w^*}{\longrightarrow}\frac{1}{2}\int_0^T\langle\partial_\xi( G^2\partial_\xi F(t)), \phi(\theta_t)\big\rangle\alpha(t)\,\dif t\quad\text{in}\quad L^\infty(\Omega),$$
According to the hypotheses on the initial data,
$$\big\langle F_0^\varepsilon,\phi\big\rangle\alpha(0)=\big\langle \ind_{u_0>\xi},\phi\big\rangle\alpha(0).$$
We intend to prove a similar convergence result for the stochastic term as well. Since
$$\big\langle F^\varepsilon(t),\partial_\xi\big(g\phi(\theta_t)\big)\big\rangle\alpha(t)\overset{w}{\longrightarrow}\big\langle F(t),\partial_\xi\big(g\phi(\theta_t)\big)\big\rangle\alpha(t)\quad \text{in} \quad L^2(\Omega\times [0,T])$$
and the stochastic integral $\varPhi\mapsto\int_0^T\varPhi\,\dd W$ regarded as bounded linear operator from $L^2(\Omega\times[0,T])$ to $L^2(\Omega)$ is weakly continuous it follows
$$\int_0^T\langle\partial_\xi F^\varepsilon(t) g,\phi(\theta_t)\rangle\alpha(t)\,\dif W(t)\overset{w}{\longrightarrow}\int_0^T\langle\partial_\xi F(t) g,\phi(\theta_t)\rangle\alpha(t)\,\dif W(t)\quad\text{in}\quad L^2(\Omega).$$

Furthermore, since we established convergence of all the terms in \eqref{formul} except for the third one on the right hand side, multiplying \eqref{formul} by $\varepsilon$ gives the convergence of this remaining term to $0$, that is,
\begin{equation}\label{klad}
\int_0^T\big\langle\ind_{u^\varepsilon(t)>\xi}-F^\varepsilon(t),\phi(\theta_t)\big\rangle\alpha(t)\,\dif t\overset{}{\longrightarrow} 0\quad\text{in}\quad L^2(\Omega).
\end{equation}
As the next step, we show that the same convergence holds true if we replace the test function $\phi(\theta)\alpha$ by a general function $\beta\in L^1([0,T]\times\R^N\times\R)$. To this end, recall that linear combinations of tensor functions of the form $\phi\alpha$, where $\phi\in C^2_c(\R^N\times\R)$, $\alpha\in C^1_c([0,T))$, are dense in $L^1([0,T]\times\R^N\times\R)$. Hence if $\beta\in L^1([0,T]\times\R^N\times\R)$ then there exists $\sum_{i=1}^n\phi_i\alpha_i$ with $\phi_i\in C^2_c(\mr^N\times\R)$ and $\alpha_i\in C^1_c([0,T))$, $i=1,\dots,n$, such that
$$\bigg\|\beta-\sum_{i=1}^n\phi_i\alpha_i\bigg\|_{L^1_{t,x,\xi}}<\delta$$
so, due to the fact that $\ind_{u^\varepsilon>\xi}-F^\varepsilon\in [-1,1]$, we obtain
\begin{align*}
\bigg|\int_0^T\big\langle\ind_{u^\varepsilon(t)>\xi}-F^\varepsilon(t),\beta(t,\theta_t)\big\rangle\,\dif t\bigg|&\leq \delta+\sum_{i=1}^n\bigg|\int_0^T\big\langle\ind_{u^\varepsilon(t)>\xi}-F^\varepsilon(t),\phi_i(\theta_t)\big\rangle\alpha_i(t)\,\dif t\bigg|
\end{align*}
thus
$$\int_0^T\big\langle\ind_{u^\varepsilon(t)>\xi}-F^\varepsilon(t),\beta(t,\theta_t)\big\rangle\,\dif t\overset{}{\longrightarrow} 0\quad\text{in}\quad L^2(\Omega).$$
Consequently, we may take $\beta(t,\pi_t)$ instead of $\beta$ to finally deduce
$$\ind_{u^\varepsilon>\xi}- F^\varepsilon\overset{w^*}{\longrightarrow} 0\quad\text{in}\quad L^\infty(\Omega\times[0,T]\times\mr^N\times\mr)$$
and, in particular, for all $\phi\in C^1_c(\R)$,\footnote{Here $\langle\cdot,\cdot\rangle_\xi$ denotes the duality between the space of distributions on $\R$ and $C^1_c(\R)$.}
\begin{equation}\label{lll}
\langle\partial_\xi\ind_{u^\varepsilon>\xi}-\partial_\xi F^\varepsilon,\phi\rangle_\xi\overset{w^*}{\longrightarrow} 0\quad\text{in}\quad L^\infty(\Omega\times[0,T]\times\mr^N).
\end{equation}
In order to obtain the convergence in the remaining term of \eqref{formul} and in view of the kinetic formulation \eqref{eq:weakkinformul}, we need to show that the term $\frac{1}{\varepsilon}(\ind_{u^\varepsilon>\xi}-F^\varepsilon)$ can be written as $\partial_\xi m^\varepsilon$ where $m^\varepsilon$ is a random nonnegative measure over $[0,T]\times\mr^N\times\mr$ bounded uniformly in $\varepsilon$. If we define
\begin{equation}\label{meas}
\begin{split}
m^\varepsilon(\xi)&=\frac{1}{\varepsilon}\int_{-\infty}^\xi \big(\ind_{u^\varepsilon>\zeta}-F^\varepsilon(\zeta)\big)\,\dif\zeta=\frac{1}{\varepsilon}\int_{-\infty}^\xi \big(\chi_{u^\varepsilon}(\zeta)-f^\varepsilon(\zeta)\big)\,\dif\zeta,
\end{split}
\end{equation}
it is easy to check that $m^\varepsilon\geq 0$ a.s. since $F^\varepsilon\in[0,1]$. Indeed, $m^\varepsilon(-\infty)=m^\varepsilon(\infty)=0$ and $m^\varepsilon(t,x,\cdot)$ is increasing if $\xi\in(-\infty,u^\varepsilon(t,x))$ and decreasing if $\xi\in(u^\varepsilon(t,x),\infty)$.

%
%
%

Let us proceed with a uniform estimate for $(u^\varepsilon)$ and $(m^\varepsilon)$.

\begin{prop}\label{densities}
The set of local densities $(u^\varepsilon)$ satisfies the uniform estimate
\begin{equation*}
\E\sup_{0\leq t\leq T}\|u^\varepsilon(t)\|^2_{L^1_x}\leq C\,\E\|u_0\|^2_{L^1_x}.
\end{equation*}

\begin{proof}
It follows from the definition of $u^\varepsilon$ in \eqref{dens} and from \eqref{eq:hsol} that
\begin{equation*}
\begin{split}
u^\varepsilon(t)&=\me^{-\frac{t}{\varepsilon}}\int_\mr\mathcal{S}(t,0)\chi_{u_0}\,\dif\xi+\frac{1}{\varepsilon}\int_0^t\me^{-\frac{t-s}{\varepsilon}}\int_\mr\mathcal{S}(t,s)\chi_{u^\varepsilon(s)}\,\dif\xi\,\dif s.
\end{split}
\end{equation*}
Let us now define the following auxiliary function
$$H(s)=\left|\int_\mr\mathcal{S}(t,s)\chi_{u^\varepsilon(s)}\,\dif\xi\right|.$$
Then
$$H(t)\leq \me^{-\frac{t}{\varepsilon}} H(0)+(1-\me^{-\frac{t}{\varepsilon}})\max_{0\leq s\leq t} H(s)$$
and we conclude that $H(t)\leq H(0),\,t\in[0,T]$. In order to estimate $H(0)$, we apply Corollary \ref{prop:update1} and obtain
\begin{equation*}
\begin{split}
\E\sup_{0\leq t\leq T}\|u^\varepsilon(t)\|^2_{L^1_x}&\leq \E\sup_{0\leq t\leq T}\big\|\mathcal{S}(t,0)\chi_{u_0}\big\|^2_{L^1_{x,\xi}}
\leq C\,\E\|u_0\|^2_{L^{1}_{x}}.
\end{split}
\end{equation*}
\end{proof}
\end{prop}

\begin{prop}\label{kin}
For all $t^*\in[0,T]$ it holds true that
\begin{equation}\label{kin1}
\stred\bigg|\int_{[0,t^*]\times\mr^N\times\mr}\,\dif m^\varepsilon(t,x,\xi)\bigg|^2+\stred\langle f^\varepsilon(t^*),\xi\rangle^2\leq C\big(\E\|u_0\|_{L^2_x}^4+\E\|u_0\|^2_{L^1_{x}}\big).
\end{equation}

\begin{proof}
Here we follow the ideas of \cite{gess}. Let $\{0=t_0<t_1<\cdots<t_n=t^*\}$ be a partition of $[0,t^*]$ with step size $h$, to be chosen below.
It follows immediately from the formula \eqref{eq:hsol} that for every $t_i\in[0,t^*)$, $t\mapsto f^\varepsilon(t_i+t)$ is a solution to \eqref{bgk4} on $[0,t^*-t_i]$ with the initial condition $f^\varepsilon(t_i)$ and therefore the corresponding version of \eqref{eq:bgkweak} holds true, namely, for all $\phi\in C^1_c(\mr^N\times\mr)$ and all $t\in[t_i,t_{i+1}]$
\begin{align*}
\dif \langle f^\varepsilon(t),\phi(\theta_{t_i,t})\rangle&=-\langle \partial_\xi f^\varepsilon g,\phi(\theta_{t_i,t})\rangle\,\dif W+\frac{1}{2}\langle\partial_\xi(G^2\partial_\xi f^\varepsilon),\phi(\theta_{t_i,t})\rangle\,\dif t+\langle\partial_\xi m^\varepsilon,\phi(\theta_{t_i,t})\rangle.
\end{align*}

Now, we need to test by $\phi(\xi)=\xi$. Since $f^\varepsilon\in L^\infty(0,T;L^1(\mr^N\times\mr))$ a.s., we can test by constants, in particular, we do not need compactly supported test functions. Therefore, let $\phi_R\in C^1(\mr)$ be an approximation of $\phi$ which is bounded, monotone increasing, i.e. $\partial_\xi\phi_R\geq0$, and preserves the sign, i.e. $\xi\sgn\phi_R(\xi)\geq 0$, and which satisfies $|\partial_\xi\phi_R|,|\partial^2_\xi\phi_R|\leq C$ uniformly in $R$. Using the weak formulation for $f^\varepsilon$ above we deduce
\begin{align*}
\int_{t_i}^{t}\langle m^\varepsilon,\partial_\xi\phi_R(\theta^0_{t_i,s})\rangle\,\dif s+\langle f^\varepsilon(t),\phi_R(\theta^0_{t_i,t})\rangle&=\langle f^\varepsilon(t_i),\phi_R\rangle-\int_{t_i}^{t}\langle\partial_\xi f^\varepsilon g,\phi_R(\theta^0_{t_i,s})\rangle\,\dif W_s\\
&\quad+\frac{1}{2}\int_{t_i}^{t}\langle\partial_\xi(G^2\partial_\xi f^\varepsilon),\phi(\theta^0_{t_i,s})\rangle\,\dif s.
\end{align*}
Observe that $0\leq \sgn(\xi) f^\varepsilon(\xi)\leq 1$ as a consequence of \eqref{eq:hsol}. Since also $\sgn(\xi)\theta^0_{t_i,t}(\xi)\geq 0$ due to the assumption \eqref{eq:null}, the second term on the left hand side is nonnegative. Moreover, on a small time interval $\partial_\xi\theta^0_{t_i,t}$ remains close to its initial value, that is, choosing $h$ sufficiently small (which can be justified using Theorem \ref{thm:existence}) we may assume that for every $i$
$$\inf_{t_i\leq t\leq t_{i+1}}\partial_\xi\theta_{t_i,t}^0\geq \frac{1}{2}$$
and consequently also the first term on the left hand side is nonnegative. Thus we deduce
\begin{align*}
\begin{aligned}
&\stred\bigg|\int_{t_i}^{t}\langle m^\varepsilon,(\partial_\xi\phi_R)(\theta^0_{t_i,s})\partial_\xi\theta^0_{t_i,s}\rangle\,\dif s\bigg|^2+\stred\langle f^\varepsilon(t),\phi_R(\theta^0_{t_i,t})\rangle^2\\
&\leq C\,\stred\langle f^\varepsilon(t_i),\phi_R\rangle^2+C\,\E\bigg|\int_{t_i}^{t}\langle\partial_\xi f^\varepsilon g,\phi_R(\theta^0_{t_i,s})\rangle\,\dif W_s\bigg|^2+C\,\E\bigg|\int_{t_i}^{t}\langle\partial_\xi(G^2\partial_\xi f^\varepsilon),\phi_R(\theta^0_{t_i,s})\rangle\,\dif s\bigg|^2.
\end{aligned}
\end{align*}
Moreover, it follows from \eqref{eq:hsol}, Lemma \ref{prop:update} and Proposition \ref{densities} that
\begin{align}\label{eq:9}
\stred\|f^\varepsilon(t)\|^2_{L^1_{x,\xi}}\leq \sup_{0\leq s\leq t\leq T}\E\|\mathcal{S}(t,s)\chi_{u^\varepsilon(s)}\|^2_{L^1_{x,\xi}}\leq C\,\E\|u_0\|^2_{L^1_{x}}
\end{align}
and therefore the third term on the right hand side can be estimated as follows
\begin{align*}
\E\bigg|\int_{t_i}^{t}\langle\partial_\xi(G^2\partial_\xi f^\varepsilon),\phi_R(\theta^0_{t_i,s})\rangle\,\dif s\bigg|^2&\leq C(t-t_i)^2\,\E\|u_0\|^2_{L^1_{x}}.
\end{align*}
where the constant $C$ does not depend on $R$ due to the assumption on the derivatives of $\phi_R$ above. For the stochastic integral we have
\begin{align*}
\E\bigg|\int_{t_i}^{t}\langle\partial_\xi f^\varepsilon g,\phi_R(\theta^0_{t_i,s})\rangle\,\dif W_s\bigg|^2&=\E\int_{t_i}^{t}\langle\partial_\xi f^\varepsilon g,\phi_R(\theta^0_{t_i,s})\rangle^2\,\dif s\\
&\leq C\,\E \int_{t_i}^t\langle f^\varepsilon,\phi_R(\theta^0_{t_i,s})\rangle^2\,\dif s+C\,\E\int_{t_i}^t \|f^\varepsilon\|_{L^1_{x,\xi}}^2\,\dif s\\
&\leq C\,\E \int_{t_i}^t\langle f^\varepsilon,\phi_R(\theta^0_{t_i,s})\rangle^2\,\dif s+C(t-t_i)\,\E\|u_0\|_{L^1_x}^2
\end{align*}
hence the Gronwall lemma yields
\begin{align}\label{eq:8a}
\begin{aligned}
&\stred\bigg|\int_{t_i}^{t}\langle m^\varepsilon,(\partial_\xi\phi_R)(\theta^0_{t_i,s})\partial_\xi\theta^0_{t_i,s}\rangle\,\dif s\bigg|^2+\stred\langle f^\varepsilon(t),\phi_R(\theta^0_{t_i,t})\rangle^2\leq C_h\Big(\stred\langle f^\varepsilon(t_i),\phi_R\rangle^2+\E\|u_0\|_{L^1_x}^2\Big).
\end{aligned}
\end{align}

Therefore, if $i=0$ then we estimate the first term on the right hand side of \eqref{eq:8a} by
\begin{align}\label{eq:100}
\stred\langle f^\varepsilon_0,\phi_R\rangle^2&\leq \E \langle\chi_{u_0},\xi\rangle^2=\frac{1}{4}\,\E\|u_0\|_{L^2_x}^4
\end{align}
and obtain by Fatou's lemma
\begin{align}\label{eq:10}
\begin{aligned}
\stred\bigg|\int_{0}^{t_{1}}&\langle m^\varepsilon,1\rangle\,\dif t\bigg|^2+\stred\langle f^\varepsilon(t_{1}),\theta^0_{0,t_{1}}\rangle^2\leq C_h\Big(\E\|u_0\|_{L^2_x}^4+\E\|u_0\|^2_{L^1_{x}}\Big).
\end{aligned}
\end{align}
In order to get a similar estimate on $[t_1,t_2]$ we go back to \eqref{eq:8a} and assume without loss of generality (using Theorem \ref{thm:existence} again) that $h$ was small enough so that for every $i$
$$\sup_{t_i\leq t\leq t_{i+1}}|\xi-\theta^0_{t_i,t}|\leq 1.$$
Consequently, by \eqref{eq:9} and \eqref{eq:10}
\begin{align*}
\stred\langle f^\varepsilon(t_1),\xi\rangle^2&\leq C\,\stred\langle f^\varepsilon(t_{1}),\theta^0_{0,t_{1}}\rangle^2+C\,\stred\langle f^\varepsilon(t_1),(\xi-\theta^0_{0,t_1})\rangle^2\leq C\big(\E\|u_0\|_{L^2_x}^4+\E\|u_0\|_{L^1_{x}}^2\big).
\end{align*}
Iterating the above technique finitely many times, the claim follows.
\end{proof}
\end{prop}

As a consequence of Proposition \ref{densities}, the assumptions of Lemma \ref{kinetcomp} are satisfied for $\nu^\varepsilon_{t,x}=\delta_{u^\varepsilon(t,x)=\xi}$ and hence there exists a Young measure $\nu_{t,x}$ vanishing at infinity such that $\nu^\varepsilon\rightarrow \nu$ in the sense given by this Lemma. We deduce from \eqref{lll} that $\partial_\xi F=-\nu$ hence $F$ is a kinetic function.

Next, we verify the second estimate from \eqref{integrov}.
Due to the definition of $m^\varepsilon$ in \eqref{meas}, it follows from \eqref{kin1} that
$$\E\bigg|\frac{1}{\varepsilon}\int_0^T\langle\chi_{u^\varepsilon(t)}-f^\varepsilon(t),\xi\rangle\,\dif t\bigg|^2+\E\langle f^\varepsilon(t),\xi\rangle^2\leq C\big(\E\|u_0\|_{L^2_x}^4+\E\|u_0\|^2_{L^1_x}\big)$$
which implies
\begin{equation}\label{eq:101}
\E\bigg|\int_0^T\langle\chi_{u^\varepsilon(t)},\xi\rangle\,\dif t\bigg|^2\leq C\big(\E\|u_0\|_{L^2_x}^4+\E\|u_0\|^2_{L^1_x}\big).
\end{equation}
Rewriting the left hand side by the same argument as in \eqref{eq:100} we deduce
$$\E\bigg|\int_0^T\|u^\varepsilon(t)\|_{L^2_x}^2\dif t\bigg|^2\leq C\big(\E\|u_0\|_{L^2_x}^4+\E\|u_0\|^2_{L^1_x}\big)$$
and as a consequence the estimate \eqref{eqqq} follows.

Finally, in order to show that $F$ is a generalized kinetic solution to \eqref{eq}, we will prove that there exists a kinetic measure $m$ such that, for all $\phi\in C^2_c(\R^N\times\R)$ and $\alpha\in C^1_c([0,T))$,
\begin{equation}\label{nino}
\int_0^T\big\langle\partial_\xi m^\varepsilon,\phi(\theta_t)\big\rangle\alpha(t)\,\dd t\overset{w}{\longrightarrow}\int_0^T\big\langle\partial_\xi m,\phi(\theta_t)\big\rangle\alpha(t)\,\dd t\quad\text{in}\quad L^1(\Omega).
\end{equation}
According to \eqref{kin1} and the fact that $m^\varepsilon\geq0$ a.s. we deduce that each $m^\varepsilon$ is a nonnegative finite measure over $[0,T]\times\R^N\times\R$ a.s. 
Besides, due to the convergence in \eqref{formul}, we obtain that for all $\phi\in C^2_c(\R^N\times\R)$ and all $\alpha\in C^1_c([0,T))$ the left hand side of \eqref{nino} indeed converges weakly in $L^1(\Omega)$ to some limit. Besides, due to Proposition \ref{kin}, the set of measures $(m^\varepsilon)$ is bounded in $L^2_w(\Omega;\mathcal{M}_b([0,T]\times\mr^N\times\mr))$, i.e. the space of weak-star measurable mappings from $\Omega$ to $\mathcal{M}_b([0,T]\times\mr^N\times\mr)$ with finite $L^2(\Omega;\mathcal{M}_b([0,T]\times\mr^N\times\mr))$-norm. Since $L^2_w(\Omega;\mathcal{M}_b([0,T]\times\mr^N\times\mr))$ is the dual of the separable space $L^2(\Omega;C_0([0,T]\times\mr^N\times\mr))$, the Banach-Alaoglu theorem applies and yields existence of $m\in L^2_w(\Omega;\mathcal{M}_b([0,T]\times\mr^N\times\mr))$ such that, up to a subsequence,
\begin{equation}\label{convm1}
m^\varepsilon\overset{w^*}{\longrightarrow}m\quad\text{in}\quad L^2_w(\Omega;\mathcal{M}_b([0,T]\times\mr^N\times\mr)).
\end{equation}
This in turn verifies the convergence (and identification of the limit) in \eqref{nino}. It remains to prove that $m$ is indeed a kinetic measure. Clearly, since all $m^\varepsilon$ are nonnegative, the same remains valid for $m$. The points (i) and (ii) from Definition \ref{def:kinmeasure} follow directly from the construction of $m$ and the uniform estimate \eqref{kin1}. The remaining point Definition \ref{def:kinmeasure}(iii) can be justified as follows. Let $\phi\in C_0(\R^N\times\R)$ and define
$$x^\varepsilon(t):=\int_{[0,t]\times\R^N\times\R}\phi(x,\xi)\dd m^\varepsilon(s,x,\xi).$$
If $\vartheta\in L^\infty(\Omega)$ and $\gamma\in L^\infty(0,T)$ then by Fubini's theorem
$$\E\bigg[\vartheta\int_0^T\gamma(t)x^\varepsilon(t)\dd t\bigg]=\E\bigg[\vartheta\int_{[0,T]\times\R^N\times\R}\phi(x,\xi)\varGamma(s)\dd m^\varepsilon(s,x,\xi)\bigg],$$
where $\varGamma(s)=\int_s^T\gamma(t)\dd t$ is continuous and $\varGamma(T)=0$. Since the right hand side converges to
$$\E\bigg[\vartheta\int_{[0,T]\times\R^N\times\R}\phi(x,\xi)\varGamma(s)\dd m(s,x,\xi)\bigg]$$
due to \eqref{convm1}, we may apply Fubini's theorem again to deduce that
$$x(t):=\int_{[0,t]\times\R^N\times\R}\phi(x,\xi)\dd m(s,x,\xi)$$
is a weak limit in $L^1(\Omega\times[0,T])$ of progressively measurable processes and is therefore also progressively measurable.

Altogether, we have proved that $m$ is a kinetic measure and $F$ is a generalized kinetic solution to \eqref{eq}. Since any generalized kinetic solution is actually a kinetic one, due to Reduction Theorem \ref{thm:reduction}, it follows that $F=\ind_{u>\xi}$ and $\nu=\delta_u$, where $u\in L^4(\Omega;L^2(0,T;L^2(\mr^N)))$ is the unique kinetic solution to \eqref{eq}.

The weak-star convergence of $f^\varepsilon$ to $\chi_u$ follows immediately from \eqref{weakstar} and therefore the proof is complete.
\end{proof}

\section*{Acknowledgment}

The author wishes to thank the anonymous referees for providing many useful suggestions.



\begin{thebibliography}{19}
\bibitem{BVW} C. Bauzet, G. Vallet, P. Wittbold, A degenerate parabolic-hyperbolic Cauchy problem with a stochastic force, Journal of Hyp. Diff. Eq. 12 (3) (2015) 501-533.
\bibitem{bauzet} C. Bauzet, G. Vallet, P. Wittbold, The Cauchy problem for conservation laws with a multiplicative noise, Journal of Hyp. Diff. Eq. 9 (4) (2012) 661-709.
\bibitem{vov1} F. Berthelin, J. Vovelle, A BGK approximation to scalar conservation laws with discontinuous flux, Proc. of the Royal Society of Edinburgh A 140 (5) (2010) 953-972.
\bibitem{car} J. Carrillo, Entropy solutions for nonlinear degenerate problems, Arch. Rational Mech. Anal. 147 (1999) 269-361.
\bibitem{caruana} M. Caruana, P. Friz, Partial differential equations driven by rough paths, Proc. Lond. Math. Soc. (3) 100 (2010), no. 1, 177-215.
\bibitem{cohn} D.L. Cohn, Measure Theory, Birkh\"auser, Boston, Basel, Stuttgart, 1980.
\bibitem{karlsen} C.Q. Chen, Q. Ding, K.H. Karlsen, On nonlinear stochastic balance laws, Arch. Rational Mech. Anal. 204 (2012) 707-743.
\bibitem{chen} G.\,Q. Chen, B. Perthame, Well-posedness for non-isotropic degenerate parabolic-hyperbolic equations, Ann. Inst. H. Poincar\'e Anal. Non Lin\'eaire 20 (4) (2003) 645-668.
\bibitem{crisan} D. Crisan, J. Diehl, P.\,K. Friz, and H. Oberhauser, Robust filtering: correlated noise and multidimensional
observation, Ann. Appl. Probab. 23 (2013), no. 5, 2139-2160.
\bibitem{degen2} A. Debussche, M. Hofmanov\'a, J. Vovelle, Degenerate parabolic stochastic partial differential equations: Quasilinear case, to appear in Ann. Probab., arXiv:1309.5817.
\bibitem{debus2} A. Debussche, J. Vovelle, Invariant measure of scalar first-order conservation laws with stochastic forcing, to appear in Probability Theory and Related Fields, arXiv:1310.3779.
\bibitem{debus} A. Debussche, J. Vovelle, Scalar conservation laws with stochastic forcing, J. Funct. Anal. 259 (2010) 1014-1042.
\bibitem{DV} A. Debussche, J. Vovelle, Scalar conservation laws with stochastic forcing, revised version, \url{http://math.univ-lyon1.fr/~vovelle/DebusscheVovelleRevised.pdf}.
\bibitem{DOR15} J. Diehl, H. Oberhauser, S. Riedel, A L\'evy area between Brownian motion and rough paths with applications to robust nonlinear filtering and rough partial differential equations, Stoch. Pr. and their Appl. 125 (2015) 161-181.
\bibitem{diperna} R.\,J. DiPerna, P.\,L. Lions, Ordinary differential equations, transport theory and Sobolev spaces, Invent. Math. 98 (1989) 511-547.
\bibitem{feng} J. Feng, D. Nualart, Stochastic scalar conservation laws, J. Funct. Anal. 255 (2) (2008) 313-373.
\bibitem{friz2} P.\,K. Friz, B. Gess, Stochastic scalar conservation laws driven by rough paths,  to appear in Ann. Inst. H. Poincar\'e Analyse Non Lin\'eaire, arXiv:1403.6785.
\bibitem{friz} P. Friz, N.\,B. Victoir, Multidimensional Stochastic Processes as Rough Paths, Theory and Applications, Cambridge Studies in Advanced Mathematics, vol. 120, Cambridge University Press, Cambridge, 2010.
\bibitem{GS} B. Gess, P.\,E. Souganidis, Long-time behavior, invariant measures and regularizing effects for stochastic scalar conservation laws, arXiv:1411.3939.
\bibitem{gess} B. Gess, P.\,E. Souganidis, Scalar conservation laws with multiple rough fluxes, to appear in Comm. in Mathematical Sciences, arXiv:1406.2978.
\bibitem{bgk} M. Hofmanov\'a, A Bhatnagar-Gross-Krook approximation to stochastic scalar conservation laws, arXiv:1305.6450, to appear in Ann. Inst. H. Poincar\'e Probab. Statist.
\bibitem{hof} M. Hofmanov\'a, Degenerate parabolic stochastic partial differential equations, Stoch. Pr. Appl. 123 (12) (2013) 4294-4336.
\bibitem{holden} H. Holden, N.\,H. Risebro, Conservation laws with a random source, Appl. Math. Optim. 36 (2) (1997) 229-241.
\bibitem{vov} C. Imbert, J. Vovelle, A kinetic formulation for multidimensional scalar conservation laws with boundary conditions and applications, SIAM J. Math. Anal. 36 (2004), no. 1, 214-232.
\bibitem{jakubow} A. Jakubowski, The almost sure Skorokhod representation for subsequences in nonmetric spaces, Teor. Veroyatnost. i Primenen 42 (1997), no. 1, 209-216; translation in Theory Probab. Appl. 42 (1997), no. 1, 167-174 (1998).
\bibitem{kavian} O. Kavian, Introduction \`a la th\'eorie des points critiques, Springer, 1993.
\bibitem{kim} J.\,U. Kim, On a stochastic scalar conservation law, Indiana Univ. Math. J. 52 (1) (2003) 227-256.
\bibitem{kruzk} S.\,N. Kru\v{z}kov, First order quasilinear equations with several independent variables, Mat. Sb. (N.S.) 81 (123) (1970) 228-255.
\bibitem{lieb} E.\,H. Lieb, M. Loss, Analysis, 2nd ed., AMS 2001.
\bibitem{lps} P.\,L. Lions, B. Perthame, P.\,E. Souganidis, Scalar conservation laws with rough (stochastic) fluxes, Stochastic Partial Differential Equations:
Analysis and Computations 1 (2013), no. 4,
664--686.

\bibitem{lps1} P.\,L. Lions, B. Perthame, P.\,E. Souganidis, Scalar conservation laws with rough (stochastic) fluxes; the spatially dependent case, arXiv:1403.4424.
\bibitem{lpt1} P.L. Lions, B. Perthame, E. Tadmor, Formulation cin\'etique des lois de conservation scalaires multidimensionnelles, C.R. Acad. Sci. Paris (1991) 97-102, S\'erie I.
\bibitem{lions} P.L. Lions, B. Perthame, E. Tadmor, A kinetic formulation of multidimensional scalar conservation laws and related equations, J. Amer. Math. Soc. 7 (1) (1994) 169-191.

\bibitem{lyons1} T.\,J. Lyons, M. Caruana, T. L\'evy, Differential equations driven by rough paths, Lectures
from the 34th Summer School on Probability Theory held in Saint Flour, July 6-24, 2004, Lecture Notes in Mathematics 1908, Springer, Berlin, 2007.
\bibitem{lyons2} T.\,J. Lyons, Z. Qian, System Control and Rough Paths, Oxford Mathematical Monographs, Oxford University Press, 2002.
\bibitem{malek} J. M\'alek, J. Ne\v{c}as, M. Rokyta, M. R\r{u}\v{z}i\v{c}ka, Weak and Measure-valued Solutions to Evolutionary PDEs, Chapman \& Hall, London, Weinheim, New York, 1996.

\bibitem{perth} B. Perthame, Kinetic Formulation of Conservation Laws, Oxford Lecture Ser. Math. Appl., vol. 21, Oxford University Press, Oxford, 2002.
\bibitem{tadmor} B. Perthame, E. Tadmor, A kinetic equation with kinetic entropy functions for scalar conservation laws, Comm. Math. Phys. 136 (3) (1991) 501-517.
\bibitem{protter} P.E. Protter, Stochastic Integration and Differential Equations, Springer, 2004.
\bibitem{revuz} D. Revuz and M. Yor, Continuous martingales and Brownian motion, third ed., Grundlehren der Mathematischen Wissenschaften, vol. 293, Springer-Verlag, Berlin, 1999.
\bibitem{stoica} B. Saussereau, I.\,L. Stoica, Scalar conservation laws with fractional stochastic forcing: Existence, uniqueness and invariant measure, Stoch. Pr. Ap. 122 (2012) 1456-1486.
\bibitem{wittbold} G. Vallet, P. Wittbold, On a stochastic first order hyperbolic equation in a bounded domain, Infin. Dimens. Anal. Quantum Probab. Relat. Top. 12 (4) (2009) 613-651.
\bibitem{young} L.\,C. Young, Lectures on the Calculus of Variations and Optimal Control Theory, Saunders, Philadelphia, 1969.
\end{thebibliography}
\end{document}